\pdfoutput=1
\documentclass[12pt,leqno]{amsart}
\usepackage{amsmath}
\usepackage{amsfonts}
\usepackage{amscd}
\usepackage{amssymb}
\usepackage{amsthm}
\usepackage{mathdots}
\usepackage{mathtools}
\usepackage{bm}
\usepackage[all]{xy}
\usepackage[linktocpage=true]{hyperref}
\usepackage{mathrsfs}
\usepackage{geometry}
\usepackage{pdflscape}
\usepackage{tikz,tikz-cd, color}
\usetikzlibrary{shapes,arrows,positioning}
\usepackage{wasysym, stackengine, makebox, graphicx}

\usepackage{adjustbox}
\usetikzlibrary{matrix}
\usetikzlibrary{decorations.pathmorphing}
\tikzset{
  symbol/.style={
    draw=none,
    every to/.append style={
      edge node={node [sloped, allow upside down, auto=false]{$#1$}}}
  }
}
\usepackage{stmaryrd}
\usepackage[shortlabels]{enumitem}
\usepackage{xcolor}
\usepackage{booktabs}
\usepackage{float}
\usepackage{diagbox}
\usepackage{blindtext}
\tikzset{
  symbol/.style={
    draw=none,
    every to/.append style={
      edge node={node [sloped, allow upside down, auto=false]{$#1$}}}
  }
}
\usetikzlibrary{3d,positioning}
\usepackage{mathrsfs}
\usepackage{multirow}
\usepackage{tabularx}
\mathchardef\mhyphen="2D
\newtheorem{thm}{Theorem}
\newtheorem*{thm*}{Theorem}

\newtheorem{prop}[thm]{Proposition}

\newtheorem{proposition}[thm]{Proposition}

\newtheorem{cl}[thm]{Corollary}
\newtheorem{thm&defn}[thm]{Theorem \& Definition}

\newtheorem{theorem}[thm]{Theorem}
\newtheorem{lemma}[thm]{Lemma}
\newtheorem{lm}[thm]{Lemma}

\newtheoremstyle{named}{}{}{\itshape}{}{\bfseries}{.}{.5em}{#3}
\theoremstyle{named}

\theoremstyle{definition}

\newtheorem{definition}[thm]{Definition}
\newtheorem*{notation*}{Notation}

\newtheorem{remark}[thm]{Remark}
\newtheorem{rmk}[thm]{Remark}

\theoremstyle{remark}

\newtheorem*{rem*}{Remark}
\numberwithin{equation}{section}
\numberwithin{thm}{section}
\usepackage{fullpage}
\newcommand{\widehata}{\widehat{\phantom{a}}}
\newcommand{\qsstar}{\star_{\operatorname{small}}}
\renewcommand{\P}{\mathbb{P}}
\renewcommand{\O}{\mathcal{O}}
\newcommand{\N}{\mathbb{N}}
\newcommand{\Z}{\mathbb{Z}}
\newcommand{\Q}{\mathbb{Q}}
\newcommand{\R}{\mathbb{R}}
\newcommand{\C}{\mathbb{C}}
\newcommand{\vn}{\varnothing}

\makeatletter
\newcommand*{\rom}[1]{\expandafter\@slowromancap\romannumeral #1@}
\makeatother
\DeclarePairedDelimiter{\gen}{\langle}{\rangle}

\newcommand{\quotient}[2]{{\raisebox{.2em}{$#1$}\!\!\left/\!\!\raisebox{-.2em}{$#2$}\right.}}
\let\hom\relax
\DeclareMathOperator{\hom}{Hom}
\newcommand{\6}{\partial}

\newcommand{\fSigma}{\mathbf{\Sigma}}
\newcommand{\bbox}{\operatorname{Box}}
\newcommand{\X}{\mathcal X}

\DeclareMathOperator{\Cone}{Cone}
\usepackage{accents}

\title{Quantum invariance under Non-smooth Toric Flops of the simplest type}

\author{Tsung-Chen Chen}
\address[]{Department of Mathematics, National Taiwan University, Taipei 10617, Taiwan}
\email{d13221003@ntu.edu.tw}

\author{Jia-Hua Chong}
\address[]{Department of Mathematics, National Taiwan University, Taipei 10617, Taiwan}
\email{d12221001@ntu.edu.tw}

\author{Hui-Wen Lin}
\address[]{Department of Mathematics and Taida Institute for Mathematical Sciences (TIMS), National Taiwan University, Taipei 10617, Taiwan}
\email{linhw@math.ntu.edu.tw}

\begin{document}
\begin{abstract}
    For a non-smooth toric flop of the simplest type (see \eqref{tflop}, \eqref{ab}), we provide the {\it quantum product formulas} and construct the {\it canonical quantum correspondence} via {\it generalized} analytic continuation which is achieved through a {\it regularization} map. Moreover, the generating functions of Gromov--Witten invariants with ancestors are invariant for {\it all genera}.
\end{abstract}
\maketitle

\section{Introduction}
\subsection{Motivation and Our contribution}
It is well-known that in Mori's theory the minimal models of a given projective manifold of complex dimension at least 3 are not unique. These birationally equivalent minimal models are also known to be connected by a sequence of geometric surgeries called flops (by Y. Kawamata). It is thus a fundamental question to study the geometry of flops. In the smooth category, flops preserve cohomology groups but not the classical cup product structure. During the period from 2010 to 2016, Y.-P. Lee, H.-W. Lin and C.-L. Wang proved the result for “ordinary flops” that the quantum product structure is indeed preserved in the sense of “analytic continuations” across the boundary of the K\"ahler cones.

\begin{theorem}[\cite{LLW1, LLW2, LLW3, LLW4}] For a general ordinary $\mathbb{P}^r$-flop $X \dashrightarrow X'$, the big quantum cohomology rings of $X$ and $X'$ are invariant under the correspondence given by the graph closure in the sense of analytic continuation over the Novikov variables corresponding to the extremal rays. 
\end{theorem}

By supplementing Reid's theory on the toric minimal model program, we got two facts that
{\it any two $K$-equivalent $\mathbb{Q}$-factorial terminal
toric varieties can be connected to each other by a sequence of
toric flops} and {\it any smooth toric flop is an ordinary flop} (See \cite{MR5082801} A.1+A.2).
Therefore, $K$-equivalent toric manifolds which are connected by smooth toric flops admit canonically
isomorphic integral cohomology groups via the graph closure and thus own Quantum invariance. Next, we have to study non-smooth toric flops. Since Gromov--Witten theory has been extended to the category of orbifolds, the first non-smooth case to be treated is the category of simplicial toric varieties. However, the quantum invariance is already unknown in the literature for the simplest possible case when some stabilizer is of order 2. The current paper provides answers to this case.

First of all, we mention some history related to this work. In \cite{coates2009crepant}, Coates analyzed an example about the toric flop ($\vec{a}=(1, 1, 1), \vec{b}=(2,1))$ and gave the correspondence via classical analytic continuation. In \cite{coates2018crepant}, the authors claimed that quantum cohomology rings are ``abstractly" isomorphic under K-equivalent toric Deligne--Mumford stacks related by a single toric wall-crossing. Under D equivalence, it is well-known that the cohomology support of the Fourier--Mukai kernel leads to equivalence of cohomology (by GRR) but in general ``with degrees mixed". 
 
In this paper, we construct the ``degree-preserving" and ``pairing-preserving" quantum correspondence for the simplest non-trivial example, which we believe is the first such construction in the study of quantum invariants under orbifold flops. The correspondence which preserves the orbifold degree and the orbifold Poincar\'e pairing is the most natural one, since through this correspondence, we can extend the quantum invariance to all genera. The main contribution of this work lies in introducing a regulization map in our search for such a correspondence, thereby generalizing the notion of analytic continuation.

\subsection{Set-up}
In \cite{BCS}, the authors generalized a simplicial toric variety to a toric Deligne--Mumford stack corresponding to a combinatorial object called a stacky fan. Let $\Sigma$ be a simplicial fan in $N\otimes \Q$ with $N$ a finitely generated abelian group of free rank $n$ and $X_{\Sigma}$ be the corresponding simplicial toric variety. Denote by $v_\rho\in N$ the minimal lattice point along $\rho\in\Sigma(1)$, where $\Sigma(1)$ is the set of 1-dimensional faces of $\Sigma$. $X_\Sigma$ is covered by orbifolds of the form $\C^n/N(\sigma)$ with $\sigma$ running through maximal cones of $\Sigma$, where $N(\sigma):=N/(\sum_{\rho\in\sigma(1)}\Z v_\rho)$ acts on $\C^n$ diagonally, so a smooth replacement for $X_\Sigma$ is constructed via some stack $\mathcal X(\mathbf\Sigma)\longrightarrow X_{\Sigma}$ on which some decent theories on smooth varieties are carried over. Explicitly, let $\beta:\Z^{\Sigma(1)}\longrightarrow N$ be a homomorphism that maps the standard basis $\{e_\rho:\rho\in\Sigma(1)\}$ to $b_\rho\in N$ which lies in the direction of $v_\rho$ in $N\otimes\Q$. The triple $\mathbf\Sigma=(N,\Sigma,\beta)$ is called a {\it stacky fan}. 
A stacky fan encodes a group action on a quasi-affine variety and the quotient is called the toric Deligne--Mumford stack $\mathcal X(\mathbf\Sigma)$.

For $\tau\in\Sigma$, we define $\mathcal V(\tau)\subset\mathcal X(\mathbf\Sigma)$ to be a toric substack whose stacky fan is $\mathbf\Sigma/\tau:=(\Sigma/\tau,\beta_\tau,N(\tau))$ where $\beta_\tau$ is the restriction of $\beta$ to $e_{\rho}$ for $\rho\in\Sigma(1)\setminus\tau(1)$ and $\tau+\rho\in\Sigma$. 
One of the distinguished features of the stacky structure is that it naturally gives rise to an additional object called the inertia stack $\mathcal I\mathcal X$. It can be described combinatorially as follows.
For $v\in N$, let $\sigma(\bar v)\in\Sigma$ be the minimal cone containing $\bar v\in N_\Q$. We define $\operatorname{Box}(\mathbf\Sigma)$ to be the set of elements $v\in N$ such that $\bar v=\sum_{\rho\in\sigma(\bar v)(1)}a_\rho \bar v_\rho$ for some $a_\rho\in[0,1)$. In \cite{BCS},
they showed that
if $\Sigma$ is a complete fan, then \[\mathcal I\mathcal X(\mathbf\Sigma)=\coprod_{v\in\operatorname{Box(\mathbf\Sigma)}}\mathcal V(\sigma(\bar v))=\coprod_{v\in\operatorname{Box}(\fSigma)}\mathcal{X}(\fSigma/\sigma(\bar{v})).\]
The orbifold Chow group of $\mathcal X(\mathbf\Sigma)$ can be defined by 
$$ A_{\operatorname{orb}}^*(\mathcal X(\mathbf\Sigma)) = \bigoplus_{v\in\operatorname{Box}(\fSigma)}A^{*-{\rm age}(\mathcal{X}(\fSigma/\sigma(\bar{v})))}(\mathcal{X}(\fSigma/\sigma(\bar{v}))),$$
equipped with the orbifold Poincar\'e pairing
\[(T_1,T_2)_{\mathcal X}=\int_{\mathcal I\mathcal X(\Sigma)}T_1\cup\iota^*T_2\quad\text{for }T_1,T_2\in A^*_{\operatorname{orb}}(\mathcal X(\fSigma)),\]
where $\iota$ is the involution $\mathcal I\mathcal X\to\mathcal I\mathcal X,~(x,g)\mapsto (x,g^{-1})$.

For simplicial toric flops, their local models can be described as follows. We consider the sequences of positive integers $\vec{a}=(a_0,\ldots,a_r)\in \N^{r+1}$ and $\vec{b}=(b_0,\ldots,b_{r'})\in \N^{r'+1}$ such that 
\begin{equation}\label{flop condition}
    \sum_i a_i=\sum_j b_j.
\end{equation}Let $N=\Z^{r+1+r'+1}/\langle(\vec{a},-\vec{b})\rangle$ be the abelian group with generators $v_0,\ldots,v_r,w_0,\ldots,w_{r'}$ and the relation $\sum a_iv_i-\sum b_jw_j=0$. Set $w_{r'+1}=-\sum a_iv_i=-\sum b_jw_j$ and let $\Sigma^+$ be the fan in $N\otimes\R$ with maximal cones
\[\sigma_{ij}=\Cone(v_0,\ldots,\widehat{v}_i,\ldots,v_r,w_0,\ldots,\widehat{w}_j,\ldots,w_{r'+1})\]
where $i\in\{0,1,\ldots,r\}$ and $j\in\{0, 1, \ldots,r'+1\}$.
The morphism $\rho:\Z^{r+r'+3}\to N$ is defined by sending $v_0,\ldots,v_r,w_0,\ldots,w_{r'+1}$ to 
$\bar{v}_0,\ldots,\bar{v}_r$, $\bar{w}_0,\ldots,\bar{w}_{r'+1}$. The morphism $\rho$ and the fan $\Sigma^+$ give rise to a stacky fan $\fSigma^+$, which defines a toric Deligne--Mumford stack
\[\mathbb{P}(\vec{a};\vec{b})\coloneqq\mathcal{X}(\fSigma^+)=\left[\begin{tikzcd}[row sep=1pt]
(\C^\times)^2 \arrow[rr,"\alpha"] && (\C^{r+1}\setminus \{0\})\times(\C^{r'+2}\setminus \{0\}) \\
(s,u) \arrow[rr] && (s^{\vec{a}},u^{\vec{b}}s^{-\vec{b}},u)
\end{tikzcd}\right]\]
where $s^{\vec{a}}$ denotes the tuple $(s^{a_0},\ldots,s^{a_r})$. Indeed, this is the weighted projective stack bundle $\mathbb{P}_{\mathbb{P}(\vec{a})}(\O\oplus\O(-\vec{b}))$ over the weighted projective stack $\mathbb{P}(\vec{a})$ with weights $1$ on $\O$ and $b_i$ on $\O(-b_i)$.
To get a toric flop of $\mathbb{P}(\vec{a};\vec{b})$, we consider the weighted blow-up of $\P(\vec{a};\vec{b})$ along the zero section $\P(\vec{a})$, 
which turns out to be the weighted projective stack bundle
$\mathbb{P}_{\P(\vec{a})\times\P(\vec{b})}(\O\oplus\O(-1,-1))= \mathcal{X}(\widetilde{\fSigma})$ with a stacky fan $\widetilde{\fSigma}$. The blow-up morphism $\varphi^+:\mathbb{P}_{\P(\vec{a})\times\P(\vec{b})}(\O\oplus\O(-1,-1))\to\P(\vec{a};\vec{b})$ restricted on the exceptional locus is the projection $\P(\vec{a})\times\P(\vec{b})\to\P(\vec{a})$, so we may blow-down another direction and get the morphism $\varphi^-:\P_{\P(\vec{a})\times\P(\vec{b})}(\O\oplus\O(-1,-1))\to\P(\vec{b};\vec{a})$. Here, by symmetry of $\vec{a}$ and $\vec{b}$, there is a stacky fan $\fSigma^-$ such that $\mathbb{P}(\vec{b};\vec{a}) = \mathcal{X}(\fSigma^-)$. The simplicial toric flop can be connected by common blow-up in the following diagram:
\begin{equation} \label{tflop}
\begin{tikzcd} 
& \mathcal{Y}\coloneqq\P_{\P(\vec{a})\times\P(\vec{b})}(\O\oplus\O(-1,-1)) \arrow[ld,"\varphi^+"'] \arrow[rd,"\varphi^-"] & \\
\mathcal{X}^+\coloneqq\P(\vec{a};\vec{b}) && \mathcal{X}^-\coloneqq\P(\vec{b};\vec{a}).
\end{tikzcd}
\end{equation}

There is a surjective map
\[\Lambda_{\mathcal{X}^+}\coloneqq\left\{(d_1,d_3)\in\Q_{\geq 0}^2\middle|\ d_1\in\bigcup_i\tfrac{1}{a_i}\Z,\ d_3\in \left(d_1+\bigcup_j\tfrac{1}{b_j}\Z\right)\cup\Z \right\}\longrightarrow\bbox(\fSigma^+)\]
by $\nu^+\colon (d_1,d_3)\longmapsto\sum_{i=0}^r\lceil a_id_1\rceil v_i+\sum_{i=0}^{r'}\lceil b_i(d_3-d_1)\rceil w_i+\lceil d_3\rceil w_{r'+1}$. Let $\mathcal{X}^+_{(d_1,d_3)}$ be the twisted sector corresponding to $\nu^+(d_1,d_3)\in\operatorname{Box}(\mathbf{\Sigma^+})$ and then 
\[A_{\operatorname{orb}}^*(\mathcal{X}^+) = \bigoplus A^{*-{\rm age}(\mathcal{X}^+_{(d_1, d_3)})}(\mathcal{X}^+_{(d_1, d_3)}).\]
A similar result holds for $\mathcal{X}^-$. For $I\mathcal{Y}$, we consider
\[\nu:\Lambda_{\mathcal{Y}}=\left\{(d_1,d_2,d_3)\in\Q_{\geq 0}^3\middle|\ d_1\in\bigcup_i\tfrac{1}{a_i}\Z,\ d_2\in\bigcup_j\tfrac{1}{b_j}\Z,\ d_3\in\Z\cup\left(d_1+d_2 +\Z\right)\right\}\to\operatorname{Box}(\widetilde{\fSigma})\]
by $\nu:(d_1,d_2,d_3)\longmapsto\sum_{i=0}^r\lceil a_id_1\rceil v_i+\sum_{i=0}^{r'}\lceil b_id_2\rceil w_i+\lceil d_3\rceil w_{r'+1}+\lceil d_3-d_2-d_1\rceil(-w_{r'+1})$ and 
$$ A_{\operatorname{orb}}^*(\mathcal{Y}) = \bigoplus A^{*-{\rm age}(\mathcal{Y}_{(d_1, d_2, d_3)})}(\mathcal{Y}_{(d_1, d_2, d_3)}).$$ 
More precisely, the morphisms $\varphi^+$ and $\varphi^-$ induce maps in twisted sectors via projections
\[\begin{tikzcd}[row sep=1pt]
\operatorname{Box}(\widetilde{\fSigma}) \arrow{r} & \operatorname{Box}(\fSigma^+) && \operatorname{Box}(\widetilde{\fSigma}) \arrow{r} & \operatorname{Box}(\fSigma^-) \\
(d_1,d_2,d_3) \arrow[r] & (d_1,d_3) &&(d_1,d_2,d_3) \arrow[r] & (d_2,d_3).
\end{tikzcd}\]
For $d_3\notin\Z$, 
\[\mathcal{X}^+_{(d_1,d_3)}\simeq\P(\{a_i:a_id_1\in \Z\})\times\P(\{b_j:b_j(d_3-d_1)\in\Z\})\]
lies in the infinite divisor of the projective stack bundle $\mathcal{X}^+$ and thus lies in the isomorphism part of $\varphi^+$. For $d_3\in\Z$ and $d_1+d_2\notin\Z$, the morphism between the twisted sectors is given by
\[\begin{tikzcd}[column sep=3pt]
\mathcal{Y}_{(d_1,d_2,d_3)}=\P(\{a_i:a_id_1\in\Z\})\times\P(\{b_j:b_jd_2\in\Z\}) \arrow[d] & \\
\P(\{a_i:a_id_1\in\Z\}) \arrow[r,hook,"\text{zero section}"] & \mathcal{X}^+_{(d_1,d_3)}.
\end{tikzcd}\]
For $d_3\in\Z$ and $d_1+d_2\in\Z$, the morphism between twisted sectors
\[\mathcal{Y}_{(d_1,d_2,d_3)}\longrightarrow\mathcal{X}^+_{(d_1,d_3)}\]
is a weighted blow-up.
\subsection{Main results}
When orbifold Chow groups have nontrivial twisted sectors, the divisor-generation property of cohomology rings does not hold. For $\mathcal{X}^\pm$, to resolve this problem, we consider the set $S_o\subset N$ which consists of the extra $\operatorname{age}$ 1 box elements appearing in the mirror transform of the small $I$-function and find the following useful result about generating sets. 

\begin{theorem}{\rm(= Theorem\autoref{generator of quantum cohomology})}
    The orbifold quantum cohomology of ${\mathcal X}^\pm$ with product $\star_\tau$ restricted to $\tau\in A^1({\mathcal X}^\pm)+\langle y^s:s\in S_o\rangle$ is generated by $A^1({\mathcal X}^\pm)+\langle y^s:s\in S_o\rangle$.
\end{theorem}

For a smooth toric flop, i.e. an ordinary $\mathbb{P}^r$-flop, the extremal curve class $\ell^+$ is sent to $-(\ell^-)$ on the other side and thus the comparison can be made only after analytic continuation. Indeed, we consider the basic rational
function
\begin{equation*}
f(q^\ell) := \frac{q^\ell}{1 - (-1)^{r + 1}q^\ell} = \sum_{d \ge 1} (-1)^{(d
-1)(r + 1)} (q^\ell)^{d},
\end{equation*}
which satisfies the functional equation
\begin{equation*}
f(q^\ell) + f(q^{-\ell}) = (-1)^r.
\end{equation*}
Since all functions involved are expressed in terms of this basic rational function and its derivatives, the quantum product $\star$ is invariant under a smooth toric flop via such a simple analytic continuation. 
For a non-smooth toric flop, say  $\vec{a}=(1,1,1,1), \vec{b}=(2,2)$,
we find that 
\[(\xi-h)\star (\xi-h)=g_0^+(q^\ell)(\xi-h)^2,\quad g_0^+(q^\ell)=\sum_{m\geq 0}\binom{2m}{m}^2(q^\ell)^m.\]
Under the analytic continuation $q^\ell\mapsto (q^\ell)^{-1}$ via Mellin--Barnes integral, 
\[g_0^+\simeq \frac{i}{4\pi}g_1^-+\left(\frac{1}{4}-\frac{2i\log 2}{\pi}\right)g_0^-,\]
where $g_0^-=\sum_{k\geq 0}\frac{1}{2^{8n}}(q^\ell )^{k+1/2}\binom{2k}{k}^2$ and $g_1^-=g_0^-\log (q^\ell)+4\sum_{k\geq 1}\frac{1}{2^{8n}}(q^\ell)^{k+1/2}\binom{2k}{k}^2\sum_{n=1}^k\frac{1}{2k-1}$.
However, Gromov--Witten invariants are by definition logarithm-free, so a suitable framework for generalizing analytic continuation is needed. 

Our strategy is to put all functions in the set of coefficients for the small quantum product $\star$ on $\mathcal X^+$ into some Picard--Vessiot extension $F^+$, to put all functions in the set of coefficients for the small quantum product $\star$ on $\mathcal X^-$ under $q^\ell\mapsto (q^{\ell'})^{-1}$ into another Picard--Vessiot extension $F^-$ and then construct a differential isomorphism from $F^+$ to $F^-$ to obtain the regularization map. The notion of {\it regularization maps} was created by a joint idea with S.-Y. Lee (\cite{Lee26}).

In this paper, we carry out this strategy for {\it non-smooth toric flops of the simplest type}: 
\begin{equation}\label{ab}
\begin{split}    
\vec a &=(1)^{\times (2p+q)}=(1, 1, \ldots,1),\\ 
\vec b &=(2)^{\times p}\times (1)^{\times q}=(2,\ldots,2,1,\ldots,1)
\end{split}
\end{equation}
with $p\geq 1$, $q\geq 0$. In Section \ref{small-product}, we identify all the formulas of the small quantum product with $h$ and $\xi$, and all the functions that appear in them. The key principle underlying the derivation of all formulas is the following lemma.
\begin{lm}{\rm(= Lemma\autoref{birkhoff factorization})}
If $\{T_\mu\}$ is a basis for $A^*_{\operatorname{orb}}(\mathcal X)$, then for each $\mu$, there exists a mirror transformation $\tau$ and an operator $P_\mu\in \{Q^{d_s}:s\in \bbox(\fSigma)\}^{-1}\C[z][\![\operatorname{NE}(\mathcal X),x]\!]\langle z\frac{\partial}{\partial t_\rho}:\rho\in\Sigma(1)\rangle$ such that $P_\mu I^{S_o}=(\hat T_\mu J)(\tau,-z)$.   
\end{lm}
For computing a small quantum product on $A^*_{\operatorname{orb}}(\mathcal X)$, we make use of the small $I$-function. Since the nonzero box elements $(d_1,d_3)$ with $d_3\in\Z$ have ages greater than 1, this implies that $S_o = \emptyset$ and thus $I(t_1h+t_2\xi)=J(t_1h+t_2\xi)$ and by Theorem\autoref{generator of quantum cohomology}, the small quantum cohomology ring is generated by divisors. We compute two-point invariants, i.e. $h\star_{small}$ and $\xi{\star_{small}}$, by extracting the $z^0$-terms in $\hat h P_\mu(z^{-1}I)=h\star_{small}T_\mu+O(z^{-1})$ and $\hat \xi P_\mu(z^{-1}I)=\xi\star_{small}T_\mu+O(z^{-1})$ from the above lemma and \eqref{qde}.  
Our entire effort is devoted to finding $P_\mu$ for  each $T_\mu$ in a fixed basis of $A^*_{\operatorname{orb}}(\mathcal X)$. 

We now proceed to find a Picard--Vessoit extension that contains all relevant functions.  
To this end, we introduce a hypergeometric differential operator $\mathcal{L}$  from the Picard--Fuchs equation $\square_\ell$ of $I$ and the Picard--Vessiot extension $F_{\mathcal{L}}$ of $\mathcal{L}$ is the differential field we are looking for. 
\begin{thm}{\rm(= Corollary\autoref{proper-field})}
The differential fields
\[F^+((Q_2^+)^{1/2})(\eqqcolon F_0^+) \text{ and } F^-((Q_1^-Q_2^-)^{1/2})(\eqqcolon F_0^-)\]
are the smallest differential field extensions of $\mathbb{C}(Q^+_1,(Q^+_2)^{1/2})$ and $\mathbb{C}(Q_1^-,(Q_1^-Q_2^-)^{1/2})$ respectively that contain all three-point invariants $\langle D,T_1,T_2 \rangle^{\mathcal{X}^\pm}$ for any divisor $D$ and arbitrary orbifold cohomology classes $T_1$ and $T_2$.
\end{thm}
Any orbifold cohomology class $[\mathscr{F}]\in A^*_{\operatorname{orb}}(\mathcal{X}^+\times\mathcal{X}^-)$ induces an additive group homomorphism
\[\begin{tikzcd}[row sep=0pt]
A^*_{\operatorname{orb}}(\mathcal{X}^+) \arrow[rr] && A^*_{\operatorname{orb}}(\mathcal{X}^-)  \\
T \arrow[rr,mapsto] && (p_-)_*\left([\mathscr{F}]\cup_{\operatorname{orb}}p_+^*I^*T\right),
\end{tikzcd}\]
where $p_\pm$ is the projection from $\mathcal{X}^+\times\mathcal{X}^-$ to $\mathcal{X}^\pm$ and $[\mathscr{F}]$ preserves degree of classes if and only if $[\mathscr{F}]\in A_{\operatorname{orb}}^{\dim\mathcal{X}^+}(\mathcal{X}^+\times\mathcal{X}^-)$. 
Our goal is to find a ``quantum correspondence" arising from the fiber product of the flop and it was shown in \cite{MR5082801} that for toric flops, the fiber product shares the same support as the graph closure, whose normalization coincides with the common blow-up. Therefore, the correspondence should come from  the common blow-up $\mathcal{Y}$ of $\mathcal{X}^+$ and $\mathcal{X}^-$  with corrected orbifold degree 
\[[\mathscr{F}_{c,c_\infty}]=[\mathcal{Y}_{(0,0,0)}]+(-1)^{p+q}\sum_{i=0}^{p-1}c_ih_+^{p-1-i}h_-^i[\mathcal{Y}_{(0,\frac{1}{2},0)}]+c_\infty[\mathcal{Y}_{(1,\frac{1}{2},\frac{1}{2})}],\]
where $h_\pm$ is the pullback of $h$ from $\mathcal{X}^\pm$ to $\mathcal{Y}$. To ensure a perfect correspondence between the orbifold Chow groups and the quantum products, we introduce the following new notion of ``quantum correspondence".
\begin{definition}{\rm(= Definition \ref{quantum correspond})}
A \textbf{quantum correspondence} for $\mathcal X^+\dashrightarrow\mathcal X^-$ consists of the following data:
\begin{enumerate}
    \item a correspondence $[\mathscr F]\in A^{\dim \mathcal X}_{\operatorname{orb}}(\mathcal X\times \mathcal X')_\C$ preserving the orbifold Poincar\'e pairing,

    \item differential field extensions $L^\pm$ of $\C(Q^\pm_1,Q_2^\pm)$ containing the $n$-point invariants with $n\geq 3$,
    
    \item a differential field isomorphism $\phi:L^+\to L^-$ 
    under $Q^d\mapsto Q^{\mathscr F(d)}$ for $d=\ell^+$, $\gamma^+$.
    \end{enumerate}
    satisfying that the induced map
    \[\Phi\coloneqq [\mathscr{F}]\otimes\phi:A_{\operatorname{orb}}^*(\mathcal{X}^+)_\C\otimes L^+[\![\tau^\mu]\!]\longrightarrow A_{\operatorname{orb}}^*(\mathcal{X}^-)_\C\otimes L^-[\![\tau^\mu]\!]\]
    is a ring isomorphism with respect to the big quantum product structure, that is,
    \begin{equation}
    \Phi(T_1\star_{\tau_{\operatorname{big}}}T_2)=\Phi(T_1)\star_{\Phi(\tau_{\operatorname{big}})}\Phi(T_2),
    \end{equation}
    where $\tau_{\operatorname{big}}=\sum\tau^\mu T_\mu$ and $\Phi$ acts trivially on $\tau^\mu$.
We denote it by $(\phi,[\mathscr F])$ and call $\phi$ a \textbf{regularization map}.
\end{definition}
To find the partner $\phi$ of the correspondence $[\mathscr{F}_{c,c_\infty}]$, we introduce a bilinear form on the solution space $V_{\mathcal{L}}$ of $\mathcal{L}$ to achieve it.
\begin{lm}{\rm (See Proposition \ref{bilinear_pairing})}
If $G_{\mathcal{L}}$  is the differential Galois group of $F_{\mathcal{L}}$ over $\mathbb{C}(Q)$, then there exists a $G_\mathcal{L}$-invariant nondegenerate bilinear form $B_{\mathcal{L}}:V_{\mathcal L}\times V_{\mathcal L}\to\C$. Moreover, if $\{g_i\}_{i=0}^{n-1}$ is the normalized Frobenius basis of $\mathcal{L}$ at $Q=0$, then we have 
\[B_{\mathcal L}(g_i,g_j)=(-1)^j\delta_{i+j,n-1}.\]
\end{lm}
By Theorem \ref{well_def_regularization},
for any $c=(c_0,\ldots,c_{p-1})$ such that $c_ic_{p-1-i}=(-1)^{p+q}$, the $\C$-linear map
\[\begin{tikzcd}[row sep=1pt]
V_{\mathcal L^+} \arrow[rr,"\phi_c"] & & V_{\mathcal L^-} \\
g_i^+ \arrow[rr, maps to] && c_i^{-1}2^{-p}g_i^-
\end{tikzcd}\]
preserves the bilinear form  under classical analytic continuation $\gamma: Q_1^+ \rightarrow (Q_1^-)^{-1}$, so one can find $\sigma\in\mathrm{O}(V_{\mathcal{L}^-},B_{\mathcal{L}^-})=G_{\mathcal{L}^-}$ such that $\sigma(\gamma(g_i^+))=c_i^{-1}2^{-p}g_i^-$
and $\psi_c=\sigma\circ\gamma$ is a well-defined differential field isomorphism from $F^+$ to $F^-$. Although all the data we have collected so far come from small quantum products, the result for big quantum products can be obtained by applying Theorem \ref{generator of quantum cohomology} and WDVV equation. We are thus led to the following main theorem.
\begin{thm}{\rm(=Theorem\autoref{main result})}
 If $c_ic_{p-1-i}=(-1)^{p+q}$ for all $0\leq i\leq p-1$ and $c_\infty=1$, then $(\phi_c,[\mathscr{F}_{c,1}])$ is a quantum correspondence. 
\end{thm}
Once we have this crucial quantum correspondence in hand, the method in \cite{ILLW12} allows us to extend our result to all genera. 
\begin{thm}{\rm (= Theorem\autoref{higher-genus-invariance})}
The quantum correspondence $\Phi_c=\phi_c\otimes[\mathscr{F}_c]$ preserves the ancestor potentials in the sense that
\[\phi_c(\mathcal{A}_{\mathcal{X}^+}(t,\tau)^{48})=\mathcal{A}_{\mathcal{X}^-}(\Phi_c(t),\Phi_c({\tau}))^{48},\]
where $\Phi_c(t)$ acts trivially on $t_k^\mu$ and $\overline{\psi}^k$. Moreover, every Gromov--Witten invariant 
\[\gen{T_1,\ldots,T_n}_{g,n}^{\mathcal{X}^\pm}(\tau),\quad \forall g\geq 0\]
is valued in $F_0^\pm[\![\tau]\!]$, and $\Phi_c$ identifies Gromov--Witten potentials for all genus $g\geq 0$ in the sense that
\[\phi_c(\overline{F}_g^{\mathcal{X}^+}(0,\tau))=\overline{F}_g^{\mathcal{X}^-}(0,\Phi_c(\tau)).\]
\end{thm}
\subsection{Future direction}
 
 In the course of studying other examples, we came to realize that generalized analytic continuation plays an important role. Moreover, we have accumulated considerable evidence that finding the quantum correspondence is equivalent to finding the {\it appropriate regularization map}. We hope to complete this project by developing a unified approach in the subsequent work.

\emph{Acknowledgments.} 
We are grateful to C.-L. Wang for calling our attention to deal with the higher genus case. We also thank S.-Y. Lee for useful discussions related to this paper. H.-W. Lin is supported by the National Science and Technology Council (NSTC). We are grateful to Taida Institute of Mathematical Sciences (TIMS) for its constant support which makes this collaboration possible.

\section{Preliminary}\label{sec:preliminary}
\subsection{Gromov--Witten theory for smooth Deligne--Mumford stacks}
In this section, we briefly review the Gromov--Witten theory for smooth Deligne--Mumford stacks and particularly for toric stacks which were developed in \cite{abramovich2008gromov}, \cite{tseng2010orbifold} and \cite{coates2015mirror}. Let $\mathcal X$ be a proper smooth Deligne--Mumford stack with projective coarse moduli space $X$ and let $d\in H_2(X,\Z)\cap \operatorname{NE}(\mathcal X)$. Let $\overline{\mathcal M}_{g,n}(\mathcal X,d)$ be the moduli space of $n$-pointed stable maps $f:(\mathcal C,x_1,\ldots,x_n)\to \mathcal X$ from a twisted curve of genus $g$ and with $f_*[\mathcal C]=d$ (\cite{abramovich2008gromov}). The moduli space carries
\begin{enumerate}
    \item evaluation maps $\operatorname{ev}_i:\overline{\mathcal M}_{g,n}(\mathcal X,d)\to \mathcal I\mathcal X,$ for $1\leq i\leq n,$
    \item descendant classes $\psi_i\in A^1(\overline{\mathcal M}_{g,n}(\mathcal X,d))$ for $1\leq i\leq n,$
    \item weighted virtual fundamental class $[\overline{\mathcal M}_{g,n}(\mathcal X,d)]^{w}\in A_*(\overline{\mathcal M}_{g,n}(\mathcal X,d))$.
\end{enumerate} (See \cite{tseng2010orbifold} and \cite{coates2015mirror}).
For $T_1,\ldots, T_n\in A^*_{\operatorname{orb}}(\mathcal X),~k_1,\ldots,k_n\in\Z_{\geq0},$ the Gromov--Witten invariant with descendants is defined to be 
\[\langle T_1\psi_1^{k_1},\ldots, T_n\psi_n^{k_n}\rangle_{g,n,d}\coloneqq\int_{[\overline{\mathcal M}_{g,n}(\mathcal X,d)]^w}(\operatorname{ev_1}^*T_1)\psi_1^{k_1}\cdots(\operatorname{ev}_n^*T_n)\psi_n^{k_n}.\]
Axioms like WDVV equation, fundamental class axiom, string equation, divisor equation, dilaton equation and topological recursive relation for $g=0$ hold (see (\cite{tseng2010orbifold}).

The (big) quantum product $\star_\tau$ is defined by
\[T_\mu\star_\tau T_\nu=\sum_n\sum_d\frac{Q^d}{n!}\langle T_\mu,T_\nu,T^\kappa,\tau^{\otimes n}\rangle_{0,n+3,d}T_\kappa=\gen{T_\mu,T_\nu,T^\kappa}(\tau)\cdot T_\kappa,\]
where $\{T_\mu\}$ is a basis for $A^*_{\operatorname{orb}}(\mathcal X)$ and $\{T^\mu\}$ denotes its dual basis with respect to orbifold Poincar\'e pairing.
\subsection{Givental's Lagrangian cone}The Givental's symplectic space is 
$$\mathcal H=A_{\operatorname{orb}}^*(\mathcal X)[\![\operatorname{NE(\mathcal X)\cap H_2(X,\Z)}]\!]((z^{-1}))$$
which is equipped with the symplectic form
\[\Omega(f,g)=\operatorname{Res}_{z=0}(f(-z),g(z))_\mathcal X^{\operatorname{orb}}\]
and has the polarization $\mathcal H=\mathcal H_+\oplus\mathcal H_-$
where $\mathcal H_+=A_{\operatorname{orb}}^*(\mathcal X)[\![\operatorname{NE(\mathcal X)\cap H_2(X,\Z)}]\!][z],~\mathcal H_-=z^{-1}A_{\operatorname{orb}}^*(\mathcal X)[\![\operatorname{NE(\mathcal X)\cap H_2(X,\Z)}]\!][z^{-1}]$. The Givental's Lagrangian cone $\mathcal L_{\mathcal X}\subset\mathcal H$ consists of elements of the form
\[-1z+\mathbf t(z)+\sum_n\sum_d\frac{Q^d}{n!}\bigg\langle\mathbf t^{\otimes n},\frac{T_\mu}{-z-\psi}\bigg\rangle_{0,n+1,d}T^\mu\]for some $\mathbf t(z)\in \mathcal H_+$.
The $J$-function is defined to be a finite dimensional slice:
\[\tau=\sum t^\mu T_\mu(\in A^*_{\operatorname{orb}}(\mathcal X))\mapsto \mathcal L_\mathcal X\cap(-z+\tau+\mathcal H_-)=\{J(\tau,-z)\}.\] Using TRR, one can show that 
\begin{equation}\label{qde}
    z\frac{\partial}{\partial t^\mu}\left(z\frac{\partial}{\partial t^\nu} J\right)=\sum_\kappa C_{\mu,\nu}^\kappa(\tau)\left(z\frac{\partial}{\partial t^\kappa}J\right),\tag{QDE}
\end{equation}
which recovers the quantum product $T_\mu\star_\tau T_\nu=\sum_\kappa C_{\mu,\nu}^\kappa(\tau)T_\kappa$. The axioms of Gromov--Witten invariants reflect geometric property of $\mathcal L_\mathcal X$ (see \cite{tseng2010orbifold}): If $\mathbf f\in \mathcal L_\mathcal X, T_\mathbf f=T_\mathbf f\mathcal L_\mathcal X,$ then 
\begin{enumerate}
    \item $T_\mathbf f\cap\mathcal L_\mathcal X=zT_\mathbf f$
    \item $\mathcal L_\mathcal X$ is ruled by $zT_{J(\tau,-z)}$ over $\tau\in A_{\operatorname{orb}}^*(\mathcal X)[\![\operatorname{NE}(\mathcal X)\cap H_2(X,\Z)]\!]$.
\end{enumerate}

\subsection{Orbifold cohomology of toric Deligne--Mumford stack}
The orbifold cohomology for a proper smooth Deligne--Mumford stack $\mathcal X$ is defined to be $A^*(\mathcal I\mathcal X)$ with multiplication $T_\mu\cup_{\operatorname{orb}}T_\nu=\sum_\kappa\langle T_\mu,T_\nu,T^\kappa\rangle_{0,3,0}T_\kappa,$ i.e. the quantum product restricted to $d=\tau=0$. For a toric Deligne--Mumford stack, a presentation of ``Stanley-Reisner" style has been given in \cite{BCS}:
Let $\Q[N]^\Sigma$ denote the $\Q$-vector space of the group ring $\Q[N]$ with multiplication
\[y^{c_1}\cdot y^{c_2}=\begin{cases}
    y^{c_1+c_2}&\text{if there exists }\sigma\in \Sigma\text{ such that $\bar c_1,\bar c_2\in\sigma$},\\
    0&\text{otherwise}.
\end{cases}\]
Then there is a graded-algebra-isomorphism
\[A_{\operatorname{orb}}^*(\mathcal X)\simeq\frac{\Q[N]^\Sigma}{\left\langle\sum_{\rho\in\Sigma(1)}\theta(v_\rho)y^{v_\rho}:\theta\in\hom(N,\Z)\right\rangle}\]
under the following identification: For $v\in N,$ let $\sigma\in\Sigma$ be a cone containing $v$, we write $v=\sum_{\rho\in\sigma(1)} a_\rho \bar v_\rho$ for some $a_\rho\geq0$. We set $s=\sum_\rho\langle a_\rho\rangle \bar v_\rho\in\operatorname{Box(\fSigma)}$ and $\deg_{\operatorname{orb}} y^v=\sum a_\rho$. Then $y^v$ is identified with the cohomology class in ordinary cup product
\[\prod_{\rho\in\sigma(1)}i^*D_\rho^{\lfloor a_\rho\rfloor}y^s\in A^*(\mathcal X(\fSigma/\sigma(\bar s)))\subset A^*(\mathcal I\mathcal X)\]
where $y^s$ denotes the fundamental cycle of $\mathcal X(\fSigma/\sigma(\bar s))$, $i:\mathcal X(\fSigma/\sigma(\bar s))\to\mathcal X$ is the closed immersion and $D_\rho$ is the first chern class associated to the tautological line bundle $L_\rho,$ $\rho\in\Sigma(1)$.

\subsection{Mirror theorem}
The mirror theorem for toric Deligne--Mumford stacks with semi-projective coarse moduli spaces was established into \cite{coates2015mirror}.
Let $S\to \operatorname{Box}(\fSigma)$ be a map from a set $S$ in $N$.

\begin{enumerate}
    \item The $S$-extended Mori cone $\operatorname{NE}^S(\mathcal X)$ is identified as the image of the embedding $\operatorname{NE}(\mathcal X)_\Q\oplus\bigoplus_{s\in S}\Q_{\geq0}\to \mathbb L^S_\Q=\ker (\beta^S:\Q^{\Sigma(1)\sqcup S}\to N)$ via
    \begin{equation}
        (d,(n_s)_{s\in S})\mapsto(D_\rho.d-\sum_{s\in S} s_\rho n_s)_\rho\times(n_s)_{s\in S},
    \end{equation}
    where $s_\rho\in\Q_{\geq 0}$ such that $\bar s=\sum_{\rho\in \sigma(\overline{s})} s_\rho \bar{v}_\rho$.
    
    \item $\Lambda^S(\subset\mathbb L_\Q^S)$ consists of $(\lambda_\rho)_\rho\times(n_s)_{s\in S}$ such that $n_s\in\Z$ for all $s\in S$ and $\{\rho:\lambda_\rho\notin\Z\}$ defines a cone in $\Sigma$. 
    
    \item The reduction function $\nu^S:\Lambda^S\to\operatorname{Box}(\fSigma)$ is defined by
    \begin{align}
        \lambda\mapsto\sum_\rho\lceil\lambda_\rho\rceil v_\rho+\sum_{s\in S}\lceil n_s\rceil s.
    \end{align}
    Note that $\bar \nu^S(\lambda)=\sum_{\rho}\langle-\lambda_\rho\rangle \bar v_\rho\in \bar N$ and every $s\in\operatorname{Box}(\fSigma)$ is of the form $\sum_\rho\lceil\lambda_\rho\rceil v_\rho$ for some $(\lambda_\rho)_\rho\in\Lambda^{\emptyset}$.
    
    \item For $\lambda=(d,(n_s)_{s\in S})\in H_2^S(X)$, the extended Novikov variables are 
    $$\widetilde Q^\lambda:=q^de^{\sum_\rho t_\rho(D_\rho.d)}\prod_{s\in S}x_s^{n_s}, \quad Q^\lambda:=Q^d:=q^de^{\sum_\rho t_\rho(D_\rho.d)}.$$
    
    \item The $A^*_{CR}(\mathcal X)[\![\operatorname{NE}^S(X)]\!]((z^{-1}))$-valued function
    \begin{equation}
    I^S(t,x,Q,z)=ze^{\frac{\sum_\rho t_\rho D_\rho}{z}}\sum_{\lambda\in\Lambda^S\cap \operatorname{NE}^S(X)}[y^{v^S(\lambda)}]\frac{\widetilde Q^\lambda}{\prod_{s\in S}(n_s!z^{n_s})}\prod_{\rho}\frac{\prod_{a\leq0,\langle a\rangle=\langle \lambda_\rho\rangle}(D_\rho+az)}{\prod_{a\leq\lambda_\rho,\langle a\rangle=\langle\lambda_\rho\rangle}(D_\rho+az)}
    \end{equation} lies in the Lagrangian cone, where $y^s$ denotes the fundamental class of the inertia component $\mathcal X(\fSigma/\sigma(\bar s))\subset\mathcal I\mathcal X$.
\end{enumerate}

\section{Generation by divisors}
The divisor-generating property of cohomology rings is important in reconstructing the Gromov--Witten invariants from invariants of fewer points for the case of smooth toric varieties. When orbifold Chow groups have nontrivial twisted sectors, the property does not hold. In this section, we will give a strategy to resolve it for local models of simplicial toric flops.

\subsection{Local models for simplicial toric flops} Consider the sequences of the positive integers $\vec{a}=(a_0,\ldots,a_r)\in \N^{r+1}$, $\vec{b}=(b_0,\ldots,b_{r'})\in \N^{r'+1}$ and $v_0, \dots, v_r, w_0,\dots, w_{r'}$ be $(r+r'+2)$ elements. We define
\begin{enumerate}
    \item the finitely generated abelian group $N=\Z v_0\oplus\cdots\oplus\Z w_{r'}/(\sum a_i v_i-\sum b_j w_j)$,
    \item the homomorphism $\beta:\Z v_0\oplus\cdots\oplus\Z w_{r'}\oplus\Z e\to N,~v_i\mapsto v_i,w_j\mapsto w_j,e\mapsto -\sum a_iv_i$,
    \item the homomorphism $\widetilde\beta:\Z v_0\oplus\cdots\oplus\Z w_{r'}\oplus\Z e\oplus\Z(-e)\to N,~v_i\mapsto v_i,w_j\mapsto w_j,e\mapsto -\sum a_iv_i,~-e\mapsto \sum a_iv_i$,
    \item $\Sigma^+$ to be the fan in $N_\R$ containing maximal cones in
\begin{align*}
&\{\Cone(v_0,\ldots,\widehat{v}_i,\ldots,v_r,w_0,\ldots,\widehat{w}_j,\ldots,w_{r'},e):0\leq i\leq r,0\leq j\leq r'\}\\&\cup\{\Cone(v_0,\ldots,\widehat{v}_i,\ldots,v_r,w_0,\ldots,w_{r'}):0\leq i\leq r\},
\end{align*}
 \item $\Sigma^-$ to be the fan in $N_\R$ containing maximal cones in
\begin{align*}
&\{\Cone(v_0,\ldots,\widehat{v}_i,\ldots,v_r,w_0,\ldots,\widehat{w}_j,\ldots,w_{r'},e):0\leq i\leq r,0\leq j\leq r'\}\\&\cup\{\Cone(v_0,\ldots,v_r,w_0,\ldots,\widehat{w}_j,\ldots,w_{r'}):0\leq j\leq r'\},
\end{align*}
\item $\widetilde\Sigma$ to be the common subdivision of $\Sigma$ and $\Sigma'$ defined by adjoining $-e=\sum a_iv_i$,
\item three stacky fans $\fSigma^+:=(N,\Sigma^+,\beta),~\fSigma^-:=(N,\Sigma^-,\beta)$, and $~\widetilde\fSigma:=(N,\widetilde\Sigma,\widetilde\beta)$.
\end{enumerate}
The local model of a toric flop is defined to be the associated toric stack (rational) morphism (following the construction in \cite{BCS})\[\mathcal X^+\coloneqq\mathcal X(\fSigma^+)\dashrightarrow\mathcal X^-\coloneqq \mathcal X(\fSigma^-).\]
Alternatively, $\mathcal X^+$ is constructed as the quotient stack
\[\mathbb{P}(\vec{a};\vec{b})\coloneqq\left[\quotient{\left((\C^{r+1}\setminus \{0\})\times (\C^{r'+2}\setminus \{0\}\right)}{\ (\C^\times)^2}\right]\]
under the action of $(\C^\times)^2$ on $(\C^{r+1}\setminus \{0\})\times(\C^{r'+2}\setminus \{0\})$ by
\[\begin{tikzcd}[row sep=1pt]
(\C^\times)^2 \arrow[rr] && (\C^{r+1}\setminus \{0\})\times(\C^{r'+2}\setminus \{0\}) \\
(s,u) \arrow[rr,|->] && (s^{\vec{a}},u^{\vec{b}}s^{-\vec{b}},u).
\end{tikzcd}\]
For $\mathcal X^-$, the quotient stack is defined similarly by interchanging $\vec a$ and $\vec b$. The birational map is resolved by considering the weighted blow-up of $\P(\vec{a};\vec{b})$ along the closed substack
\[\left[\quotient{\left((\C^{r+1}\setminus \{0\})\times \{0\}^{\times (r'+1)}\times\C^\times\right)}{\ (\C^\times)^2}\right]\simeq\P(\vec{a}),\]
which turns out to be the weighted projective stack bundle $\mathcal Y=\mathbb{P}_{\P(\vec{a})\times\P(\vec{b})}(\O\oplus\O(-1,-1))$. Indeed, $\mathcal{Y}$ is a quotient stack via
\[\begin{tikzcd}[row sep=1pt]
(\C^\times)^3 \arrow[rr] && (\C^{r+1}\setminus \{0\})\times(\C^2\setminus \{0\})\times(\C^{r'+1}\setminus \{0\}) \\
(s,t,u) \arrow[rr,|->] && (s^{\vec{a}},u,s^{-1}t^{-1}u,t^{\vec{b}}),
\end{tikzcd}\]
and is exactly the toric stack corresponding to $\widetilde\fSigma$. Also, the blow-up morphism $$\varphi^+:\mathbb{P}_{\P(\vec{a})\times\P(\vec{b})}(\O\oplus\O(-1,-1))\to\P(\vec{a},\vec{b})$$ comes from the morphism
\[\begin{tikzcd}[row sep=1pt]
(\C^{r+1}\setminus \{0\})\times(\C^2\setminus \{0\})\times(\C^{r'+1}\setminus \{0\}) \arrow[rr] && (\C^{r+1}\setminus \{0\})\times (\C^{r'+2}\setminus\{0\}) \\
(\vec{x},z_1,z_2,\vec{y}) \arrow[rr,|->] && (\vec{x},\vec{y}z_2^{\vec{b}},z_1)
\end{tikzcd}\]
together with the morphism $(\C^\times)^3\to(\C^\times)^2$ defined by $(s,t,u)\mapsto (s,u)$. The morphism $\varphi^+$ restricted on exceptional locus is the projection $\P(\vec{a})\times\P(\vec{b})\to\P(\vec{a})$ and we can
contract the $\P(\vec a)$-fiber to get the morphism $\varphi^-:\mathcal Y\to \mathcal X^-$ via
\[\begin{tikzcd}[row sep=1pt]
    (\C^{r+1}\setminus \{0\})\times(\C^2\setminus \{0\})\times(\C^{r'+1}\setminus \{0\}) \arrow[rr] && (\C^{r'+1}\setminus \{0\})\times (\C^{r+2}\setminus\{0\}) \\
(\vec{x},z_1,z_2,\vec{y}) \arrow[rr,|->] && (\vec y, \vec x z_2^{\vec a}, z_1)
\end{tikzcd}\]
together with the morphism $(\C^\times)^3\to(\C^\times)^2,~(s,t,u)\mapsto (t,u)$. 

\subsection{Mori cone and box elements}
$A^1_\Q(\mathcal X^+)$ is generated by $h^+,\xi^+$ which correspond to the line bundles associated to the characters $(s,u)\mapsto s$ and $(s,u)\mapsto u$ respectively. Indeed, $D_{v_i}=a_i h^+,~D_{w_j}=b_j(\xi^+-h^+),~D_e=\xi^+$ for $0\leq i\leq r,0\leq j\leq r'$.
The Mori cone $\operatorname{NE}_\Q(\mathcal X^+)$ is generated by classes $\ell^+,\gamma^+$ representing exceptional curves and fiber curves respectively. We have the pairing:

\[\begin{tabular}{c|c|c}
& $h^+$ & $\xi^+$ \\
\hline $\ell^+$ & $1$ & $0$ \\
\hline $\gamma^+$ & $0$ & $1$ \\
\end{tabular}\]
Hence $\bbox(\fSigma^+)$ is parametrized as follows by $\nu=\nu^{S=\emptyset}$ defined in Section 2.4:
\[\Lambda_{\mathcal{X}^+}\coloneqq\left\{(d_1,d_3)\in\Q_{\geq 0}^2\middle|\ d_1\in\bigcup_i\tfrac{1}{a_i}\Z,\ d_3\in \left(d_1+\bigcup_j\tfrac{1}{b_j}\Z\right)\cup\Z \right\}\longrightarrow\bbox(\fSigma^+)\]
\[\nu^+: (d_1,d_3)\longmapsto\sum_{i=0}^r\lceil a_id_1\rceil v_i+\sum_{i=0}^{r'}\lceil b_i(d_3-d_1)\rceil w_i+\lceil d_3\rceil e.\]
Similarly, we have $h^-$, $\xi^-$, $\ell^-$, $\gamma^-$ for $\mathcal{X}^-$ and the parametrization for $\bbox(\fSigma^-)$:
\[\Lambda_{\mathcal{X}^-}\coloneqq\left\{(d_2,d_3)\in\Q_{\geq 0}^2\middle|\ d_2\in\bigcup_j\tfrac{1}{b_j}\Z,\ d_3\in\left(d_2+\bigcup_i\tfrac{1}{a_i}\Z\right)\cup\Z\right\}\longrightarrow\bbox(\fSigma^-)\]
\[\nu^-: (d_2,d_3)\longmapsto\sum_{i=0}^r\lceil a_i(d_3-d_2)\rceil v_i+\sum_{i=0}^{r'}\lceil b_id_2\rceil w_i+\lceil d_3\rceil e.\]
Also, we have the parametrization for $\bbox(\widetilde\fSigma)$:
\[\Lambda_{\mathcal{Y}}=\left\{(d_1,d_2,d_3)\in\Q_{\geq 0}^3\middle|\ d_1\in\bigcup_i\tfrac{1}{a_i}\Z,\ d_2\in\bigcup_j\tfrac{1}{b_j}\Z,\ d_3\in\Z\cup\left(d_1+d_2+\Z\right)\right\}\longrightarrow\bbox(\widetilde{\fSigma})\]
\[(d_1,d_2,d_3)\longmapsto\sum_{i=0}^r\lceil a_id_1\rceil v_i+\sum_{i=0}^{r'}\lceil b_id_2\rceil w_i+\lceil d_3\rceil e+\lceil d_3-d_2-d_1\rceil(-e).\]

By abuse of notation, we will write $(d_1,d_3)\in\bbox(\fSigma^+)$ in place of its image and similarly for $\bbox(\fSigma^-)$ and $\bbox(\widetilde\fSigma)$. As shown in \cite{BCS}, each box element corresponds to an inertia component and we can determine them for $\mathcal X^+$, $\mathcal X^-$ and $\mathcal Y$ respectively as follows.
For $d\in\Z$, let $\vec a(d)\coloneqq\{a_i:{a_i d\in\Z}\}$, $\vec b(d)\coloneqq\{b_j:{b_j d\in\Z}\}$.
\begin{enumerate}        
    \item $\X^+_{(d_1,d_3)}=\begin{cases}
        \P(\vec a(d_1);\vec b (d_1))&d_3\in\Z\\
        \P(\vec a(d_1))\times \P(\vec b(d_3-d_1))&d_3\notin\Z
    \end{cases}$
    \item $\X^-_{(d_2,d_3)}=\begin{cases}
        \P(\vec b(d_2);\vec a (d_2))&d_3\in\Z\\
        \P(\vec b(d_2))\times \P(\vec a(d_3-d_2))&d_3\notin\Z
    \end{cases}$
    \item $\mathcal Y_{(d_1,d_2,d_3)}=\begin{cases}
        \P_{\P(\vec a(d_1))\times \P(\vec b(d_2))}(\mathcal O\oplus \mathcal O(-1,-1))&\text{\fbox{Type \rom{1}} if }d_3\in\Z\text{ and }d_1+d_2\in\Z\\\P(\vec a(d_1))\times \P(\vec b(d_2))&\text{\fbox{Type \rom{2}} if }d_3\in\Z\text{ and }d_1+d_2\notin\Z\\
        \P(\vec a(d_1))\times\P(\vec b(d_2))&\text{\fbox{Type \rom{3}} if }d_3\notin\Z
    \end{cases}$
\end{enumerate}

The simplicial toric flop $\mathcal X^+\leftarrow\mathcal Y\rightarrow \mathcal X^-$ induces morphisms on inertia components 
$\mathcal X^+_{(d_1,d_3)}\leftarrow \mathcal Y_{(d_1,d_2,d_3)}\rightarrow \mathcal X^-_{(d_2,d_3)}$ which can be classified into the following three types.
\begin{enumerate}[label=\fbox{Type \Roman*}, leftmargin=*, align=left]
    \item 
    \[\begin{tikzcd}
    &\arrow[ld,"\text{blow up}",swap]\P_{\P(\vec a(d_1))\times \P(\vec b(d_2))}(\mathcal O\oplus \mathcal O(-1,-1))\arrow[rd,"\text{blow up}"]&\\
    \P(\vec a(d_1);\vec b (d_1))&&\P(\vec b(d_2);\vec a (d_2))
    \end{tikzcd}\]
    \item 
    \[\begin{tikzcd}[column sep=40pt]
    \arrow[d,"\text{zero section}",swap]\P(\vec a(d_1))&\arrow[l,"\text{projection}",swap]\arrow[ld]\P(\vec a(d_1))\times \P(\vec b(d_2))\arrow[r,"\text{projection}"]\arrow[rd]&\P(\vec b(d_2))\arrow[d,"\text{zero section}"]\\
    \P(\vec a(d_1);\vec b (d_1))&&\P(\vec b(d_2);\vec a (d_2))
    \end{tikzcd}\]
    \item 
    \[\begin{tikzcd}
    &\arrow[ld,"\operatorname{id}",swap]\P(\vec a(d_1))\times \P(\vec b(d_2))\arrow[rd,"\operatorname{swap}"]& \\
    \P(\vec a(d_1))\times \P(\vec b(d_2))&&\P(\vec b(d_2))\times \P(\vec a(d_1))
    \end{tikzcd}\]
\end{enumerate}

\subsection{Generators for quantum product}
In this section, $\mathcal X=\mathcal X(\fSigma)$ is either $\mathcal X(\fSigma^+)$ or $\mathcal X(\fSigma^-)$.
Let $S_o$ consist of elements in $\bbox(\fSigma^+)$ such that $\operatorname{age}(s)=1,\sigma(\bar s)\in\Sigma^+\setminus \Sigma^-$ or elements in $\bbox(\fSigma^-)$ such that $\operatorname{age}(s)=1,\sigma(\bar s)\in\Sigma^-\setminus \Sigma^+$. By counting degree, one can show that $I^{S_o}=J(\tau,z)$ for some mirror transformation $\tau=t_1h+t_2\xi+\sum_{s\in S_o}\tau_s y^s$ with $\tau\equiv t_1h+t_2\xi+\sum _{s\in S_o}x_sy^s~\operatorname{mod}Q,|x|^2$. In particular, $\tau$ is invertible as a formal series in $Q,x$. We will show that the orbifold quantum cohomology of $\mathcal X$ with quantum product $\star_\tau$ restricted to $\tau\in A^1(\mathcal X)+\langle y^s:s\in S_o\rangle$  is generated by $A^1(\mathcal X)+\langle y^s:s\in S_o\rangle$. Here, the product $\star_\tau$ can be regarded as a extended small quantum product. Define the localized operator ring
\[\mathcal{R}\coloneqq\{Q^{d_s}:s\in \bbox(\fSigma)\}^{-1}\C[z][\![\operatorname{NE}(\mathcal X),x]\!][z{\partial}/{\partial t_\rho}:\rho\in\Sigma(1)],\]
which inherits the natural $x$-adic structure. One has the following important lemma.

\begin{lemma}\label{birkhoff factorization}
    If $\{T_\mu\}$ is a basis for $A^*_{\operatorname{orb}}(\mathcal X)$,
    then for each $\mu$, there exists a mirror transformation $\tau$ and an operator $P_\mu\in\mathcal{R}$ such that $P_\mu I^{S_o}=(\hat T_\mu J)(\tau,-z)$.
\end{lemma}
\begin{proof}
    Let $S=\bbox(\fSigma)$. We prove the statement for $I^S=I^S(t,(x_s)_{s\in S})$ and the assertion of the lemma follows from that $I^S|_{x_{s'}=0,\ s'\notin S}=I^{S_0}$ and the restriction to $x_{s'}=0$ $(s'\notin S)$ commutes with any operator in $\mathcal{R}$.
    
    For $s\in S$, we write $s=\sum_{\rho\in\sigma(\bar s)} s_\rho \bar v_\rho$ and pick $d_s\in\operatorname{NE}(\mathcal X)$ with $\lambda_\rho\coloneqq D_\rho.d_s\geq0$ such that $s=\sum_\rho \lceil\lambda_\rho\rceil v_\rho$.
    Each $s\in S$ corresponds to a Picard--Fuchs operator
    \[\square_s=z\frac{\partial}{\partial x_s}-Q^{-d_s}\prod_{\rho}\prod_{a=1}^{\lceil\lambda_\rho\rceil}\left(z\frac{\partial}{\partial t_\rho}-\sum_{s'\in\bbox(\fSigma)} s'_\rho x_{s'}z\frac{\partial}{\partial x_{s'}}-(a-1)z\right),\qquad\square_sI^S=0.\]
    Hence $z\partial_{x_s}I^S=R_sI^S$ for some $R_s\in\mathcal R[zx_{s'}\partial_{x_{s'}}:s'\in S]$. A further substitution of $z\partial_{x_{s'}}I^S$ by $R_{s'}I^S$ $(s'\in S)$ into $R_sI^S$ raises the $x$-degree of the coefficients of the terms involving the differentials of the form $z\6_{x_{s'}}$ by 1. Iterating this process in the $x$-adic topology, an operator $\widetilde R_s\in\mathcal R$ is constructed such that $\widetilde{R}_s I^S=z\partial_{x_s}I^S$. Thus it suffices to show that $P_\mu$ exists in 
    \[\{Q^{d_s}:s\in \bbox(\fSigma)\}^{-1}\C[z][\![\operatorname{NE}(\mathcal X),x]\!][z\partial/\partial t_\rho,z\partial/\partial x_s:\rho\in\Sigma(1),s\in S].\]
    It suffices to prove it for a particular basis which consists of elements of the form $T_\mu=D_{\rho_1}\cdots D_{\rho_k} y^s$. We define $\check T_\mu=z\frac{\partial}{\partial t_{\rho_1}}\cdots z\frac{\partial}{\partial t_{\rho_k}}\cdot z\frac{\partial}{\partial x_s}$. In particular, $\check{T}_\mu z^{-1}I^S\equiv T_\mu~\operatorname{mod}Q,x.$ 
    If $\check T_\mu z^{-1}I^S$ does not have terms with positive $z$-powers, then $\check T_\mu z^{-1}I^S=(\hat T_\mu z^{-1}J)(\tau)$. Otherwise, we carry out the Birkhoff factorization as follows. We pick a monomial order $>$ of $\Lambda^S\cap \operatorname{NE}^S(\mathcal{X})\subset (\frac{1}{m}\Z_{\geq0})^\Sigma\times\Z_{\geq0}^S\subset \Q^\Sigma\times \Q^S$ for large $m\in\N$. There exists $\lambda_0>0$ such that \begin{align}\label{naive_quan:q-expansion}\check T_\mu z^{-1}I^S=T_\mu+\sum_{0<\lambda<\lambda_0}\widetilde Q^{\lambda}C_\lambda+\widetilde Q^{\lambda_0}C_{\lambda_0}+\sum_{\lambda>\lambda_0}\widetilde Q^{\lambda}C_\lambda,\end{align}
    where $C_\lambda \in A^*_{\operatorname{orb}}(\mathcal X)\otimes z^{-1}\C[\![z^{-1}]\!]$ for $\lambda<\lambda_0$,  
    $C_\lambda \in A^*_{\operatorname{orb}}(\mathcal X)\otimes (\C[z]\oplus\C[\![z^{-1}]\!])$ for $\lambda\geq\lambda_0$, and $C_{\lambda_0}$ has negative valuation $-k$ in the variable $1/z$ for some $k\geq0$.
    
    If $C_{\lambda_0}=\sum_\nu z^kC_{k,\lambda_0}^\nu T_\nu+ \sum_{k'<k}\sum_\nu z^{k'}C_{k',\lambda_0}^\nu T_\nu$,
    then we subtract $\widetilde Q^{\lambda_0}z^k\sum_\nu C_{k,\lambda_0}^\nu \check T_\nu(z^{-1} I^S)$ from the expansion \eqref{naive_quan:q-expansion}. The $\widetilde Q^\lambda$-coefficient still lies in $A^*_{\operatorname{orb}}(\mathcal X)\otimes z^{-1}\C[\![z^{-1}]\!]$ for $\lambda<\lambda_0$ and 
    for $\lambda=\lambda_0$, the new $\widetilde Q^{\lambda_0}$-coefficient is 
    \[C'_{\lambda_0}=\sum_{k'<k}\sum_\nu z^{k'}C_{k',\lambda_0}^\nu T_\nu.\]
    Iterating the process until the $\widetilde Q^{\lambda_0}$-term has coefficient in $A^*_{\operatorname{orb}}(\mathcal X)\otimes z^{-1}\C[\![z^{-1}]\!]$. By induction on the monomial order, an operator $P_\mu$ is constructed formally in $\widetilde Q$ such that $z^{-1}P_\mu I^S\mathcal \in z^{-1}L_\mathcal X$ has no positive $z$-powers, and $z^{-1}P_\mu I^S= T_\mu+O(1/z)$. Hence it is $\hat T_\mu J$ up to a mirror transformation $\tau$, which can be determined by $P_{1}I^S=\hat 1J(\tau,z)=J(\tau,z)$.

    % of \[(\check T_\mu-\widetilde Q^{\lambda_0}z^k(\sum_\nu C_{k,\nu} \check T_\nu))(z^{-1} I^S)=T_\mu+\sum_{0<\lambda<\lambda_0}\widetilde Q^{\lambda}C_\lambda+\widetilde Q^{\lambda_0}C'_{\lambda_0}+\sum_{\lambda>\lambda_0}\widetilde Q^{\lambda}C'_\lambda\]
    % with new coefficients $C_\lambda'$ in $A^*_{\operatorname{orb}}(\mathcal X)\otimes (\C[z]\oplus\C[\![z^{-1}]\!])$ for $\lambda>\lambda_0$ 
    % and $$C'_{\lambda_0}=\sum_{k'<k}\sum_\nu z^{k'}C_{k',\lambda_0,\nu} T_\nu.$$ 
    % We iterate the 
    % process until the $\widetilde Q^{\lambda_0}$-term has coefficients in
    % $A^*_{\operatorname{orb}}(\mathcal X)\otimes z^{-1}\C[\![z^{-1}]\!]$. By induction on 
    % $\Lambda^S\cap\operatorname{NE}^S(\mathcal X)$ with respect to ``$>$", an operator $P_\mu$ is constructed formally in $\widetilde Q$ such that $z^{-1}P_\mu I^S\mathcal \in z^{-1}L_\mathcal X$ has no positive $z$-powers, and $z^{-1}P_\mu I^S= T_\mu+O(1/z)$. Hence it is $\hat T_\mu J$ up to a mirror transformation $\tau$, which can be determined by $P_{1}I^S=\hat 1J(\tau,z)=J(\tau,z)$.
\end{proof}

\begin{theorem}\label{generator of quantum cohomology}
    The orbifold quantum cohomology of $\mathcal X$ with quantum product $\star_\tau$ restricted to $\tau\in A^1(\mathcal X)+\langle y^s:s\in S_o\rangle$  is generated by $A^1(\mathcal X)+\langle y^s:s\in S_o\rangle$.
\end{theorem}
\begin{proof}
    By the previous lemma, there exists $P_\mu$ such that $P_\mu I^{S_o}(t,x)=\hat T_\mu J(\tau,z)$. After the coordinate change by the inverse mirror transform $t=t(\tau)=\tau_hh+\tau_\xi\xi+\sum_{s\in S_o} t_s(\tau)y^s$, the operator on the left hand side becomes an operator in
    \[\{q^{d_s}:s\in \bbox(\fSigma)\}^{-1}\C[z][\![\operatorname{NE}(\mathcal X),x]\!][z\partial/\partial \tau_\rho,z\partial/\partial \tau_s:\rho\in\Sigma(1),s\in S_o].\]
    Hence a relation between the quantization of $T_\mu$ and that of divisors, $\{y^s,s\in S_o\}$ is obtained by $P_\mu(z,t,Q,\hat D_\rho,\hat{y}^s)J=\hat T_\mu J$, which implies a relation in the quantum cohomology ring
    by $P_\mu(0,t,Q,D_\rho\star_\tau,y^s\star_\tau)1=T_\mu$ (cf. the proof of \cite[Theorem 10.3.1]{cox1999mirror}).
    
    Let $\tau_{\operatorname{big}}=\sum \tau^\mu T_\mu$. For $T_\mu$, if we define $s_\mu(\tau)=zT_\mu+z\sum_{n,d}\frac{q^d}{n!}\langle \frac{T_\mu}{z-\psi},T^\nu,\tau^{\otimes n}\rangle_{0,n+2,d} T_\nu$, then by TRR, we still have that $\hat T_\mu s_\nu=T_\mu\star_{\tau_{\tau_{\operatorname{big}}}} s_\nu$, i.e. $\{s_\nu\}$ forms a basis of flat sections of the Givental connection $\nabla^g_{T_\mu}=\hat T_\mu-T_\mu\star_{\tau_{\operatorname{big}}}$. Also, by the string equation, $J(\tau,z)=\sum(s_\mu,1)^{\operatorname{orb}}_\mathcal XT^\mu$ and the remaining arguments similar to the proof of \cite[Theorem 10.3.1]{cox1999mirror} still work.
\end{proof}

\begin{rmk}\label{rmk_of_generator_of_quantum_cohomology}
Let $F$ be a field containing special 3-point invariants $\langle T_0,T_1,T_2 \rangle^{\mathcal{X}}(\tau)$ for any $T_0\in A^1(\mathcal{X})+\gen{y^s:s\in S_0}$ and arbitrary orbifold cohomology classes $T_1$ and $T_2$. It is clear that $F[D_\rho\star_\tau,y^s\star_\tau]1\subseteq A^*_{\operatorname{orb}}(\mathcal{X})_\C\otimes F$. On the other hand, Theorem \ref{generator of quantum cohomology} ensures that there exists a field extension $\widetilde{F}$ such that $A_{\operatorname{orb}}^*(\mathcal{X})_\C\otimes \widetilde{F}\subseteq \widetilde{F}[D_\rho\star_\tau,y^s\star_\tau]1$. Hence we can conclude that
\[F[D_\rho\star_\tau,y^s\star_\tau:\rho\in A^1(\mathcal{X}),s\in S_0]1=A_{\operatorname{orb}}^*(\mathcal{X})_\C\otimes F\]
by the dimension counting.
\end{rmk}

\section{\texorpdfstring{Non-smooth toric flops of the simplest type ($\vec{a} =(1)^{\times (2p+q)}$, $\vec b=(2)^{\times p}\times (1)^{\times q}$}{a})}
We are devoted to the study of the case of $\vec a =(1)^{\times (2p+q)}$, $\vec b=(2)^{\times p}\times (1)^{\times q}$ with $p\geq 1$, $q\geq 0$. By Section 3.2, the associated simplicial toric flop $\mathcal X^+\leftarrow\mathcal Y\rightarrow \mathcal X^-$ induces morphisms on inertia components 
$\mathcal X^+_{(d_1,d_3)}\leftarrow \mathcal Y_{(d_1,d_2,d_3)}\rightarrow \mathcal X^-_{(d_2,d_3)}$. In this case, there are three morphisms on inertia components: 
\begin{enumerate}[leftmargin=*, align=left]
    \item[Type I : $\boxed{(d_1,d_2,d_3)=(0,0,0)}$]  \[\begin{tikzcd}
        &\arrow[ld,"\text{blow up}",swap]\mathcal Y_{(0,0,0)}=\mathcal Y\arrow[rd,"\text{blow up}"]&\\
        \mathcal X^+_{(0,0)}=\mathcal X^+&&\mathcal X^-_{(0,0)}=\mathcal X^-.
    \end{tikzcd}\]

    \item[Type II : $\boxed{(d_1,d_2,d_3)=(0,1/2,0)}$] \[\begin{tikzcd}
        \arrow[d,"\text{zero section}",swap]\P^{2p+q-1}&\arrow[l,"\text{projection}",swap]\arrow[ld]\P^{2p+q-1}\times \P((2)^{\times p})\arrow[r,"\text{projection}"]\arrow[rd]&\P((2)^{\times p})\arrow[d,"id"]\\
        \P(\vec a;\vec b)=\mathcal X^+&&\P((2)^{\times p};\emptyset)=\mathcal X^-_{(1/2,0).}
    \end{tikzcd}\]
    
    \item[Type III : $\boxed{(d_1,d_2,d_3)=(0,1/2,1/2)}$]  \[\begin{tikzcd}[column sep=0pt]
        &\arrow[ld,"\operatorname{id}",swap]\P^{2p+q-1}\times\P((2)^{\times p})\arrow[rd,"\operatorname{swap}"]&\\
        \P^{2p+q-1}\times \P((2)^{\times p})=\mathcal X^+_{(0,1/2)}&&\P((2)^{\times p})\times\P^{2p+q-1}=\mathcal X^-_{(1/2,1/2).}
    \end{tikzcd}\]
\end{enumerate}

Let $y_{\infty}^\pm$ be the fundamental class $[\mathcal X^\pm_{(0,1/2)}]\in A^*_{\operatorname{orb}}(\mathcal X^\pm)$ and $y_{1/2}^-=[\mathcal X^-_{(1/2,0)}]\in A^*_{\operatorname{orb}}(\mathcal X^-)$.
The orbifold cohomology $A^*_{\operatorname{orb}}(\mathcal X^+)$ is generated by $h^+,\xi^+,y^+_\infty$ with relations
\[(h^+)^{2p+q}=0,\ \xi^+(\xi^+-h^+)^{p+q}=0,\ (\xi^+-h^+)^py^+_\infty=0,\ (y_\infty^+)^2=\xi^+(\xi^+-h^+)^q,\]
and $A^*_{\operatorname{orb}}(\mathcal X^-)$ is generated by 
by $h^-,\xi^-,y^-_\infty,y_{1/2}^-$ with relations
\[(h^-)^{p+q}=0,\ \xi^-(\xi^--h^-)^{2p+q}=0,\ (h^-)^py_\infty^-=0,\ (\xi^--h^-)^{2p+q}y_\infty^-,\ (y_\infty^-)^2=\xi^-(h^-)^q,\]
\[(y_{1/2}^-)^2=(h^-)^q(\xi^--h^-)^{2p+q},\ (h^-)^py_{1/2}^-=0,\ \xi^-y_{1/2}^-=0.\]

\subsection{Correspondence}
Any orbifold cohomology class $[\mathscr{F}]\in A^*_{\operatorname{orb}}(\mathcal{X}^+\times\mathcal{X}^-)_{\mathbb{C}}$ induces an additive group homomorphism
\[\begin{tikzcd}[row sep=0pt]
A^*_{\operatorname{orb}}(\mathcal{X}^+) \arrow[rr] && A^*_{\operatorname{orb}}(\mathcal{X}^-)  \\
T \arrow[rr,mapsto] && (p_-)_*\left([\mathscr{F}]\cup p_+^*I^*T\right),
\end{tikzcd}\]
where $p_\pm$ is the projection from $\mathcal{X}^+\times\mathcal{X}^-$ to $\mathcal{X}^\pm$. By abuse of notation, we still denote this homomorphism by $[\mathscr{F}]$. By routine counting of dimension, $[\mathscr{F}]$ preserves degree of classes if and only if $[\mathscr{F}]\in A_{\operatorname{orb}}^{\dim\mathcal{X}^+}(\mathcal{X}^+\times\mathcal{X}^-)$. 

Our goal is to find a ``quantum correspondence" arising from the fiber product of the flop and it was shown in \cite{MR5082801} that for toric flops, the fiber product shares the same support as the graph closure, whose normalization coincides with the common blow-up. Therefore, we introduce an ansatz correspondence from  the common blow-up $\mathcal{Y}$ of $\mathcal{X}^+$ and $\mathcal{X}^-$  with corrected orbifold degree 
\[[\mathscr{F}_{c,c_\infty}]=[\mathcal{Y}_{(0,0,0)}]+(-1)^{p+q}\sum_{i=0}^{p-1}c_ih_+^{p-1-i}h_-^i[\mathcal{Y}_{(0,\frac{1}{2},0)}]+c_\infty[\mathcal{Y}_{(0,\frac{1}{2},\frac{1}{2})}],\]
where $h_\pm$ is the pullback of $h^\pm$ from $\mathcal{X}^\pm$ to $\mathcal{Y}$. We can show that the correspondence $[\mathscr{F}_{c,c_\infty}]$ sends
\begin{align*}
h^a(\xi-h)^b & \longmapsto\begin{cases}
h^b(\xi-h)^a & \!\!\!\text{\scalebox{0.8}{if $\phantom{2p+q\leq\ } a+b\leq p+q-1$}}  \\
h^b(\xi-h)^a+c_{a+b-p-q}(-1)^{p+q+b}h^{a+b-p-q}y^-_{1/2} & \!\!\!\text{\scalebox{0.8}{if $\phantom{2}p+q\leq a+b\leq 2p+q-1$}} \\
h^b(\xi-h)^a+(-1)^{2p+q-1-a}h^{a+b-2p-q}(\xi-h)^{2p+q
} & \!\!\!\text{\scalebox{0.8}{if $2p+q\leq a+b$}}
\end{cases} \\
h^a(\xi-h)^by^+_\infty & \longmapsto c_\infty h^b(\xi-h)^ay^-_\infty.
\end{align*}
For simplicity, we often omit the superscript $\pm$ on $h$ and $\xi$  unless it is necessary.

\begin{lm}\label{lm_preserve_orbifold_prod}
The correspondence $[\mathscr{F}_{c,c_\infty}]$ preserves the orbifold Poincar\'e pairing if and only if $c_\infty^2=1$ and $c_ic_{p-1-i}=(-1)^{p+q}$ for all $0\leq i\leq p-1$.
\end{lm}

\begin{proof}
For the basis
\[\{h^i(\xi-h)^j\mid 0\leq i\leq 2p+q-1,\ 0\leq j\leq p+q\}\cup\{h^i(\xi-h)^jy_\infty^+\mid 0\leq i\leq 2p+q-1,\ 0\leq j\leq p-1\}\]
of $A_{\operatorname{orb}}^*(\mathcal{X}^+)$, we determine its dual basis with respect to the orbifold Poincar\'e pairing as follows:
\begin{align*}
&\left\{2^{-p}\xi h^{2p+q-1-i}(\xi-h)^{p+q-1-j}:0\leq i\leq 2p+q-1,\ 0\leq j\leq p+q-1\right\} \\
\cup& \left\{2^{-p}h^{2p+q-1-i}:0\leq i\leq 2p+q-1\right\} \\
\cup& \{2^{-p}h^{2p+q-1-i}(\xi-h)^{p-1-j}y_\infty^+\mid 0\leq i\leq 2p+q-1,\ 0\leq j\leq p-1\}.
\end{align*}
It is clear that the result follows from 
\[1=\left([\mathscr{F}_{c,c_\infty}]h^i(\xi-h)^{p+q},[\mathscr{F}_{c,c_\infty}]2^{-p}h^{2p+q-1-i}\right)_{\mathcal{X}^-}=(-1)^{p+q}c_ic_{p-1-i}\]
when $0\leq i\leq p-1$, and $c_\infty^2=1$ from the pairing on the cohomology class in $\mathcal{X}^+_{(0,\frac{1}{2})}$.
\end{proof}

For the correspondence $[\mathscr{F}_{c,c_\infty}]$ in Lemma \ref{lm_preserve_orbifold_prod}, since $h\longmapsto(\xi-h)$, $\xi\longmapsto \xi$, we have $\ell^+\longmapsto -\ell^-$, $\gamma^+\mapsto \ell^-+\gamma^-$ and thus induces the map of Novikov variables
\[Q_1^+\longmapsto (Q_1^-)^{-1},\quad Q_2^+\longmapsto Q_1^-Q_2^-\quad\text{where}\quad Q_1^\pm=Q^{\ell^\pm},~Q_2^\pm=Q^{\gamma^\pm}.\]
What conditions on $c$ and $c_\infty$ are required to meet our purpose? We start with computing their small quantum cohomology rings. 

\subsection{Small quantum cohomology} \label{small-product}

To compute a small quantum product on $A^*_{\operatorname{orb}}(\mathcal X^\pm)$, we make use of the small $I$-function, namely, the extended $I$-function with $S=\emptyset$. To simplify the notation, we drop the superscript "$\pm$" unless necessary. Since the box elements $(d_1,d_3)$ with $d_3\in\Z$ have ages greater than 1, which implies that $S_o = \emptyset$ as mentioned in Theorem \ref{generator of quantum cohomology}, we conclude that $I(t_1h+t_2\xi)=J(t_1h+t_2\xi)$ and by Theorem \ref{generator of quantum cohomology}, the small quantum cohomology ring is generated by divisors. The $I$-function encodes only one-point invariants; due to the lack of divisor-generation property for orbifold cohomology, the reconstruction theorem \cite{lee2004reconstruction} fails to construct $n$-point invariants (cf. \cite{LLW1}). Instead, we compute two-point invariants, i.e. $h\star_{small}$ and $\xi{\star_{small}}$, by extracting the $z^0$-terms in $\hat h P_\mu(z^{-1}I)=h\star_{small}T_\mu+O(z^{-1})$ and $\hat \xi P_\mu(z^{-1}I)=\xi\star_{small}T_\mu+O(z^{-1})$ from Lemma \ref{birkhoff factorization} and \eqref{qde}. Then $n$-point invariants will be reconstructed from the WDVV equation and the divisor-generation property of small quantum product Theorem from \ref{generator of quantum cohomology}. Since no non-trivial mirror transformation is needed, it is sufficient to find $P_\mu$ for each $T_\mu$. 

Let $I^\pm$ be the non-extended $I$-function of $\mathcal{X}^\pm$. By the mirror theorem in \cite{coates2015mirror}, it can be written as
\begin{align*}
z^{-1}I^+(Q,z)& = e^{\frac{t}{z}}\sum_{(d_1,d_3)\in\Lambda_{\mathcal{X}^+}}(h)_{d_1}^{2p+q}\cdot (2(\xi-h))_{2(d_3-d_1)}^p\cdot (\xi-h)_{d_3-d_1}^q\cdot (\xi)_{d_3}\cdot Q_1^{d_1}Q_2^{d_3}y^{\nu^+(d_1,d_3)} \\
& \eqqcolon \sum_{(d_1,d_3)\in\Lambda_{\mathcal{X}^+}}I^+_{(d_1,d_3)} \\
z^{-1}I^-(Q,z)&=e^{\frac{t}{z}}\sum_{(d_2,d_3)\in\Lambda_{\mathcal{X}^-}}(2h)_{2d_2}^{p}\cdot (h)_{d_2}^q\cdot (\xi-h)_{d_3-d_2}^{2p+q}\cdot (\xi)_{d_3}\cdot Q_1^{d_1}Q_2^{d_3}y^{\nu^-(d_2,d_3)} \\
&\eqqcolon\sum_{(d_2,d_3)\in\Lambda_{\mathcal{X}^-}}I^-_{(d_2,d_3)},
\end{align*}
where $t=t_1h+t_2\xi$, $Q_1^\pm=Q^{\ell^\pm}$, $Q_2^\pm=Q^{\gamma^\pm}$. Here, we identify $t_i=\log Q_i$ and
\[(D)_r\coloneqq \dfrac{\prod_{\gen{m}=\gen{r},m\leq 0}(D+mz)}{\prod_{\gen{m}=\gen{r},m\leq r}(D+mz)}\]
is a factor of hypergeometric correction (Note that $\gen{r}$ is the fractional part of $r$). We introduce two hypergeometric differential operators
\begin{align*}
\mathcal{L}^+ & =\theta_{Q_1^+}^p-(-1)^q2^{p}Q_1^+\cdot(2\theta_{Q_1^+}+1)^p \\ 
\mathcal{L}^- & =(\theta_{Q_1^-}-\tfrac{1}{2})^p-(-1)^q2^{-2p}Q_1^-\cdot \theta_{Q_1^-}^p
\end{align*}
inspired from the Picard--Fuchs equation of $I^\pm$ corresponding to the extremal curve $\ell$
\begin{align*}
\square^+_\ell & =\widehat{h}^{2p+q}-Q_1^+\cdot 2^p(\widehat{\xi}-\widehat{h})^{p+q}(2(\widehat{\xi}-\widehat{h})-z)^{p} \\
\square^-_\ell & =2^p\widehat{h}^{p+q}(2\widehat{h}-z)^p-Q_1^-(\widehat{\xi}-\widehat{h})^{2p+q}
\end{align*}
respectively, where $\theta_Q\coloneqq Q\frac{d}{dQ}$. Note that $z^p\mathcal{L}^+\widehat{h}^{p+q} = \square^+_\ell|_{\widehat{\xi}=0}$ and $2^{2p}z^p\mathcal{L}^-\widehat{h}^{p+q} = \square^-_\ell|_{\widehat{\xi}=0}$.

For $d_1>d_3$, we can factor out the term
\[(\xi-h)^{|\vec{b}(d_3-d_1)|}y^{\nu^+(d_1,d_3)}\]
from $(2(\xi-h))_{2(d_3-d_1)}^p\cdot (\xi-h)_{d_3-d_1}^q\cdot y^{\nu^+(d_1,d_3)}$ (Recall: $\vec b(d)=\{b_j:{b_j d\in\Z}\}$). Hence, $I^+_{(d_1,d_3)}=0$ when $d_1>d_2,~d_3\notin\Z_{\geq 0}$, and $\xi I^+_{(d_1,d_3)}=0$ when $d_1>d_2,~d_3\in\Z_{\geq 0}$. Also, we have the similar result for $I^-_{(d_2,d_3)}$ when $d_2>d_3$.

Given $\alpha\in\left(\bigcup\frac{1}{a_i}\Z\right)\cap[0,1)$, $\beta\in\left(\bigcup\frac{1}{b_j}\Z\right)\cap[0,1)$ and $n\in\Z_{\geq 0}$, we define the formal power series
\[I_{\alpha,n}^+\coloneqq \sum_{\substack{\gen{d_1}= \alpha \\ d_1>n}}I^+_{(d_1,n)},\quad I_{\beta,n}^-\coloneqq \sum_{\substack{\gen{d_2}= \beta \\ d_2>n}}I^-_{(d_2,n)}.\]

Now, the recipe for computing $h\star_{small}T_\mu$ and $\xi\star_{small} T_\mu$ is summarized into three steps : 
\begin{enumerate}[I.]
    \item We compute $\deg_{z^{-1}}I_{(d_1,d_3)}^\pm$ to determine whether this term contributes,
    \item We make use of the idea of the proof for Lemma \ref{birkhoff factorization} to find $P_\mu$,
    \item We apply $\widehat{h}$ and $\widehat{\xi}$ to $P_\mu(z^{-1}I^\pm)$ and extract the $z^0$-term.
\end{enumerate}

\subsubsection{$h\star_{small}$ and $\xi\star_{small}$ on $\mathcal X^+$-side}
\label{subsection_quan_prod_X^+}
A simple degree calculation shows that
\[\deg_{z^{-1}}I_{(d_1,d_3)}^+=(2p+q+1)d_3+\operatorname{age}(\nu^+(d_1,d_3))=\begin{cases}
(2p+q+1)d_3 & \text{ if } d_3\in\Z \\
(2p+q+1)d_3+\frac{q+1}{2} & \text{ if } d_3\in\frac{1}{2}\Z\setminus\Z \\
\end{cases}\]
for $d_1\in\Z_{\geq 0}$ with $d_3\geq d_1$, and 
\[\deg_{z^{-1}}I^+_{0,n}=(2p+q+1)n+p+q\]
for all $n\in\N$. Consequently, we have the expansion
\begin{equation}\label{expansion_of_I^+}
    z^{-1}I^+=I^+_{(0,0)}+I^+_{0,0}+I^+_{(0,\frac{1}{2})}+I^+_{(0,1)}+I^+_{(1,1)}+O(z^{-(3p+2q+1)})=e^{\frac{t_1h+t_2\xi}{z}}+O(z^{-(p+q)}).
\end{equation}

If $U_1 = \{\xi h^i(\xi-h)^j\mid 0\leq i\leq 2p+q-1,\ 0\leq j\leq p+q-1\}$, $U_2=\{(\xi-h)^j\mid 0\leq j\leq 3p+2q-1\}$ and $U_\infty = \{h^i(\xi-h)^jy_\infty^+\mid 0\leq i\leq 2p+q-1,\ 0\leq j\leq p-1\}$, then $U_1\cup U_2\cup U_\infty$ forms a basis for $A_{\operatorname{orb}}^*(\mathcal{X}^+)$. For $T_\mu \in U_1$ with $i+j \leq p+q-3$ or $T_\mu \in U_2$ with $j\leq p+q -2$, by the fact that
\begin{equation}\label{easy_quanize_+}
(D_1\cdots D_k)^{\widehata}(J^+)|_{t_1h+t_2\xi}=\widehat{D}_1\cdots\widehat{D}_k(I^+)
\end{equation}
for any divisors $D_1,\ldots,D_k$ with $1\leq k\leq p+q-1$, we get that
\begin{equation}\label{easy_quan_prod_+}
D\star_{\operatorname{small}}T_\mu=DT_\mu
\end{equation}
for any divisor $D$. For $T_\mu$ in $U_1$ or $U_2$ having higher degree, to find $P_\mu$, we have to correct the contribution from
$$I^+_{0,0}+I^+_{(0,\tfrac{1}{2})}+I^+_{(0,1)}+I^+_{(1,1)}.$$
For $T_\mu \in U_\infty$, to find $P_\mu$, the contribution which has to be corrected may come from 
\[I^+_{(0,0)}+I^+_{0,0}+I^+_{(0,\tfrac{1}{2})}+I^+_{(0,1)}+I^+_{(1,1)}+I^+_{0,1}+I^+_{(0,\frac{3}{2})}+I^+_{(1,\frac{3}{2})}.\]

\begin{prop}\label{infinite_cohomo_quantize_+}
If $T_\mu = \xi h^a(\xi-h)^b\in U_1$, then $P_\mu = \widehat{\xi}\widehat{h}^a(\widehat{\xi}-\widehat{h})^b$
and 
\begin{align*}
h\star_{\operatorname{small}} T_\mu & =\xi h^{a+1}(\xi-h)^b+Q_1Q_2(\xi-h)^b\delta_{a,2p+q-1} \\
(\xi-h)\star_{\operatorname{small}} T_\mu & =\xi h^a(\xi-h)^{b+1}+\dfrac{Q_2^{1/2}}{2^p}h^ay^+_\infty\delta_{b,p+q-1},
\end{align*}
where $\delta_{m,n}$ is the Kronecker delta symbol. 
\end{prop}

\begin{proof}
A direct computation shows that 
\begin{align*}
\widehat{\xi}\widehat{h}^a(\widehat{\xi}-\widehat{h})^b(I^+_{0,0}) & =0 \\
\widehat{\xi}\widehat{h}^a(\widehat{\xi}-\widehat{h})^b(I^+_{(0,\frac{1}{2})}) & =e^{t/z}\cdot\dfrac{2^{q-b}Q_2^{1/2}h^ay^+_\infty}{(2(\xi-h)+z)^{p+q-b}}=\dfrac{Q_2^{1/2}h^ay^+_\infty}{2^{p-1}z}\cdot\delta_{b,p+q-1}+O(z^{-2}) \\
\widehat{\xi}\widehat{h}^a(\widehat{\xi}-\widehat{h})^b(I^+_{(0,1)}) & =e^{t/z}\cdot \dfrac{Q_2 h^a}{2^p(2(\xi-h)+z)^p(\xi-h+z)^{p+q-b}} =O(z^{-2}) \\
\widehat{\xi}\widehat{h}^a(\widehat{\xi}-\widehat{h})^b(I^+_{(1,1)}) & =e^{t/z}\cdot \dfrac{Q_1Q_2(\xi-h)^b}{(h+z)^{2p+q-a}}=\dfrac{Q_1Q_2(\xi-h)^b}{z}\cdot\delta_{a,2p+q-1}+O(z^{-2}).
\end{align*}
Since $(3p+2q+1)-(a+b+1)\geq 2$, we have
\[\widehat{\xi}\widehat{h}^a(\widehat{\xi}-\widehat{h})^b(z^{-1}I^+)=e^{t/z}\cdot \xi h^a(\xi-h)^b+\dfrac{Q_2^{1/2}h^ay^+_\infty}{2^{p-1}z}\cdot\delta_{b,p+q-1}+\dfrac{Q_1Q_2(\xi-h)^b}{z}\cdot\delta_{a,2p+q-1}+O(z^{-2}),\]
which allows us to conclude that $(\xi h^a(\xi-h)^b)^{\widehata}(J^+)|_{t_1h+t_2\xi}=\widehat{\xi}\widehat{h}^a(\widehat{\xi}-\widehat{h})^b(I^+)$. Finally, the quantum product is obtained by applying $\widehat{D}$ to the above equation and extracting the $z^0$-coefficient.
\end{proof}

\begin{prop}\label{Frob_basis_+}
If we write the following expansion
\[(\widehat{\xi}-\widehat{h})^{p+q}\left(I_{(0,0)}^++I_{0,0}^+\right)=\sum_{n=0}^{2p+q-1}g^+_n(Q_1^+)\dfrac{h^n}{z^n}\cdot(\xi-h)^{p+q},\]
then we have
\[\mathcal{L}^+(g_n^+)=\begin{cases}
0 & \text{if } n\leq p-1 \\
\frac{t_1^{n-p}}{(n-p)!} & \text{if } n\geq p
\end{cases}\]
and $g^+_0(0)=1$. In particular, $\{g_i^+\}_{i=0}^{p-1}$ forms the Frobenius basis of $\mathcal{L}^+$ and for $T_\mu = (\xi-h)^{p+q-1} \in U_2$, we get that
\begin{align*}
(\xi-h)\star_{\operatorname{small}}T_\mu&=g_0^+\cdot (\xi-h)^{p+q},\\
h\qsstar T_\mu &= \xi(\xi-h)^{p+q-1}-g_0^+\cdot(\xi-h)^{p+q}.
\end{align*}
\end{prop}

\begin{proof}
Notice that $\widehat{\xi}\cdot(\widehat{\xi}-\widehat{h})^{p+q}(I_{(0,0)}^++I_{0,0}^+)=0$, so we have
\begin{align*}
&\ z^p\mathcal{L}^+\left((\widehat{\xi}-\widehat{h})^{p+q}(I_{(0,0)}^++I_{0,0}^+)\right)\\ 
= &\ (-1)^{p+q}\cdot \square_\ell^+(I_{(0,0)}^++I_{0,0}^+)+(\widehat{h}^p(\widehat{\xi}-\widehat{h})^{p+q}-(-1)^{p+q}\widehat{h}^{2p+q})(I_{(0,0)}^+) \\
= &\ h^p(\xi-h)^{p+q}e^{t/z}.
\end{align*}
Since $\xi(\xi-h)^{p+q}=0$ in $A_{\operatorname{orb}}^*(\mathcal{X}^+)$, the first assertion follows from comparing the coefficient of $h^n(\xi-h)^{p+q}$. 

Notice that we have the expansion
\[(\widehat{\xi}-\widehat{h})^{p+q}\left(I_{(0,0)}^++I_{0,0}^+\right)=e^{t/z}\left(1+\sum_{n=0}^{2p+q-1}f_n^+(Q_1^+)\dfrac{h^n}{z^n}\right)\cdot(\xi-h)^{p+q}\]
for some formal power series $f_n^+(Q_1^+)\in Q_1^+\C[\![Q_1^+]\!]$, and the Frobenius-type relation
\[g_n^+(Q_1^+)=(1+f_0^+(Q_1^+))\cdot\dfrac{\log^n Q_1^+}{n!}+\sum_{k=0}^{n-1} f_{n-k}^+(Q_1^+)\cdot\dfrac{\log^k Q_1^+}{k!}.\]
In particular, $g_0^+(0)=1+f_0^+(0)=1$.

Finally, by \eqref{easy_quanize_+}, $(\xi-h)\star_{\operatorname{small}}(\xi-h)^{p+q-1}$ equals the $z^0$-term in $(\widehat{\xi}-\widehat{h})(\widehat{\xi}-\widehat{h})^{p+q-1}(z^{-1}I^+)$, which is equal to $ g_0^+\cdot (\xi-h)^{p+q}$.

From the equation \eqref{easy_quan_prod_+} and Proposition \ref{infinite_cohomo_quantize_+}, we conclude that
\begin{align*}
h\qsstar(\xi-h)^{p+q-1} &= \xi\qsstar(\xi-h)\qsstar(\xi-h)^{p+q-2}-(\xi-h)\qsstar(\xi-h)^{p+q-1} \\
& =(\xi-h)\qsstar\xi\qsstar(\xi-h)^{p+q-2}-g_0^+\cdot(\xi-h)^{p+q} \\
&= (\xi-h)\qsstar\xi(\xi-h)^{p+q-2}-g_0^+\cdot(\xi-h)^{p+q} \\
&= \xi(\xi-h)^{p+q-1}-g_0^+\cdot(\xi-h)^{p+q}.
\end{align*}
\end{proof}

To take care of other elements in $U_2$, it is sufficient to consider $T_\mu$ of the form $h^k(\xi-h)^{p+q}$ with $0\leq k\leq 2p+q-1$. 

\begin{definition}\label{def_Wron^+}
Given two ordered multisets $I$ and $J$ of the same length with elements in $\{0, 1, \ldots,2p+q-1\}$, we define the associated generalized Wronskian by
\[(W^+)^J_I=\det(\theta^j_{Q_1^+}(g_i^+))_{i\in I,j\in J}.\]
For the set $[k]=\{0,1,\ldots,k\}$, we define $W_k^+=(W^+)_{[k]}^{[k]}$. For convenience, we set $W_{-1}^+=1$.
\end{definition}

\begin{prop}\label{quan_prod_+_iden}
If $T_\mu = h^k(\xi-h)^{p+q}$ with $0\leq k\leq 2p+q-1$,  then
\begin{align*}
h \star_{\rm small} T_\mu=&\ \dfrac{W^+_{k-1}\cdot W^+_{k+1}}{(W^+_k)^2}h^{k+1}(\xi-h)^{p+q}+\dfrac{Q_2^{1/2}}{2^{p-1}}\sum_{n=0}^k\theta_{Q_1^+}\left(\dfrac{(W^+)_{[k-1]}^{[k]\setminus\{n\}}}{W_k^+}\right)\cdot 2^{k-n}h^n\xi^{k-n}y_\infty^+ \\
&+\dfrac{Q_2}{2^p}\sum_{n=0}^{k-p}(-1)^{n+k}\theta_{Q_1^+}\left(\dfrac{(W^+)_{[k-1]}^{[k]\setminus\{n\}}}{W_k^+}\right)S_{n,k}^+ \\
\xi \star_{\rm small} T_\mu=&\ \dfrac{Q_2^{1/2}}{2^p}\sum_{n=0}^k\dfrac{(W^+)_{[k-1]}^{[k]\setminus\{n\}}}{W_k^+}\cdot 2^{k-n}h^n\xi^{k-n}y_\infty^++\dfrac{Q_2}{2^p}\sum_{n=0}^{k-p}(-1)^{n+k}\dfrac{(W^+)_{[k-1]}^{[k]\setminus\{n\}}}{W_k^+}S_{n,k}^+
\end{align*} 
where 
\begin{equation}\label{S_n,k,-1^+}
S_{n,k}^+=(-1)^{k-n}\sum_{a+b+c=k-n-p}\binom{p+a-1}{a}\binom{b+c}{c}2^ah^{n+a+c}(\xi-h)^{b}.
\end{equation}
\end{prop}

\begin{proof}
For $0\leq n\leq k$, if we define the operator
\begin{equation}\label{def_G_n,k^+}
\mathcal{G}_{n,k}^+\coloneqq (z-2\widehat{\xi})^{k-n}\widehat{h}^n(\widehat{\xi}-\widehat{h})^{p+q},
\end{equation}
then
\begin{align*}
&\ \dfrac{1}{W^+_k}\sum_{n=0}^k(-1)^{n+k}(W^+)_{[k-1]}^{[k]\setminus\{n\}}\mathcal{G}_{n,k}^+\left(I_{(0,0)}^++I_{0,0}^+\right) \\
= &\ \dfrac{1}{W^+_k}\sum_{n=0}^k\sum_{m=0}^{2p+q-1}(-1)^{n+k}(W^+)_{[k-1]}^{[k]\setminus\{n\}}\theta_{Q_1^+}^n(g_m^+)\dfrac{h^m}{z^{m-k}}\cdot(\xi-h)^{p+q} \\
= &\ \dfrac{1}{W_k^+}\sum_{m=0}^{2p+q-1}(W^+)_{[k-1]\cup\{m\}}^{[k]}\dfrac{h^m}{z^{m-k}}\cdot (\xi-h)^{p+q} \\
= &\ h^k(\xi-h)^{p+q}+\dfrac{1}{z}\cdot\dfrac{(W^+)_{[k+1]\setminus\{k\}}^{[k]}}{W_k^+}\cdot h^{k+1}(\xi-h)^{p+q}+O(z^{-2}).
\end{align*}
For other terms, we have
\begin{align*}
\mathcal{G}_{n,k}^+(I_{(0,\frac{1}{2})}^+) & =e^{t/z}\cdot \dfrac{(-2)^{k-n}Q_2^{1/2}h^{n}\xi^{k-n}y^+_\infty}{2^{p-1}(2\xi+z)}=\dfrac{1}{z}\dfrac{Q_2^{1/2}}{2^{p-1}}\cdot (-2)^{k-n}h^n\xi^{k-n}y^+_\infty+O(z^{-2}) \\
\mathcal{G}_{n,k}^+(I_{(0,1)}^+) & = \dfrac{e^{t/z}Q_2}{2^p}\cdot\dfrac{(-2\xi-z)^{k-n}h^n}{(\xi+z)(2(\xi-h)+z)^{p}}=\dfrac{e^{t/z}Q_2}{2^p}\left(\sum_{m=-1}^{k-n-p-1} S_{n,k,m}^+z^m+O(z^{-2})\right) \\
\mathcal{G}^+_{k,n}(I^+_{(1,1)}) & =O(z^{2p+q+1-k}),
\end{align*}
for some $S_{n,k,m}^+\in A^{k-p-1-m}(\mathcal{X}^+)$. The explicit form of $S_{n,k,-1}^+$ is given in \eqref{S_n,k,-1^+} and is denoted simply by $S^+_{n,k}$. 

For $m\geq 0$, we have
\[k-p-1-m+m\leq p+q-2.\]
By \eqref{expansion_of_I^+} and the degree calculation, we have
\[z^m(S^+_{n,k,m})^{\widehata}(I^+)=e^{t/z}S_{n,k,m}^+z^m+O(z^{-2}).\]
Consequently, we obtain the relation
\begin{align*}
&\ \dfrac{1}{W_k^+}\left(\sum_{n=0}^k(-1)^{n+k}(W^+)_{[k-1]}^{[k]\setminus\{n\}}\left(\mathcal{G}_{n,k}-\dfrac{Q_2}{2^p}\sum_{m=0}^{k-n-p-1}z^m(S_{n,k,m}^+)^{\widehata}\right)\right)(I^+) \\
=&\ h^k(\xi-h)^{p+q}+\dfrac{1}{z}\cdot\dfrac{(W^+)_{[k+1]\setminus\{k\}}^{[k]}}{W_k^+}\cdot h^{k+1}(\xi-h)^{p+q} \\
& +\dfrac{1}{z}\cdot\dfrac{Q_2^{1/2}}{2^{p-1}}\sum_{n=0}^k\dfrac{(W^+)_{[k-1]}^{[k]\setminus\{n\}}}{W_k^+}\cdot 2^{k-n}h^n\xi^{k-n}y^+_\infty+\dfrac{1}{z}\cdot\dfrac{Q_2}{2^p}\sum_{n=0}^{k-p}(-1)^{n+k}\dfrac{(W^+)_{[k-1]}^{[k]\setminus\{n\}}}{W_k^+}S_{n,k}^++O(z^{-2}),
\end{align*}
and thus
$$P_\mu = \dfrac{1}{W_k^+}\sum_{n=0}^k(-1)^{n+k}(W^+)_{[k-1]}^{[k]\setminus\{n\}}\left(\mathcal{G}_{n,k}-\dfrac{Q_2}{2^p}\sum_{m=0}^{k-n-p-1}z^m(S_{n,k,m}^+)^{\widehata}\right).$$
To complete the proof of the quantum product formula stated in the proposition, it remains to prove that
\begin{equation}\label{Wron_der_relation_+}
\theta_{Q_1^+}\left(\dfrac{(W^+)_{[k+1]\setminus\{k\}}^{[k]}}{W_k^+}\right)=\dfrac{W_{k-1}^+\cdot W_{k+1}^+}{(W_k^+)^2}.
\end{equation}
Indeed, we define the operator $\mathcal{L}^+_k$ by
\begin{equation}\label{def_L_k^+}
\mathcal{L}^+_k(f)\coloneqq\det\left(\begin{matrix}
g_0^+ & \theta_{Q_1^+}(g_0^+) & \cdots & \theta^{k}_{Q_1^+}(g_0^+) \\
g_1^+ & \theta_{Q_1^+}(g_1^+) & \cdots & \theta^{k}_{Q_1^+}(g_1^+) \\
\vdots && \ddots & \\
g_{k-1}^+ & \theta_{Q_1^+}(g_{k-1}^+) & \cdots & \theta^{k}_{Q_1^+}(g_{k-1}^+) \\
f & \theta_{Q_1^+}(f) & \cdots & \theta^{k}_{Q_1^+}(f)
\end{matrix}\right).
\end{equation}
Observe that the following operators
\[\theta_{Q_1^+}\left(\dfrac{\mathcal{L}^+_k(\cdot)}{\mathcal{L}^+_k(g_k^+)}\right)=\dfrac{W_{k-1}^+}{W_k^+}\cdot\theta_{Q_1^+}^{k+1}+\operatorname{l.o.t.}\]
and
\[\mathcal{L}^+_{k+1}(\cdot)=W_k^+\cdot \theta_{Q_1^+}^{k+1}+\operatorname{l.o.t.},\]
which vanish on the $(k+1)$-dimensional $\mathbb{C}$-vector space spanned by $\{g_0^+,\ldots,g_k^+\}$. By the theory of linear ODEs and the comparison of their leading coefficients, we conclude that
\begin{equation}\label{L_k^+_relation}
\dfrac{W_{k-1}^+}{(W_k^+)^2}\mathcal{L}^+_{k+1}(\cdot)=\theta_{Q_1^+}\left(\dfrac{\mathcal{L}^+_k(\cdot)}{\mathcal{L}^+_k(g_k^+)}\right).
\end{equation}
Substituting $g_{k+1}^+$ into this identity yields \eqref{Wron_der_relation_+}.
\end{proof}

We now turn to the case of $T_\mu\in U_\infty$.
\begin{prop}\label{quan_isom^+_easy}
If $T_\mu = h^a(\xi-h)^by_\infty^+\in U_\infty$ with $0\leq a+b\leq 2p+q-2, 0\leq b\leq p-1$, then 
\begin{align*}
h\qsstar T_\mu & = h^{a+1}(\xi-h)^{b}y_\infty^+, \\
(\xi-h)\qsstar T_\mu & = h^a(\xi-h)^{b+1}y_\infty^++\dfrac{Q_2^{1/2}}{2^{p}}\cdot h^a\cdot\delta_{b,p-1}.
\end{align*}
\end{prop}

\begin{proof}
We claim that $P_\mu = 2^pQ_2^{-1/2}\widehat{h}^a(\widehat{\xi}-\widehat{h}-z/2)^{b}\widehat{\xi}(\widehat{\xi}-\widehat{h})^{p+q}$. Indeed, the action of $P_\mu$ on $z^{-1}I^+$ is given by 
\[I_{(0,0)}^++I_{0,0}^+\longmapsto 0,\quad I_{(0,\frac{1}{2})}^+\longmapsto e^{t/z}\cdot h^a(\xi-h)^{b}y^+_\infty,\]
\[I_{(0,1)}^+\longmapsto\dfrac{e^{t/z}Q_2^{1/2}h^a}{2^b(2(\xi-h)+z)^{p-b}}=\dfrac{1}{z}\cdot\dfrac{Q_2^{1/2}}{2^{p-1}}\cdot h^a\cdot\delta_{b,p-1}+O(z^{-2}),\]
\[I_{(1,1)}^+\longmapsto \dfrac{2^pe^{t/z}Q_1Q_2^{1/2}(\xi-h)^{p+q}(\xi-h-z/2)^b}{(h+z)^{2p+q-a}}=O(z^{-2}).\]
Thus, the quantum product is obtained by applying $\widehat{D}$ to the equation
\[P_\mu(z^{-1}I^+)=h^a(\xi-h)^by^+_\infty+\dfrac{1}{z}\left(t\cdot h^a(\xi-h)^by^+_\infty+\dfrac{Q_2^{1/2}}{2^{p-1}}\cdot h^a\cdot\delta_{b,p-1}\right)+O(z^{-2}).\]
\end{proof}
For $T_\mu = h^{p+q+a}(\xi-h)^by_\infty^+\in U_\infty$ with $0\leq a,b\leq p-1$, we have to expand the $I$-function to higher-order terms:
\[z^{-1}I^+=I^+_{(0,0)}+I^+_{0,0}+I^+_{(0,\tfrac{1}{2})}+I^+_{(0,1)}+I^+_{(1,1)}+I^+_{0,1}+I^+_{(0,\frac{3}{2})}+I^+_{(1,\frac{3}{2})}+O(z^{-(4p+2q+2)}).\]
Before proceeding to the main calculation, we first investigate the properties of the power series $I^+_{0,1}$. From now on, we simply denote $\theta_{Q_1^+}$ by $\theta$.

\begin{lm}\label{g_a,b^+}
If we carry out the expansion as follows:
\[\widehat{h}^{p+q}(\widehat{\xi}-\widehat{h})^{p+q}\left(I^+_{(1,1)}+I^+_{0,1}\right)=\left(\sum_{n=0}^{2p+q-1}g^+_{n,1}(Q_1^+)\dfrac{h^n}{z^n}\right)\cdot\dfrac{Q_2\cdot (\xi-h)^{p+q}}{z^{p+1}}.\]
and, for $0\leq a,b\leq p-1$, define the hypergeometric operator
\[\mathcal{L}^+_{a,b}\coloneqq \theta^{p-a}(\theta-1)^a-(-1)^q4^pQ_1^+\cdot\left(\theta-\tfrac{1}{2}\right)^{p-b}\left(\theta+\tfrac{1}{2}\right)^b,\]
then $g^+_{a,b,n}\coloneqq \theta^a(\theta-\frac{1}{2})^bg^+_{n,1}$ is the solution of $\mathcal{L}_{a,b}^+$ when $n\leq a-1$ and satisfies the relation
\[\theta^a\left((-1)^q\cdot 2^{-2p}g_n^+\right)=(\theta-\tfrac{1}{2})^{p-b}\left(g^+_{a,b,n}\right).\]
Furthermore, $g_{a,b,a}^++(-1)^{p+q-b
}\cdot 2^{p-b}$ is the solution of $\mathcal{L}_{a,b}^+$ and satisfies the relation
\[\theta^a((-1)^q\cdot 2^{-2p}g_a^+)=(\theta-\tfrac{1}{2})^{p-b}(g_{a,b,a}^++(-1)^{p+q-b}\cdot 2^{-p-b}).\]
\end{lm}

\begin{proof}
Notice that $(\widehat{\xi}-z)\cdot(\widehat{\xi}-\widehat{h})^{p+q}(I^+_{(1,1)}+I^+_{0,1})=0$. We have that
\begin{align*}
z^{p}\mathcal{L}^+_{0,0}\left(\widehat{h}^{p+q}(\widehat{\xi}-\widehat{h})^{p+q}\left(I^+_{(1,1)}+I^+_{0,1}\right)\right) & =(\widehat{\xi}-\widehat{h})^{p+q}\cdot\square_\ell^+\left(I^+_{(1,1)}+I^+_{0,1}\right) \\
& =(\widehat{\xi}-\widehat{h})^{p+q}\widehat{h}^{2p+q}\left(I_{(1,1)}^+\right) \\
& =e^{t/z}\cdot \dfrac{(\xi-h)^{p+q}}{z}\cdot Q_1^+Q_2^+.
\end{align*}
From the identity
\[(\theta-1)^a(\theta-\tfrac{1}{2})^b\mathcal{L}^+_{0,0}=\mathcal{L}^+_{a,b}\theta^a(\theta-\tfrac{1}{2})^b,\]
we get that
\[\mathcal{L}_{a,b}^+\left(\theta^a(\theta-\tfrac{1}{2})^b\widehat{h}^{p+q}(\widehat{\xi}-\widehat{h})^{p+q}\left(I^+_{(1,1)}+I^+_{0,1}\right)\right)=e^{t_1h/z}\cdot\left(\dfrac{h}{z}\right)^a\left(\dfrac{h}{z}+\dfrac{1}{2}\right)^b\cdot \dfrac{Q_1^+Q_2^+\cdot(\xi-h)^{p+q}}{z^{p+1}}.\]
By extracting the coefficient of $h^n(\xi-h)^{p+q}$, we obtain
\[\mathcal{L}_{a,b}^+\left(\theta_{Q_1^+}^a(\theta_{Q_1^+}-\tfrac{1}{2})^bg^+_{n,1}\right)=0\]
for $n\leq a-1$, and
\[\mathcal{L}_{a,b}^+\left(\theta_{Q_1^+}^a(\theta_{Q_1^+}-\tfrac{1}{2})^{p-1}g_{a,1}^++(-1)^{p+q-b}\cdot 2^{-p-b}\right)=2^{-b}Q_1^++\mathcal{L}_{a,b}^+((-1)^{p+q-b}\cdot 2^{-p-b})=0.\]

Recall that the Picard--Fuchs equation of $I^+$ corresponding to the fiber curve $\gamma^+$ is
\[\square_{\gamma^+}=2^p\widehat{\xi}(\widehat{\xi}-\widehat{h})^{p+q}(2(\widehat{\xi}-\widehat{h})-z)^p-Q_2^+.\]
The second assertion is deduced from the following computation:
\begin{align*}
&\ (\theta-\tfrac{1}{2})^{p-b}\left(\theta^a(\theta-\tfrac{1}{2})^b\widehat{h}^{p+q}(\widehat{\xi}-\widehat{h})^{p+q}\left(I^+_{(1,1)}+I^+_{0,1}\right)\right) \\
=&\ \dfrac{(-1)^p\cdot 2^{-2p}}{z^{a+p+1}}\cdot \widehat{h}^{p+q+a}\cdot 2^p\widehat{\xi}(\widehat{\xi}-\widehat{h})^{p+q}(2(\widehat{\xi}-\widehat{h})-z)^p\left(I^+_{(1,1)}+I^+_{0,1}\right) \\
=&\ \dfrac{(-1)^p\cdot 2^{-2p}}{z^{a+p+1}}\cdot \widehat{h}^{p+q+a}\left(Q_2^+\cdot I^+_{0,0}\right) \\
= &\ \dfrac{(-1)^q\cdot 2^{-2p}Q_2^+}{z^{p+1}}\cdot\left(\theta^a(\widehat{\xi}-\widehat{h})^{p+q}\left(I^+_{(0,0)}+I^+_{0,0}\right)-\theta^a(\widehat{\xi}-\widehat{h})^{p+q}(e^{t/z})\right),
\end{align*}
where we use the fact that
\[\theta^a(\widehat{\xi}-\widehat{h})^{p+q}(e^{t/z})=\dfrac{1}{z^a}\cdot h^a(\xi-h)^{p+q}\cdot e^{t/z}=O(z^{-a}).\]
\end{proof}

\begin{prop}\label{quan_isom^+_hard}
If, for any $0\leq a,b\leq p-1$ such that $k\coloneqq a+b-p\geq 0$, we define 
\[W_{a,b}^+=\det\left(\begin{matrix}
g_0^+ & \theta_{Q_1^+}(g_0^+) & \cdots & \theta^{k}_{Q_1^+}(g_0^+) & g_{a,b,0}^+ \\
g_1^+ & \theta_{Q_1^+}(g_1^+) & \cdots & \theta^{k}_{Q_1^+}(g_1^+) & g_{a,b,1}^+ \\
\vdots && \ddots & \vdots & \vdots \\
g_{k+1}^+ & \theta_{Q_1^+}(g_{k+1}^+) & \cdots & \theta^{k}_{Q_1^+}(g_{k+1}^+) & g_{a,b,k+1}^+
\end{matrix}\right)\]
and  let $T_\mu = h^{p+q+a}(\xi-h)^by^+_\infty$, then
\begin{align*}
h\qsstar T_\mu & =h^{p+q+a+1}(\xi-h)^by^+_\infty+(-1)^b\cdot (Q_2^+)^{1/2}\cdot 2^p\theta_{Q_1^+}\left(\dfrac{W^+_{a,b}}{W^+_k}\right)\cdot h^{k+1}(\xi-h)^{p+q} \\
& -(-1)^b\cdot 2Q_2\cdot\sum_{0\leq i,n\leq k}(-1)^{i+n}\theta_{Q_1^+}\left(g_{a,b,i}^+\dfrac{W^{[k]\setminus\{n\}}_{[k]\setminus\{i\}}}{W_k^+}\right)\cdot (-2)^{k-n}h^n\xi^{k-n}y^+_\infty \\
\xi\qsstar T_\mu & =\xi h^{p+q+a}(\xi-h)^by^+_\infty+\dfrac{(Q_2^+)^{1/2}}{2^p}\cdot h^{p+q+a}\cdot\delta_{b,p-1} \\
& +(-1)^b\cdot 2^{p-1}(Q_2^+)^{1/2}\cdot\dfrac{W^+_{a,b}}{W^+_k}\cdot h^{k+1}(\xi-h)^{p+q} \\
& -(-1)^b\cdot 2Q_2^+\sum_{0\leq i,n\leq k}(-1)^{i+n}g_{a,b,i}^+\dfrac{(W^+)^{[k]\setminus\{n\}}_{[k]\setminus\{i\}}}{W_k^+}\cdot (-2)^{k-n}h^n\xi^{k-n}y^+_\infty.
\end{align*}
\end{prop}

\begin{proof}
We claim that 
$$P_\mu = 2^pQ_2^{-1/2}\widehat{h}^{p+q+a}(\widehat{\xi}-\widehat{h}-z/2)^{b}\widehat{\xi}(\widehat{\xi}-\widehat{h})^{p+q}-(-1)^b\cdot 2^p\cdot Q_2^{1/2}\mathcal{H}^+_{a,b},$$
where 
\begin{equation}\label{def_H_a,b^-}
\mathcal{H}^+_{a,b}\coloneqq\dfrac{1}{W_k^+}\sum_{0\leq i,n\leq k}g_{a,b,i}^+(-1)^{i+n}W_{[k]\setminus\{i\}}^{[k]\setminus\{n\}}\mathcal{G}^+_{n,k},
\end{equation}
with $k=a+b-p$ (The definition of
$\mathcal{G}^+_{n,k}$ is in \eqref{def_G_n,k^+}). \\
Indeed,
 the action of $2^pQ_2^{-1/2}\widehat{h}^{p+q+a}(\widehat{\xi}-\widehat{h}-z/2)^{b}\widehat{\xi}(\widehat{\xi}-\widehat{h})^{p+q}$ on $z^{-1}I^+$ is given by
\begin{align*}
I_{(0,0)}^++I_{0,0}^+ & \longmapsto 0,\quad I_{(0,\frac{1}{2})}^+\longmapsto e^{t/z}\cdot h^{p+q+a}(\xi-h)^{b}y^+_\infty, \\
I^+_{(1,1)}+I^+_{0,1} & \longmapsto (-1)^b\cdot 2^{p}\left(\sum_{n=0}^{a+b-p+1} g^+_{a,b,n}(Q_1^+)\cdot z^{a+b-p-n}\cdot h^n\right)\cdot Q_2^{1/2}\cdot (\xi-h)^{p+q}+O(z^{-2}), \\
I_{(0,1)}^+ & \longmapsto\dfrac{e^{t/z}Q_2^{1/2}h^{p+q+a}}{2^b(2(\xi-h)+z)^{p-b}}=\dfrac{1}{z}\cdot\dfrac{Q_2^{1/2}}{2^{p-1}}\cdot h^{p+q+a}\cdot\delta_{b,p-1}+O(z^{-2}), \\
I^+_{(0,\frac{3}{2})} & \longmapsto\dfrac{e^{t/z}Q_2h^{p+q+a}}{2^{2p}(\xi-h+\tfrac{z}{2})^{p+q}(\xi-h+z)^{p-b}(\xi+\frac{z}{2})}=O(z^{-2}),\\
I^+_{(1,\frac{3}{2})} & \longmapsto\dfrac{e^{t/z}Q_1Q_2(\xi-h)^b}{(h+z)^{p-a}(\xi+\tfrac{z}{2})}=O(z^{-2}).
\end{align*}
The proof of the claim is completed by the following computation
\begin{align*}
\mathcal{H}^+_{a,b}\left(I_{(0,0)}^++I_{0,0}^+\right) & =\dfrac{1}{W_k^+}\sum_{0\leq i, n\leq k}\sum_{m=0}^{2p+q-1}g_{a,b,i}^+(-1)^{i+n}W_{[k]\setminus\{i\}}^{[k]\setminus\{n\}}\theta^n(g_m)\dfrac{h^m}{z^{m-k}}\cdot (\xi-h)^{p+q} \\
& =\dfrac{1}{W_k^+}\sum_{i=0}^k\sum_{m=0}^{2p+q-1}g_{a,b,i}^+W^{[k]}_{\{0,\ldots,i-1,m,i+1,\ldots,k\}}\dfrac{h^m}{z^{m-k}}\cdot(\xi-h)^{p+q} \\
& =\sum_{i=0}^k g_{a,b,i}^+\cdot z^{a+b-p-i}\cdot h^i(\xi-h)^{p+q} \\
& +\dfrac{1}{z}\cdot \dfrac{1}{W_k^+}\sum_{i=0}^{k}g^+_{a,b,i}(-1)^{k-i}W^{[k]}_{[k+1]\setminus\{i\}}\cdot h^{k+1}(\xi-h)^{p+q}+O(z^{-2}).
\end{align*}

Since $k\leq p-2$, $\mathcal{H}_{a,b}^+(I_{(0,1)}^+)=O(z^{-2})$. Together with the identity
\[g^+_{a,b,k+1}-\dfrac{1}{W_k^+}\sum_{i=0}^{k}g^+_{a,b,i}(-1)^{k-i}W^{[k]}_{[k+1]\setminus\{i\}}=\dfrac{W^+_{a,b}}{W^+_k},\]
we can determine the quantum product by
\begin{align*}
P_\mu(z^{-1}I^+) & =h^{p+q+a}(\xi-h)^by^+_\infty+\dfrac{t}{z}\cdot h^{p+q+a}(\xi-h)^by^+_\infty+\dfrac{1}{z}\cdot\dfrac{Q_2^{1/2}}{2^{p-1}}\cdot h^{p+q+a}\cdot\delta_{b,p-1}\\
& +\dfrac{(-1)^b\cdot 2^pQ_2^{1/2}}{z}\cdot\dfrac{W_{a,b}^+}{W_k^+}\cdot h^{k+1}(\xi-h)^{p+q} \\
& -\dfrac{(-1)^b\cdot 2Q_2}{z}\sum_{0\leq i,n\leq k}(-1)^{i+n}g_{a,b,i}^+\dfrac{(W^+)^{[k]\setminus\{n\}}_{[k]\setminus\{i\}}}{W_k^+}\cdot (-2)^{k-n}h^n\xi^{k-n}y^+_\infty +O(z^{-2}).
\end{align*}
\end{proof}

\subsubsection{$h\star_{small}$ and $\xi\star_{small}$ on $\mathcal X^-$-side}
\label{subsection_quan_prod_X^-}
A straightforward degree counting yields
\small
\begin{align*}
\deg_{z^{-1}}I_{(d_2,d_3)}^-=(2p+q+1)d_3+\operatorname{age}(\nu^-(d_2,d_3))=\begin{cases}
(2p+q+1)d_3 & d_2\in\Z,\ d_3\in\Z \\
(2p+q+1)d_3+p+q &d_2\in\frac{1}{2}\Z\setminus\Z,\ d_3\in\Z \\
(2p+q+1)d_3+\frac{q+1}{2} & d_2,d_3\in\frac{1}{2}\Z\setminus\Z
\end{cases}
\end{align*}\normalsize
for  $d_3\geq d_2$, and 
\[\deg_{z^{-1}}I^-_{\alpha,n}=(2p+q+1)n+(2p+q)\delta_{\alpha,0}+\operatorname{age}(\nu^-(\alpha,n))=\begin{cases}
(2p+q+1)n+2p+q &  \alpha=0 \\
(2p+q+1)n+p+q &  \alpha=\frac{1}{2}
\end{cases}\]
for all $n\in\N$. Consequently, we have the expansion
\small
\begin{align*}
z^{-1}I^- & =I^-_{(0,0)}+I^-_{\frac{1}{2},0}+I^-_{(\frac{1}{2},\frac{1}{2})}+I^-_{0,0}+I^-_{(0,1)}+I^-_{(1,1)}+O(z^{-(3p+2q+1)})=e^{\frac{t_1h+t_2\xi}{z}}+O(z^{-(p+q)}) \\
& =I^-_{(0,0)}+I^-_{\frac{1}{2},0}+I^-_{(\frac{1}{2},\frac{1}{2})}+I^-_{0,0}+I^-_{(0,1)}+I^-_{(1,1)}+I^-_{(\frac{1}{2},1)}+I^-_{\frac{1}{2},1}+I^-_{(\frac{1}{2},\frac{3}{2})}+I^-_{(\frac{3}{2},\frac{3}{2})}+O(z^{-4p-2q-1}).
\end{align*}\normalsize
The proof for $\mathcal{X}^-$ is entirely analogous  to that for $\mathcal{X}^+$. We  therefore omit the repetitive arguments and routine computation, and leave the details to the reader. 

\begin{lm}\label{easy_quan_-}
For any divisor $D_1,\ldots, D_k$ with $1\leq k\leq p+q-1$,
\[(D_1\cdots D_k)^{\widehata}(J^-)|_{t_1h+t_2\xi}=\widehat{D}_1\cdots\widehat{D}_k(I^-)\]
and 
\[D\star_{\operatorname{small}}T=DT\]
for any divisor $D$ and $T\in A^{\leq p+q-2}(\mathcal{X}^-)$.
\end{lm}

\begin{prop}\label{infinite_cohomo_quantize_-}
If $T_\mu = \xi h^a(\xi-h)^b$ with $0\leq a\leq p+q-1, 0\leq b\leq 2p+q-1$, then
$P_\mu=\widehat{\xi}\widehat{h}^a(\widehat{\xi}-\widehat{h})^b$
and 
\begin{align*}
h\star_{\operatorname{small}} T_\mu & =\xi h^{a+1}(\xi-h)^b+\dfrac{(Q_1^-Q_2^-)^{1/2}(\xi-h)^by^-_\infty}{2^p}\cdot\delta_{a,p+q-1}, \\
(\xi-h)\star_{\operatorname{small}} T_\mu & =\xi h^a(\xi-h)^{b+1}+Q_2^-h^a\cdot\delta_{b,2p+q-1},
\end{align*}
where $\delta_{m,n}$ is the Kronecker delta symbol. In particular, we get that
\[\xi\qsstar h^{p+q-1}=\xi\qsstar h\qsstar h^{p+q-2}=h\qsstar\xi\qsstar h^{p+q-2}=\xi h^{p+q-1}.\]
\end{prop}

\begin{proof}
By a direct computation, we get that
\small
\begin{align*}
&\widehat{\xi}\widehat{h}^a(\widehat{\xi}-\widehat{h})^b(I^-_{\frac{1}{2},0}+I^-_{0,0})  =0, \\
&\widehat{\xi}\widehat{h}^a(\widehat{\xi}-\widehat{h})^b(I^-_{(\frac{1}{2},\frac{1}{2})})  =e^{t/z}\cdot\dfrac{(Q_1^-Q_2^-)^{1/2}(\xi-h)^by^-_\infty}{2^p(h+\tfrac{z}{2})^{p+q-a}}=\dfrac{(Q_1^-Q_2^-)^{1/2}(\xi-h)^by^-_\infty}{2^{p-1}z}\cdot\delta_{a,p+q-1}+O(z^{-2}), \\
&\widehat{\xi}\widehat{h}^a(\widehat{\xi}-\widehat{h})^b(I^-_{(0,1)}) =e^{t/z}\cdot\dfrac{Q_2^-h^a}{(\xi-h+z)^{2p+q-b}}=\dfrac{Q_2^-h^a}{z}\cdot\delta_{b,2p+q-1}+O(z^{-2}), \\
&\widehat{\xi}\widehat{h}^a(\widehat{\xi}-\widehat{h})^b(I^-_{(1,1)})  =e^{t/z}\cdot\dfrac{Q_1^-Q_2^-(\xi-h)^b}{2^p(h+z)^{p+q-a}(2h+z)^p}=O(z^{-2}). \\
\end{align*}\normalsize
\end{proof}

\begin{prop}\label{Frob_basis_-}
If we carry out the expansion as follows:
\begin{align*}
\widehat{h}^{p+q}\left(I^-_{\frac{1}{2},0}\right) & =2^{-p}\cdot\sum_{n=0}^{p-1}g^-_n(Q_1^-)\dfrac{h^n}{z^n}\cdot y^-_{1/2} \\
\widehat{h}^{p+q}\left(I^-_{0,0}\right) & =\sum_{n=0}^{p-1}g^-_{n,2}(Q_1^-)\dfrac{h^n}{z^n}\cdot \dfrac{(\xi-h)^{2p+q}}{z^p},
\end{align*}
then we have
\[\mathcal{L}^-(g_n^-)=0,\quad \mathcal{L}^-(g_{n,2}^-)=2^{-2p}Q_1^-\cdot\frac{t_1^n}{n!}\]
for all $0\leq n\leq p-1$, and $(Q_1^-)^{-1/2}g^-_0(Q_1^-)|_{Q_1^-=0}=1$. In particular, $\{g_i^-\}_{i=0}^{p-1}$ forms the Frobenius basis of $\mathcal{L}^-$ and 
\[h\star_{\operatorname{small}}h^{p+q-1}=2^{-p} g_0^-\cdot y_{1/2}^-.\]
\end{prop}

\begin{proof}
The first assertion follows from the following relations
\[z^p\mathcal{L}^-\left(\widehat{h}^{p+q}(I^-_{\frac{1}{2},0})\right)=2^{-2p}\square_\ell^-(I^-_{\frac{1}{2},0})=0,\]
\[z^p\mathcal{L}^-\left(\widehat{h}(I^-_{0,0})\right)=2^{-2p}\square_\ell^-(I^-_{0,0})=\widehat{h}^{p+q}(\widehat{h}-\tfrac{z}{2})^pI^-_{(1,0)}=e^{t/z}\cdot 2^{-2p}Q_1^-\cdot(\xi-h)^{2p+q}\]
which relies on the fact that $\widehat{\xi}I_{\frac{1}{2},0}=\widehat{\xi}I_{0,0}=0$. By direct expansion, we will have that
\[g_0^-(Q_1^-)\in (Q_1^-)^{1/2}+(Q_1^-)^{1/2}\C[\![Q_1^-]\!].\]
Using the same argument as in Proposition \ref{Frob_basis_+}, we deduce that $\{g_i^-\}$ forms the Frobenius basis and
$h\star_{\operatorname{small}}h^{p+q-1}$ equals the $z^0$-term in $\widehat{h}\left((h^{p+q-1})^{\widehata}(z^{-1}I^-)\right)$, which turns out to be $2^q g_n^-(Q_1^-)\cdot y_{1/2}^-$.
\end{proof}

Combining the preceding propositions, we complete the case of $T_\mu\in A^*(\mathcal{X}^-)$. For $T_\mu$ in the exceptional twisted sector $\mathcal{X}^-_{(\frac{1}{2},0)}$, with the same setting in Definition \ref{def_Wron^+}, we define the generalized Wronskian as
\[(W^-)^J_I=\det(\theta^j_{Q_1^-}(g_i^-))_{i\in I,j\in J},\]
and $W_k^-:=(W^-)_{[k]}^{[k]}$ for simplicity.

\begin{prop}\label{quan_prod_-_exceptional}
If $T_\mu = h^ky^-_{1/2}$ with $0\leq k\leq p-1$, then
\begin{align*}
(\xi-h)\qsstar T_\mu & = -\dfrac{W_{k-1}^-\cdot W_{k+1}^-}{(W_k^-)^2}h^{k+1}y_{1/2}^--2^p\cdot\dfrac{W_{p-2}^-}{W_{p-1}^-}\cdot\dfrac{2^{-2p}Q_1^-}{1-(-1)^q2^{-2p}Q_1^-}(\xi-h)^{2p+q}\cdot\delta_{k,p-1}\notag\\
& -2(Q_1^-Q_2^-)^{1/2}\sum_{n=0}^k (-1)^n\theta_{Q_1^-}\left(\dfrac{(W^-)^{[k]\setminus\{n\}}_{[k-1]}}{W_k^-}\right)\cdot 2^{k-n}(\xi-h)^n\xi^{k-n}y_\infty^- ,\\
\xi\qsstar T_\mu & = (Q_1^-Q_2^-)^{1/2}\sum_{n=0}^k (-1)^n\dfrac{(W^-)^{[k]\setminus\{n\}}_{[k-1]}}{W_k^-}\cdot 2^{k-n}(\xi-h)^n\xi^{k-n}y_\infty^-.
\end{align*}
\end{prop}

\begin{proof}
For $0\leq n\leq k$, we define the operator
\[\mathcal{G}_{n,k}^-\coloneqq \widehat{h}^{p+q}(\widehat{h}-\widehat{\xi})^n(z-2\widehat{\xi})^{k-n}.\]
By the same calculation as in Proposition \ref{quan_prod_+_iden}, assuming that $\widehat{\xi}f=0$, we obtain the formula
\[\dfrac{1}{W_k^-}\sum_{n=0}^k(-1)^{n+k}(W^-)_{[k-1]}^{[k]\setminus\{n\}}\mathcal{G}^-_{n,k}(f)=\dfrac{z^k}{W_k^-}\mathcal{L}_k^-(\widehat{h}^{p+q}f),\]
where $\mathcal{L}_k^-$ is defined analogously to \eqref{def_L_k^+} by replacing the superscript $+$ with $-$. In particular, 
\[\dfrac{2^p}{W^-_k}\sum_{n=0}^k(-1)^{n+k}(W^-)_{[k-1]}^{[k]\setminus\{n\}}\mathcal{G}^-_{n,k}(I_{\frac{1}{2},0}^-)=h^ky^-_{1/2}+\dfrac{1}{z}\cdot\dfrac{(W^-)_{[k+1]\setminus\{k\}}^{[k]}}{W_k^-}\cdot h^{k+1}y^-_{1/2}+O(z^{-2})\]
and
\[\dfrac{2^p}{W^-_k}\sum_{n=0}^k(-1)^{n+k}(W^-)_{[k-1]}^{[k]\setminus\{n\}}\mathcal{G}^-_{n,k}(I_{0,0}^-)=\dfrac{2^p}{z}\cdot \dfrac{\mathcal{L}_k^-(g_{0,2}^-)}{W_k^-}\cdot(\xi-h)^{2p+q}\cdot\delta_{k,p-1}+O(z^{-2}).\]
For other terms, we have
\begin{align*}
\mathcal{G}^-_{n,k}(I_{(0,0)}^-) & =0, \\
\mathcal{G}^-_{n,k}(I_{(\frac{1}{2},\frac{1}{2})}^-) & =e^{t/z}\cdot \dfrac{(-1)^k\cdot 2^{k-n}\cdot (Q_1^-Q_2^-)^{1/2}(\xi-h)^n\xi^{k-n}y_\infty^-}{2^p(\xi+\frac{z}{2})} \\
& =\dfrac{(-1)^k\cdot 2^{k+1-n}\cdot (Q_1^-Q_2^-)^{1/2}}{2^pz}\cdot (\xi-h)^n\xi^{k-n}y_\infty^-+O(z^{-2}), \\
\mathcal{G}^-_{k,n}(I_{(0,1)}^-+I^-_{(1,1)}) & =O(z^{k-p-1}).
\end{align*}
Therefore, we can obtain the expansion
\begin{align*}
P_\mu(z^{-1}I^-)=h^ky_{1/2}^- & +\dfrac{1}{z}\left(\dfrac{(W^-)^{[k]}_{[k+1]\setminus\{k\}}}{W_k^-}\cdot h^{k+1}y_{1/2}^-+2^p\cdot \dfrac{\mathcal{L}_{p-1}^-(g_{0,2}^-)}{W_{p-1}^-}\cdot(\xi-h)^{2p+q}\cdot\delta_{k,p-1}\right) \\
& +2\cdot \dfrac{(Q_1^-Q_2^-)^{1/2}}{z}\sum_{n=0}^k (-1)^n\dfrac{(W^-)^{[k]\setminus\{n\}}_{[k-1]}}{W_k^-}\cdot 2^{k-n}(\xi-h)^n\xi^{k-n}y_\infty^-+O(z^{-2}).
\end{align*}
To get the quantum product formula, we invoke the equation \eqref{L_k^+_relation} once again. By comparing the coefficients of $\mathcal{L}_p^-$ and $\mathcal{L}^-$ and applying Proposition \ref{Frob_basis_-}, we show that
\begin{align*}
\theta\left(\dfrac{\mathcal{L}_{p-1}^-(g_{0,1}^-)}{W_{p-1}^-}\right)=\dfrac{W_{p-2}^-}{(W_{p-1}^-)^2}\mathcal{L}^-_p(g_{0,2}^-) & =\dfrac{W_{p-2}^-}{(W_{p-1}^-)^2}\cdot\dfrac{W^-_{p-1}}{1-(-1)^q2^{-2p}Q_1^-} \mathcal{L}^-(g_{0,2}^-) \\
& =\dfrac{W_{p-2}^-}{W_{p-1}^-}\cdot\dfrac{2^{-2p}Q_1^-}{1-(-1)^q2^{-2p}Q_1^-}.
\end{align*}
\end{proof}

Finally, we have to treat the case of $T_\mu \in A^*(\mathcal{X}^-_{(\frac{1}{2},\frac{1}{2})})$.

\begin{lm}\label{g_a,b^-}
If we expand the expression as follows:
\[(\widehat{h}-z)^{p+q}\widehat{h}^{p+q}\left(I^-_{(\frac{1}{2},1)}+I^-_{\frac{1}{2},1}\right)=\left(\sum_{n=0}^{p-1}g^-_{n,1}(Q_1^-)\dfrac{h^n}{z^n}\right)\cdot\dfrac{Q_1^-Q^-_2\cdot y^-_{1/2}}{z^{p+1}}\]
and, for $0\leq a,b\leq p-1$, we define the hypergeometric operator
\[\mathcal{L}^-_{a,b}\coloneqq (\theta+\tfrac{1}{2})^{p-a}(\theta-\tfrac{1}{2})^a-(-1)^q2^{-2p}Q_1^-\cdot\theta^{p-b}(\theta+1)^b,\]
then $g^-_{a,b,n}\coloneqq (\theta+\tfrac{1}{2})^a\theta^b g^-_{n,1}$ is the solution of $\mathcal{L}_{a,b}^-$ for all $0\leq n\leq p-1$. Moreover,
\[\theta^b\left(2^{-3p}g^-_n\right)=(\theta+\tfrac{1}{2})^{p-a}(g_{a,b,n}^-).\]
\end{lm}

\begin{proof}
Notice that $(\widehat{\xi}-z)(I^-_{(\frac{1}{2},1)}+I^-_{\frac{1}{2},1})=0$. We get that 
\small
\[z^p\mathcal{L}^-_{0,0}(Q_1^-)^{-1}\left((\widehat{h}-z)^{p+q}\widehat{h}^{p+q}(I^-_{(\frac{1}{2},1)}+I^-_{\frac{1}{2},1})\right)=2^{-2p}(Q_1^-)^{-1}(\widehat{h}-z)^{p+q}\square_\ell^-(I^-_{(\frac{1}{2},1)}+I^-_{\frac{1}{2},1})=0.\]\normalsize
Thus, $\mathcal{L}_{a,b}^-(g_{a,b,n}^-)=0$ follows from
\[(\theta-\tfrac{1}{2})^a\theta^b\mathcal{L}_{0,0}^-=\mathcal{L}_{a,b}^-(\theta+\tfrac{1}{2})^a\theta^b.\]

Recall that the Picard--Fuchs equation of $I^-$ corresponding to $\gamma^-+\ell^-$ is
\[\square^-_{\gamma^-+\ell^-}=2^{2p}\widehat{\xi}\widehat{h}^{p+q}(\widehat{h}-\tfrac{z}{2})^p-Q_1^-Q_2^-.\]
The second assertion follows from the following result:
\begin{align*}
&\ (\theta-\tfrac{1}{2})^{p-a}\left((\theta-\tfrac{1}{2})^a(\theta-1)^b(\widehat{h}-z)^{p+q}\widehat{h}^{p+q}(I^-_{(\frac{1}{2},1)}+I^-_{\frac{1}{2},1})\right) \\
=&\ \dfrac{1}{z^{b+p+1}}(\widehat{h}-z)^{p+q+b}\widehat{\xi}\widehat{h}^{p+q}(\widehat{h}-\tfrac{z}{2})^p(I^-_{(\frac{1}{2},1)}+I^-_{\frac{1}{2},1}) \\
=&\ Q_1^-Q_2^-\cdot \dfrac{1}{2^{2p}z^{p+1}}\theta^b\widehat{h}^{p+q}(I_{\frac{1}{2},0}^-).
\end{align*}
\end{proof}

\begin{prop}\label{quan_isom^-_easy}
If $T_\mu = h^a(\xi-h)^by^-_\infty$ with $0\leq a\leq p-1, 0\leq b\leq p+q-1$, then
\begin{align*}
h\qsstar T_\mu & = h^{a+1}(\xi-h)^by_\infty^-+\dfrac{(Q_1^-Q_2^-)^{1/2}}{2^p}(\xi-h)^b\cdot\delta_{a,p-1}, \\
(\xi-h)\qsstar T_\mu & = h^a(\xi-h)^{b+1}y_\infty^-.
\end{align*}
\end{prop}

\begin{proof}
Since the action of $2^p(Q_1^-Q_2^-)^{-1/2}(\widehat{h}-\tfrac{z}{2})^a(\widehat{\xi}-\widehat{h})^b\widehat{\xi}\widehat{h}^{p+q}$ on $z^{-1}I^-$ is given by 
\[I^-_{(0,0)}+I^-_{\frac{1}{2},0}+I^-_{0,0}+I^-_{(0,1)}\longmapsto 0,\quad I^-_{(\frac{1}{2},\frac{1}{2})}\longmapsto h^a(\xi-h)^b y^-_\infty,\]
\begin{align*}
I^-_{(1,1)} & \longmapsto e^{t/z}\cdot \dfrac{(Q_1^-Q_2^-)^{1/2}\cdot (\xi-h)^b}{2^p(h+\tfrac{z}{2})^{p-a}}=\dfrac{(Q_1^-Q_2^-)^{1/2}}{2^{p-1}z}(\xi-h)^b\cdot\delta_{a,p-1}+O(z^{-2}), \\
I^-_{(\frac{1}{2},1)}+I^-_{\frac{1}{2},1} & \longmapsto O(z^{a+b-2p-q}),
\end{align*}
we get that
$P_\mu = 2^p(Q_1^-Q_2^-)^{-1/2}(\widehat{h}-\tfrac{z}{2})^a(\widehat{\xi}-\widehat{h})^b\widehat{\xi}\widehat{h}^{p+q}$.
\end{proof}

\begin{prop}\label{quan_isom^-_hard}
If, for any $0\leq a,b\leq p-1$ such that $k\coloneqq a+b-p\geq -1$, we define 
\[W_{a,b}^-=\det\left(\begin{matrix}
g_0^- & \theta_{Q_1^-}(g_0^-) & \cdots & \theta^{k}_{Q_1^-}(g_0^-) & g_{a,b,0}^- \\
g_1^- & \theta_{Q_1^-}(g_1^-) & \cdots & \theta^{k}_{Q_1^-}(g_1^-) & g_{a,b,1}^- \\
\vdots && \ddots & \vdots & \vdots \\
g_{k+1}^- & \theta_{Q_1^+}(g_{k+1}^-) & \cdots & \theta^{k}_{Q_1^-}(g_{k+1}^-) & g_{a,b,k+1}^-
\end{matrix}\right)\]
and let $T_\mu = h^a(\xi-h)^{p+q+b}y^-_\infty $,
then
\small
\begin{align*}
&\ (\xi-h)\qsstar T_\mu \\
=&\ h^a(\xi-h)^{p+q+b+1}y_\infty^--(-1)^{p+q+b}2^p(Q_1^-Q_2^-)^{1/2}\theta_{Q_1^-}\left(\dfrac{W_{a,b}^-}{W_k^-}\right)h^{k+1}y_{1/2}^- \\
&+(-1)^{a+q}2^pQ_1^-Q_2^-\sum_{0\leq i,n\leq k}(-1)^{i+n}\theta_{Q_1^-}\left(g_{a,b,i}^-\dfrac{(W^-)_{[k]\setminus\{i\}}^{[k]\setminus\{n\}}}{W_k^-}\right)2^{k+1-n}(\xi-h)^n\xi^{k-n}y_\infty^- ,\\
&\ \xi\qsstar T_\mu \\
=&\ \xi h^a(\xi-h)^{p+q+b}y_\infty^-+\dfrac{(Q_1^-Q_2^-)^{1/2}}{2^p}(\xi-h)^{p+q+b}\delta_{a,p-1}+(-1)^{p+q+b}2^{p-1}(Q_1^-Q_2^-)^{1/2}\dfrac{W_{a,b}^-}{W_k^-}h^{k+1}y_{1/2}^- \\
&-(-1)^{a+q} 2^pQ_1^-Q_2^-\sum_{0\leq i,n\leq k}(-1)^{i+n}g_{a,b,i}^-\dfrac{(W^-)_{[k]\setminus\{i\}}^{[k]\setminus\{n\}}}{W_k^-}2^{k+1-n}(\xi-h)^n\xi^{k-n}y_\infty^-.
\end{align*}\normalsize
\end{prop}

\begin{proof}
The action of $2^p(Q_1^-Q_2^-)^{-1/2}(\widehat{h}-\tfrac{z}{2})^a(\widehat{\xi}-\widehat{h})^{p+q+b}\widehat{\xi}\widehat{h}^{p+q}$ on $z^{-1}I^-$ is given by 
\small
\[I^-_{(0,0)}+I^-_{\frac{1}{2},0}+I^-_{0,0}+I^-_{(0,1)}+I^-_{(\frac{1}{2},\frac{1}{2})}+I^-_{(1,1)}\mapsto h^a(\xi-h)^{p+q+b} y^-_\infty+\dfrac{(Q_1^-Q_2^-)^{1/2}}{2^{p-1}z}(\xi-h)^{p+q+b}\cdot\delta_{a,p-1}+O(z^{-2})\]
\begin{align*}
I^-_{(\frac{1}{2},1)}+I^-_{\frac{1}{2},1} & \mapsto (-1)^{p+q+b}\cdot 2^p\left(\sum_{n=0}^{p-1}g^-_{a,b,n}\cdot z^{a+b-p-n}\cdot h^n\right)\cdot(Q_1^-Q_2^-)^{1/2}\cdot y^-_{1/2} \\
I_{(\frac{1}{2},\frac{3}{2})}^-+I_{(\frac{3}{2},\frac{3}{2})}^- & \mapsto e^{t/z}\cdot\dfrac{Q_2^-\cdot h^a}{2^{2p}(\xi-h+z)^{p-b}(\xi+\tfrac{z}{2})}+e^{t/z}\cdot\dfrac{Q_1^-Q_2^-\cdot(\xi-h)^{p+q+b}}{2^{4p}(h+\tfrac{z}{2})^{p+q}(h+z)^{p-a}(\xi+\tfrac{z}{2})}=O(z^{-2}).
\end{align*}\normalsize
By the same argument as in Proposition \ref{quan_isom^+_hard}, we conclude that 
\[P_\mu=2^p(Q_1^-Q_2^-)^{-1/2}(\widehat{h}-\tfrac{z}{2})^a(\widehat{\xi}-\widehat{h})^{p+q+b}\widehat{\xi}\widehat{h}^{p+q}-(-1)^{p+q+b}\cdot 2^{2p}(Q_1^-Q_2^-)^{1/2}\mathcal{H}_{a,b}^-,\]
where $\mathcal{H}_{a,b}^-$ is defined analogously to \eqref{def_H_a,b^-} by replacing the superscript $+$ with $-$. 

For $k\leq p-2$, we have the expansion
\begin{align*}
&P_\mu(z^{-1}I^-) \\
=&\ h^a(\xi-h)^{p+q+b}y_\infty^-+\dfrac{t}{z}\cdot h^a(\xi-h)^{p+q+b}y_\infty^-+\dfrac{(Q_1^-Q_2^-)^{1/2}}{2^{p-1}z}(\xi-h)^{p+q+b}\cdot\delta_{a,p-1} \\
+&\ (-1)^{p+q+b}\cdot\dfrac{2^p(Q_1^-Q_2^-)^{1/2}}{z}\cdot\dfrac{W^-_{a,b}}{W_k^-}h^{k+1}y_{1/2}^- \\
-&\ (-1)^{a+q}\cdot\dfrac{2^pQ_1^-Q_2^-}{z}\sum_{0\leq i,n\leq k}(-1)^{i+n}g_{a,b,i}^-\dfrac{(W^-)_{[k]\setminus\{i\}}^{[k]\setminus\{n\}}}{W_k^-}2^{k+1-n}\cdot (\xi-h)^n\xi^{k-n}y_\infty^-+O(z^{-2}).
\end{align*}
\end{proof}

\subsection{Regularization map and Quantum invariance}
We observe that the functions appearing in the small quantum products on the two sides can not be matched by classical analytic continuation alone. Our strategy for resolving this issue is to construct a {\it regularization map}. 
\subsubsection{Picard--Vessiot extensions of $\mathcal{L}^+$ and $\mathcal{L}^-$}
Let $\mathcal L$ be a differential operator over $\C(Q)$ with the solution space $V_\mathcal{L}\coloneqq \operatorname{Sol}(\mathcal L)$. The \textit{Picard--Vessiot extension} $F_\mathcal{L}$ is the differential field extension of $(\C(Q),\theta:=Q\frac{d}{dQ})$ defined by adjoining solutions in $V_\mathcal L$ and their derivatives to the field $\C(Q)$. The \textit{differential Galois group} $G_\mathcal{L}$ of $F_\mathcal{L}$ over $\C(Q)$ is defined as the group of field automorphisms $F_\mathcal{L}\to F_\mathcal{L}$ that preserve $\C(Q)$ and the action of $\theta$. Since $G_\mathcal L$ acts faithfully on $V_\mathcal{L}$,  it can be naturally viewed as an (algebraic) subgroup of $\operatorname{GL}(V_\mathcal{L})$. That is,
\[G_\mathcal{L} = \left\{\sigma:F_\mathcal{L}\to F_\mathcal{L}\middle|\ \sigma|_{\C(Q)}=\operatorname{id}_{\C(Q)},\ \sigma\theta=\theta\sigma\right\}\subseteq\operatorname{GL}(V_\mathcal{L}).\]
Fundamental theorems of classical Galois theory carried over seamlessly to the differential setting (cf. \cite{van2003galois}). In particular, the fundamental Galois correspondence holds.

\begin{thm}\label{div_3pt_in_F}
Let $F^+$ denote the Picard--Vessiot extension of $\mathcal{L}^+$ over $\C(Q_1^+)$. For any divisor $D$, and for two orbifold cohomology classes $T_1$, $T_2$, the three-point function $\gen{D,T_1,T_2}^{\mathcal{X}^+}$ is contained in $F^+((Q_2^+)^{1/2})$.
Similarly, let $F^-$ denote the Picard--Vessiot extension of $\mathcal{L}^-$ over $\C(Q_1^-)$. For any divisor $D$, and for two orbifold cohomology classes $T_1$, $T_2$, the three-point function $\gen{D,T_1,T_2}^{\mathcal{X}^-}$ is contained in $F^-((Q_1^+Q_2^+)^{1/2})$.
\end{thm}

It suffices to show that for any divisor $D$ and orbifold cohomology class $T$, all coefficients of the quantum product $D \star_{\rm small} T$ lie in the differential fields specified in the Theorem \ref{div_3pt_in_F}.

\begin{lm}\label{Wron_k^+}
For $p-1\leq k\leq 2p+q-1$, 
\[W_k^+=(1-(-1)^q2^{2p}Q_1^+)^{-k-1+p/2}.\]
In particular,
\[W_k^+=\dfrac{W^+_{p-1}}{(1-(-1)^q2^{2p}Q_1^+)^{k+1-p}}\in F^+.\]
\end{lm}

\begin{proof}
By Proposition \ref{Frob_basis_+}, $\{g_0,\ldots,g_k\}$ is a Frobenius basis of the ordinary differential equation $\theta^{k+1-p}\mathcal{L}^+=0$. Therefore,
\[\theta W^+_k=\dfrac{(-1)^q2^{2p}(k+1-p/2)Q_1^+}{1-(-1)^q2^{2p}Q_1^+}\cdot W_k^+\]
implies the result by comparing coefficients.
\end{proof}

\begin{lm}\label{gen_Wron^+_in_F^+}
For $p\leq k\leq 2p+q-1$, we have the algebraic relation
\begin{equation}\label{relation_of_gen_Wron^+}
\sum_{n=0}^k(-1)^{n+k}\cdot\dfrac{(W^+)_{[k-1]}^{[k]\setminus\{n\}}}{W_k^+}\cdot x_1^nx_2^{k-n}=x_1^k-(-1)^q2^{2p}Q_1^+(x_1+x_2)^{k-p}(x_1+\tfrac{x_2}{2})^p
\end{equation}
as polynomials in $x_1$, $x_2$.
\end{lm}

\begin{proof}
Observe that $\mathcal{L}_k^+$ and $\theta^{k-p}\mathcal{L}^+$ are two operators of degree $k$ that vanish on the $k$-dimensional $\C$-vector space spanned by $\{g_0^+,\ldots,g_{k-1}^+\}$, and thus
\[\theta^{k-p}\mathcal{L}^+(\cdot)=\dfrac{1-(-1)^q2^{2p}Q_1^+}{W_{k-1}^+}\mathcal{L}_k^+(\cdot)=\dfrac{1}{W_k^+}\mathcal{L}_k^+(\cdot)=\sum_{n=0}^k(-1)^{n+k}\cdot\dfrac{(W^+)^{[k]\setminus\{n\}}_{[k-1]}}{W_k^+}\theta^n\]
by comparing their leading coefficients and Lemma \ref{Wron_k^+}. Therefore, the result follows from replacing $\theta$ with $x_1/x_2$.
\end{proof}

Now, we can simplify the formula for the quantum product in Proposition \ref{quan_prod_+_iden} as follows.
\begin{cl}\label{quan_prod_+_iden_simplify}
For $p\leq k\leq 2p+q-1$, we have
\begin{align*}
h\qsstar h^k(\xi-h)^{p+q}=&\ h^{k+1}(\xi-h)^{p+q}+(-1)^{k+q+1}Q_1^+Q_2^+(\xi-h)^{k-p} \\
\xi\qsstar h^k(\xi-h)^{p+q}=&\ \dfrac{(Q_2^+)^{1/2}}{2^p}\cdot h^ky^+_\infty+(-1)^{k+q+1}Q_1^+Q_2^+(\xi-h)^{k-p}
\end{align*} 
on $\mathcal{X}^+$.
\end{cl}

\begin{proof}
By Lemma \ref{Wron_k^+}, we have
\[\frac{W_{k-1}^+\cdot W_{k+1}^+}{(W_k^+)^2}=1.\]
By Lemma \ref{gen_Wron^+_in_F^+}, we have
\[\sum_{n=0}^k\dfrac{(W^+)_{[k-1]}^{[k]\setminus\{n\}}}{W_k^+}\cdot 2^{k-n}h^n\xi^{k-n}y_\infty^+=h^ky_\infty^+-(-1)^q2^{2p}Q_1^+(h-2\xi)(h-\xi)^py_\infty^+=h^ky_\infty^+\]
and
\begin{align*}
&\ \sum_{n=0}^k(-1)^{n+k}\dfrac{(W^+)_{[k-1]}^{[k]\setminus\{n\}}}{W_k^+}\mathcal{G}^+_{n,k}(I^+_{(0,1)}) \\
=&\ \dfrac{e^{t/z}\cdot Q_2^+}{2^p(2(\xi-h)+z)^p(\xi+z)}\left(h^k-(-1)^q2^{2p}Q_1^+(h-2\xi-z)^{k-p}(h-\xi-\tfrac{z}{2})^p\right) \\
=&\ e^{t/z}\cdot (-1)^{k+q+1}Q_1^+Q_2^+\cdot\dfrac{(\xi-h+\xi+z)^{k-p}}{\xi+z}+O(z^{-2}) \\
=&\ e^{t/z}\cdot (-1)^{k+q+1}Q_1^+Q_2^+\left(\sum_{m=1}^{k-p}\binom{k-p}{m}(\xi-h)^{k-p-m}(\xi+z)^{m-1}+\dfrac{(\xi-h)^{k-p}}{\xi+z}\right)+O(z^{-2}).
\end{align*}
By the same proof in Proposition \ref{quan_prod_+_iden}, we conclude that for $T_\mu = h^k(\xi-h)^{p+q}$,
\begin{align*}
P_\mu(z^{-1}I^+)=h^k(\xi-h)^{p+q} & +\dfrac{1}{z}\cdot\dfrac{(W^+)_{[k+1]\setminus\{k\}}^{[k]}}{W_k^+}\cdot h^{k+1}(\xi-h)^{p+q}+\dfrac{1}{z}\cdot\dfrac{Q_2^{1/2}}{2^{p-1}}h^ky^+_\infty \\
& +\dfrac{1}{z}\cdot (-1)^{k+q+1}Q_1^+Q_2^+(\xi-h)^{k-p}+O(z^{-2}).
\end{align*}
\end{proof}

From Section \ref{subsection_quan_prod_X^+} and combining with Lemmas \ref{Wron_k^+} and Corollary \ref{quan_prod_+_iden_simplify}, we conclude that
\begin{equation}\label{3pt_in_F^+_div_coho_orbcoho}
\gen{D,T_1,T_2}\in F^+((Q_2^+)^{1/2})
\end{equation}
for any $D\in A^1(\mathcal{X}^+)$, $T_1\in A^*(\mathcal{X}^+)$, and $T_2\in A^*_{\operatorname{orb}}(\mathcal{X}^+)$.

Next, we have to consider the case where $T_1\in A^*(\mathcal{X}_{(0,\frac{1}{2})})$. Recall from Lemma \ref{g_a,b^+} that $\mathcal{L}_{a,b}^+(g_{a,b,n})=0$ for $n\leq a-1$, and we have the mapping
\[\begin{tikzcd}[row sep=0pt, column sep=45pt]
V_{\mathcal{L}^+} \arrow[r,"\theta^a"] & V_{\mathcal{L}^+_{a,0}} & V_{\mathcal{L}_{a,b}^+} \arrow[l,"(\theta-\frac{1}{2})^{p-b}"']  \\
(-1)^q\cdot 2^{-2p}g_n^+ \arrow[r,mapsto] & \theta^a((-1)^q\cdot 2^{-2p}g_n^+)=(\theta-\tfrac{1}{2})^{p-b}(g_{a,b,n}^+) & g^+_{a,b,n}. \arrow[l,mapsto]
\end{tikzcd}\]
The similar result also holds for $g_{a,b,a}+(-1)^{p+q-b}\cdot 2^{p-b}$. Since there exist polynomials $M^1_{a,b}(x)$, $M^2_{a,b}(x)$ satisfying
\[M^1_{a,b}(x)x^{p-a}(x-1)^a=M^2_{a,b}(x)(x-\tfrac{1}{2})^{p-b}-1,\]
it follows that on $V_{\mathcal{L}_{a,b}^+}$,
\[0=M_{a,b}^1(\theta)\mathcal{L}_{a,b}^+=\left(M_{a,b}^2(\theta)-(-1)^q2^{2p}Q_1^+M_{a,b}^1(\theta+1)\cdot(\theta+\tfrac{1}{2})^b\right)(\theta-\tfrac{1}{2})^{p-b}-1\]
which implies that
\begin{equation}\label{g_a,b,n^+_in_F^+}
g_{a,b,n}^++(-1)^{p+q-b}\cdot 2^{p-b}\cdot\delta_{n,a}=\left(M_{a,b}^2(\theta)-(-1)^q2^{2p}Q_1^+M_{a,b}^1(\theta+1)\cdot(\theta+\tfrac{1}{2})^b\right)\theta^a((-1)^q2^{-2p}\cdot g_n^+)
\end{equation}
belongs to $F^+$ for $n\leq a$. Consequently, the set $\{g_{a,b,i}^+\}_{i=0}^{a+b+1-p}$ of functions appearing in the quantum product $D\qsstar h^{p+q+a}(\xi-h)^by^+_\infty$ is contained in $F^+$, since
\[i\leq a+b+1-p\leq a\]
when $0\leq a,b\leq p-1$. This completes the proof of Theorem \ref{div_3pt_in_F} for $\mathcal{X}^+$.

Using a similar argument for $\mathcal{X}^-$, one can easily verify that $(W^-)^{[k]\setminus\{n\}}_{[k-1]}, g_{a,b}^- \in F^-$, and deduce Theorem \ref{div_3pt_in_F} for $\mathcal{X}^-$.

\subsubsection{The Calabi--Yau operator and the associated bilinear form}
For $\mathcal L=\sum_{i=0}^n c_i\theta^i,$ the formal adjoint is defined by 
\[\mathcal{L}^\vee\coloneqq\sum_{i=0}^n(-\theta)^ic_i.\]
A differential operator $\mathcal{L}$ of order $n$ with maximal unipotent monodromy at $Q=0$ whose local exponents are $\alpha$'s is called a \textbf{Calabi--Yau} operator if 
\[Q^{-\alpha}\mathcal{L}Q^\alpha=(-1)^n\left(Q^{-\alpha}\mathcal{L}Q^\alpha\right)^\vee.\]
We restrict ourselves to the case where we have the expansion $\mathcal L=Q^\alpha(\theta^n+\sum_{k=1}^{m}Q^kc_k(\theta))Q^{-\alpha}$ for some polynomials $c_k$'s. In this case, the Calabi--Yau condition can be rewritten as
\begin{equation}\label{cy_condition}
    c_k(-k-\theta)=(-1)^nc_k(\theta).
\end{equation}

\begin{proposition}\label{bilinear_pairing}
Let $\mathcal L$ be a Calabi--Yau operator having the expansion as above. There exists a differential operator $B_{\mathcal{L}}$ which is bilinear in $f$, $g$ such that
\[Q^{2\alpha}\theta B_{\mathcal{L}}(f,g)=g\mathcal L(f)-(-1)^nf\mathcal L(g).\]
In particular, it induces a $G_\mathcal{L}$-invariant nondegenerate bilinear form $B_{\mathcal{L}}:V_{\mathcal L}\times V_{\mathcal L}\to\C$ when  $f$, $g\in V_{\mathcal{L}}$. Moreover, if $\{g_i\}_{i=0}^{n-1}$ is the normalized Frobenius basis of $\mathcal{L}$ at $Q=0$, then we have 
\[B_{\mathcal L}(g_i,g_j)=(-1)^j\delta_{i+j,n-1}.\]
\end{proposition}

\begin{proof}
We start with the case $\alpha=0$. From the identity
\[\theta\left(\sum_{i=0}^{\ell-1}(-1)^i\theta^{\ell-1-i}(Q^kf)\theta^{i}g\right)=\theta^\ell(Q^kf)g-Q^kf\cdot(-\theta)^\ell g,\]
we deduce that
\[Q^k(c_k(\theta)f)\cdot g=\left(c_k(\theta-k)(Q^kf)\right)\cdot g=f\cdot Q^kc_k(-\theta-k)(g)+\theta B_k(f,g)\]
for some $B_k(f,g)\in Q^k\gen{\theta^{n-1-i}(f)\theta^i(g):0\leq i\leq n-1}$. By \eqref{cy_condition}, 
\[\mathcal{L}(f)\cdot g-(-1)^nf\cdot \mathcal{L}(g)=\theta\left(\sum_{i=0}^{n-1}(-1)^i\theta^{n-1-i}f\cdot \theta^ig+\sum_{k=1}^m B_k(f,g)\right),\]
and thus we can define $$B_{\mathcal L}(f, g) = \sum_{i=0}^{n-1}(-1)^i\theta^{n-1-i}f\cdot \theta^ig+\sum_{k=1}^m B_k(f,g).$$
In particular, for $f$, $g\in V_{\mathcal{L}}$, we have $B_{\mathcal L}(f,g)\in \C$ and 
\[B_{\mathcal{L}}(\sigma f,\sigma g)=\sigma(B_\mathcal{L}(f,g))=B_{\mathcal{L}}(f,g)\]
for any $\sigma\in G_{\mathcal{L}}$. Let $\{g_i\}_{i=0}^{n-1}$ be the Frobenius basis at $Q=0$ such that $g_0(0)=1$. Since
\[\theta^j(g_i)\in \delta_{i,j}+\C[\log Q][\![Q]\!]\setminus\C,\]
we can get the value of the pairing 
\[B_\mathcal{L}(g_i,g_j)=\delta_{i+j, n-1}\cdot (-1)^j.\]

For arbitrary $\alpha$, we set $\widetilde{\mathcal{L}}=Q^{-\alpha}\mathcal{L}Q^\alpha$. The desired bilinear form is then constructed by defining
\[B_\mathcal{L}(f,g)\coloneqq B_{\widetilde{\mathcal{L}}}(Q^{-\alpha}f,Q^{-\alpha}g).\]
\end{proof}

In general, for an irreducible hypergeometric operator of the form
\[\prod_{j=1}^n(\theta+\beta_j-1)-Q\prod_{j=1}^n(\theta+\alpha_j)\]
where
\[\{\operatorname{exp}(2\pi i\alpha_j)\}_{j=1}^n=\{\operatorname{exp}(-2\pi i\alpha_j)\}_{j=1}^n,\quad \{\operatorname{exp}(2\pi i\beta_j)\}_{j=1}^n=\{\operatorname{exp}(-2\pi i\beta_j)\}_{j=1}^n,\]
the differential Galois group is classified in \cite{beukers1989monodromy} as either $\operatorname{Sp}(n,\C)$ or $\mathrm{O}(n,\C)$. In particular, for $\mathcal{L}^\pm$, since the elements in $G_{\mathcal{L}^\pm}$ preserve $B_{\mathcal{L}^\pm}$, we conclude that $G_\mathcal{L^\pm}=\mathrm{O}(V_\mathcal{L^\pm},B_\mathcal{L^\pm})$.

\subsubsection{Regularization map}
It is ready to construct a differential field isomorphism from $F^+$ to $F^-$.
\begin{thm}\label{well_def_regularization}
For any $c=(c_0,\ldots,c_{p-1})$ such that $c_ic_{p-1-i}=(-1)^{p+q}$, the $\C$-linear map
\[\begin{tikzcd}[row sep=1pt]
V_{\mathcal L^+} \arrow[rr,"\phi_c"] & & V_{\mathcal L^-} \\
g_i^+ \arrow[rr, maps to] && c_i^{-1}2^{-p}g_i^-
\end{tikzcd}\]
extends to a differential isomorphism $\phi_c:F^+\to F^-$ under $Q_1^+ \rightarrow (Q_1^-)^{-1}$.
\end{thm}

\begin{proof}
First, we recall the classical analytic continuation that one can pick a path $\gamma$ from $0$ to $\infty$ in $\P^1$ (w.r.t. affine variables $Q_1^+$ and $Q_1^-$), and let $\gamma(g_i^+)\in\operatorname{Sol}(\mathcal{L}^-)$ denote the analytic continuation of $g_i^+$ along the path $\gamma$. Abusing notation, we let $\gamma$ also represent the differential field isomorphism $F^+\xrightarrow{\ \sim\ } F^-$ induced by the analytic continuation along the path $\gamma$. We find that
\begin{align*}
\gamma\circ \mathcal{L}^+&=(-1)^{p+q+1}2^{2p}(Q^-)^{-1}\left((\theta_{Q^-}-\tfrac{1}{2})^p-Q^-(-1)^q2^{-2p}\theta_{Q^-}^p\right)\circ\gamma \\
&=(-1)^{p+q+1}2^{2p}(Q^-)^{-1}\mathcal{L}^-\circ\gamma\end{align*}
and
\begin{align*}
\gamma\theta_{Q^+}B_{\mathcal{L}^+}(f,g) & = (-1)^{p+q+1}2^{2p}(Q^-)^{-1}\left(\mathcal{L}^-(\gamma f)\cdot \gamma g-\gamma f\cdot \mathcal{L}^-(\gamma g)\right) \\
& =(-1)^{p+q+1}2^{2p}\theta_{Q^-}B_{\mathcal{L}^-}(\gamma f,\gamma g)
\end{align*}

as the differential operator. This implies that
\[C(f,g)\coloneqq \gamma B_{\mathcal{L}^+}(f,g)-(-1)^{p+q}2^{2p}B_{\mathcal{L}^-}(\gamma f,\gamma g)\]
is constant for all $f$, $g\in F^+$. By the construction of $B_{\mathcal{L}^\pm}$, we may write
\[C(f,g)=\sum C_{a,b}(Q_1^-)\theta_{Q^-}^a(\gamma f) \theta^b_{Q^-}(\gamma g)\]
for some polynomial $C_{a,b}$. Substituting $(f,g)=(Q^i,Q^j)$ yields
\[\sum_{0\leq a+b\leq n-1} C_{a,b}(Q_1^-)i^aj^b=0\]
for all $i$, $j\in\N$ and thus $C_{a,b}(Q_1^-)=0$ for all $a,b$. Consequently, we obtain 
\[B_{\mathcal{L}^+}(f,g)=(-1)^{p+q}2^{2p}B_{\mathcal{L}^-}(\gamma f,\gamma g)\]
for all $f$, $g\in V_{\mathcal{L}^+}$. By evaluating on the Frobenius basis as in Proposition \ref{bilinear_pairing}, 
\begin{align*}
B_{\mathcal{L}^-}(\gamma(g_i^+),\gamma(g_j^+))&=(-1)^{p+q}2^{-2p} B_{\mathcal{L}^+}(g_i^+,g_j^+) = (-1)^{p+q}2^{-2p}(-1)^j\delta_{i+j, p-1}\\
&= (-1)^{p+q}2^{-2p} B_{\mathcal{L}^-}(g_i^+,g_j^+) 
= B_{\mathcal L^-}(c_i^{-1}2^{-p}g_i^-,c_j^{-1}2^{-p}g_j^-).
\end{align*}
Therefore, one can find $\sigma\in\mathrm{O}(V_{\mathcal{L}^-},B_{\mathcal{L}^-})=G_{\mathcal{L}^-}$ such that $\sigma(\gamma(g_i^+))=c_i^{-1}2^{-p}g_i^-$. This shows that $\phi_c=\sigma\circ\gamma$ is a well-defined differential field isomorphism.
\end{proof}

Using the same notation as that in the proof of Theorem \ref{well_def_regularization}, we have the commutative diagram
\[\begin{tikzcd}[row sep=20pt, column sep=80pt]
V_{\mathcal{L}^+} \arrow[r,"\theta_{Q^+}^a"] \arrow[d, "\gamma"] \arrow[dd, "\phi_c"', bend right=60] & V_{\mathcal{L}^+_{a,p}} \arrow[d, "\gamma"] & V_{\mathcal{L}_{a,b}^+} \arrow[l,"(\theta_{Q^+}-\frac{1}{2})^{p-b}"'] \arrow[d, "\gamma"]  \\
V_{\mathcal{L}^-} \arrow[r,"(-1)^a\theta_{Q^-}^a"] \arrow[d, "\sigma"] & V_{\mathcal{L}^-_{p,a}} \arrow[d, "\sigma"] & V_{\mathcal{L}_{b,a}^-} \arrow[l,"(-1)^{p-b}(\theta_{Q^-}+\frac{1}{2})^{p-b}"'] \arrow[d, "\sigma"] \\
V_{\mathcal{L}^-} \arrow[r,"(-1)^a\theta_{Q^-}^a"] & V_{\mathcal{L}^-_{p,a}} & V_{\mathcal{L}_{b,a}^-} \arrow[l,"(-1)^{p-b}(\theta_{Q^-}+\frac{1}{2})^{p-b}"']
\end{tikzcd}\]

Note that each arrow represents a $\C$-linear isomorphism. By the relation in Lemma \ref{g_a,b^+} and \ref{g_a,b^-}, we get that
\[(-1)^{p-b}(\theta_{Q^-}+\tfrac{1}{2})^{p-b}\phi_c\left(g^+_{a,b,n}+(-1)^{p+q-b}\cdot 2^{-p-b}\cdot\delta_{n,a}\right)=c_n^{-1}(-1)^a(\theta_{Q^-}+\tfrac{1}{2})^{p-b}(g_{b,a,n}^-)\]
for $n\leq a$, and thus
\[\phi_cg_{a,b,n}^++(-1)^{p+q-b}\cdot 2^{-p-b}\cdot\delta_{n,a}=c_n^{-1}(-1)^{p+q-a-b}g_{b,a,n}^-\]
for $n\leq a$. Finally, we conclude that
\[\phi_c\left(\dfrac{W^+_{a,b}}{W_k^+}\right)=(-1)^{p+q-a-b}c_{k+1}^{-1}\cdot\dfrac{W^-_{b,a}}{W_k^-}-(-1)^{q+1}\cdot 2^{-2p+1}\cdot\delta_{b,p-1}\]
for $b\leq p-1$. 

\begin{prop}\label{preserve_quan_prod_implies_id}
Let $\sigma\in G_{\mathcal{L}^-}$ be extended to an automorphism of $F^-((Q_1^-Q_2^-)^{1/2})$ by fixing $(Q_1^-Q_2^-)^{1/2}$. If $\sigma(D\qsstar T)=D\qsstar T$ for any divisor $D$ and $T\in A_{\operatorname{orb}}^*(\mathcal{X}^-)$, then $\sigma=\operatorname{id}$.
\end{prop}

\begin{proof}
By Section \ref{subsection_quan_prod_X^-}, the assumption implies that $\sigma$ also preserves
\[g_0^- \text{ and } \frac{W_{k-1}^-\cdot W_{k+1}^-}{(W_k^-)^2}\]
for $0\leq n\leq k\leq p-1$. In particular, it preserves $W_k^-$ for $0\leq k\leq p-1$. Under the Frobenius basis $\{g_i^-\}_{i=0}^{p-1}$, the property $\sigma(W_k^-)=W_k^-$ implies that $\sigma$ is represented by a unipotent upper triangular matrix. 

Moreover, $\sigma$ also preserves
\[\dfrac{W_{a,b}^-}{W_k^-},\quad \sum_{i=0}^k(-1)^{i+n}g_{a,b,i}^-\dfrac{(W^-)_{[k]\setminus\{i\}}^{[k]\setminus\{n\}}}{W_k^-}\]
for $0\leq a,b\leq p-1$, $0\leq n\leq k$ with $k=a+b-p\geq 0$. Since $\sigma$ also preserves $(W^-)_{[k]\setminus\{i\}}^{[k]\setminus\{n\}}$, we have
\[\dfrac{1}{W_k^-}\sum_{i=0}^k(-1)^{i+n}(W^-)_{[k]\setminus\{i\}}^{[k]\setminus\{n\}}(\sigma(g^-_{a,b,i})-g^-_{a,b,i})=0\]
for all $0\leq n\leq k$. Since the matrix $((-1)^{i+j}(W^-)_{[k]\setminus\{j\}}^{[k]\setminus\{i\}}/W_k^-)_{0\leq i,j\leq k}=(W_k^-)^{-1}$ is invertible, we deduce that $\sigma(g^-_{a,b,i})=g^-_{a,b,i}$. Finally, since $g_i^-\in\mathfrak{D}_{Q_1^-}\cdot g_{a,b,i}^-$, $\sigma$ also fixes $g_i^-$ for all $i$, which implies that $\sigma=\operatorname{id}$.
\end{proof}

By the Galois correspondence,  we get the following result.
\begin{cl}\label{proper-field}
The differential fields $F^+((Q_2^+)^{1/2})$ and $F^-((Q_1^-Q_2^-)^{1/2})$ are the smallest differential field extensions of $\mathbb{C}(Q^+_1,(Q^+_2)^{1/2})$ and $\mathbb{C}(Q_1^-,(Q_1^-Q_2^-)^{1/2})$ respectively that contain all three-point invariants $\langle D,T_1,T_2 \rangle^{\mathcal{X}^\pm}$ for any divisor $D$ and arbitrary orbifold cohomology classes $T_1$ and $T_2$.
\end{cl}

\subsubsection{Quantum correspondence}
\begin{definition}\label{quantum correspond}
Let $\mathcal X^+\dashrightarrow\mathcal X^-$ be a local model for simplicial toric flops described as above. A \textbf{quantum correspondence} for $\mathcal X^+\dashrightarrow\mathcal X^-$ consists of the following data:
\begin{enumerate}
    \item\label{quantum_correspondence_1} a correspondence $[\mathscr F]\in A^{\dim \mathcal X}_{\operatorname{orb}}(\mathcal X^+\times \mathcal X^-)_\C$ preserving the orbifold Poincar\'e pairing,

    \item\label{quantum_correspondence_2} differential field extensions $L^\pm$ of $\C(Q^\pm_1,Q_2^\pm)$ containing all $n$-point invariants with $n\geq 3$ on $\mathcal{X}^\pm$,
    
    \item\label{quantum_correspondence_3} a differential field isomorphism $\phi:L^+\to L^-$ under $Q^d\mapsto Q^{\mathscr F(d)}$ for $d=\ell^+,\gamma^+$,
\end{enumerate}
satisfying that the induced map
    \[\Phi\coloneqq [\mathscr{F}]\otimes\phi:A_{\operatorname{orb}}^*(\mathcal{X}^+)_\C\otimes L^+[\![\tau^\mu]\!]\longrightarrow A_{\operatorname{orb}}^*(\mathcal{X}^-)_\C\otimes L^-[\![\tau^\mu]\!]\]
    is a ring isomorphism with respect to the big quantum product structure, that is,
    \begin{equation}\label{Psi_preserve_big_quan_prod}
    \Phi(T_1\star_{\tau_{\operatorname{big}}}T_2)=\Phi(T_1)\star_{\Phi(\tau_{\operatorname{big}})}\Phi(T_2),
    \end{equation}
    where $\tau_{\operatorname{big}}=\sum\tau^\mu T_\mu$ and $\Phi$ acts trivially on $\tau^\mu$. 

We denote it by $(\phi,[\mathscr F])$ and call $\phi$ a \textbf{regularization map}.
\end{definition}

The special case of $\tau_{\operatorname{big}}=0$ in \eqref{Psi_preserve_big_quan_prod} guarantees that
\begin{equation}\label{Psi_preserve_small_quan_prod}
\Phi(D\qsstar T)=\Phi(D)\qsstar \Phi(T)
\end{equation}
for any divisor $D$ and $T\in A_{\operatorname{orb}}^*(\mathcal{X}^+)$.

\begin{thm}\label{preserve_small_implies_preserve_big}
Let $\mathcal X^+\dashrightarrow\mathcal X^-$ be a local model for simplicial toric flops, and $(\phi,[\mathscr{F}])$ be a pair satisfies the conditions \eqref{quantum_correspondence_1}, \eqref{quantum_correspondence_2} and \eqref{quantum_correspondence_3} in Definition \ref{quantum correspond}. If the orbifold cohomology of $\mathcal{X}^+$ is generated by divisors with respect to $\qsstar$ over the definition field $L^+$ of $\phi$, then the pair $(\phi,[\mathscr{F}])$ satisfying \eqref{Psi_preserve_small_quan_prod} turns out to be a quantum correspondence.
\end{thm}

\begin{proof}
By the definition of $\star_{\operatorname{big}}$, it suffices to show that 
\begin{equation}\label{Psi_preserve_n+3_points}
\phi\gen{\alpha_1,\alpha_2,\alpha_3,\vec{\beta}}_{n+3}^{\mathcal{X}^+}=\gen{[\mathscr{F}]\alpha_1,[\mathscr{F}]\alpha_2,[\mathscr{F}]\alpha_3,[\mathscr{F}]\vec{\beta}}_{n+3}^{\mathcal{X}^-},
\end{equation}
for any $\alpha_1, \alpha_2, \alpha_3\in A_{\operatorname{orb}}^*(\mathcal{X}^+)$, $\vec{\beta} \in A_{\operatorname{orb}}^*(\mathcal{X}^+)^{\oplus n}$.
We prove it by induction on $n$. For the base case $n=0$, we have to show that it preserves 3-point invariants, under the fact that $[\mathscr{F}]$ preserves the orbifold Poincar'e pairing, which is equivalent to that it preserves the small quantum product. By assumption, we may assume that
\[\alpha_1=D_1 \star_{\rm small}\cdots \star_{\rm small}D_m,\]
so \eqref{Psi_preserve_small_quan_prod} implies that
\begin{align*}
\Phi(\alpha_1 \qsstar \alpha_2) &= \Phi(D_1 \star_{\rm small}\cdots \star_{\rm small}D_m \qsstar \alpha_2) \\
&= \Phi(D_1)\qsstar\cdots \qsstar \Phi(D_m)\qsstar \Phi(\alpha_2) \\
&= \Phi(D_1 \star_{\rm small}\cdots \star_{\rm small}D_m) \qsstar \Phi(\alpha_2) =\Phi(\alpha_1) \qsstar \Phi(\alpha_2).
\end{align*}

For general $n\geq 1$, we intend to show that \eqref{Psi_preserve_n+3_points} holds for 
$\alpha_1=D_1 \star_{\rm small}\cdots \star_{\rm small}D_m$ and 
by induction on $(n, m)$ according to lexicographic order. For the base case $(n,1)$, by the divisor axiom, it is reduced to the case consists of $(n-1, m)$ with $m\in\Z_{\geq 0}$. 

For general $(n,m)$ with $n\geq 1$ and $m\geq 2$, the WDVV equation shows that
\[\gen{D_1 \star_{\rm small}\cdots \star_{\rm small}D_m,\alpha_2,\alpha_3,\vec{\beta}}+\sum_{I\neq\vn}\gen{D_1 \star_{\rm small}\cdots \star_{\rm small}D_{m-1},D_m,(\vec{\beta})_I,T_\mu}\gen{T^\mu,\alpha_2,\alpha_3,(\vec{\beta})_J}\]
is equal to 
\[\sum_{I}\gen{D_1 \star_{\rm small}\cdots \star_{\rm small}D_{m-1},\alpha_2,(\vec{\beta})_I,T_\mu}\gen{T^\mu,D_m,\alpha_3,(\vec{\beta})_J},\]
where $I = \{i_1,\ldots, i_l\} \subset \{1, \ldots, n\}$ and $J=\{1,\ldots, n\}\setminus I$. Since $I\neq\vn$, $\gen{T^\mu,\alpha_2,\alpha_3,(\vec{\beta})_J}$ has less insertion, and thus \eqref{Psi_preserve_n+3_points} holds by induction hypothesis. Since $1$ and $m-1<m$, \eqref{Psi_preserve_n+3_points} also holds for the second equation. Hence \eqref{Psi_preserve_n+3_points} holds for $\gen{D_1 \star_{\rm small} \cdots \star_{\rm small} D_m,\alpha_2,\alpha_3,\vec{\beta}}$.
\end{proof}

\begin{rmk}\mbox{}
\begin{enumerate}
    \item If $L^\pm$ contains only special $3$-point invariants $\gen{D,T_1,T_2}^{\mathcal{X}^\pm}$ with $D\in A^1(\mathcal{X}^\pm)$ and $T_1$, $T_2\in A^*_{\operatorname{orb}}(\mathcal{X}^\pm)$, then the proof of Theorem \ref{preserve_small_implies_preserve_big} implies that $L^\pm$ contains all $n$-point invariant for $n\geq 3$ under the condition of divisor-generation. In particular, $L^+=F^+((Q_2^+)^{1/2})$ and $L^-=F^-((Q_1^-Q_2^-)^{1/2})$ satisfy the condition \eqref{quantum_correspondence_2} in Definition \ref{quantum correspond} by Theorem \ref{div_3pt_in_F} and Remark \ref{rmk_of_generator_of_quantum_cohomology}.
    
    \item If $(\phi_1,[\mathscr{F}])$ and $(\phi_2,[\mathscr{F}])$ are two quantum correspondences for a toric flop of the simplest type with $L^+=F^+((Q_2^+)^{1/2})$ and $L^-=F^-((Q_1^-Q_2^-)^{1/2})$, then $\sigma\coloneqq \phi_2\circ\phi_1^{-1}$ must be identity by Proposition \ref{preserve_quan_prod_implies_id}. This shows that the correspondence $[\mathscr F]$ is uniquely determined by the regularization map $\phi$ for the simplest type.
\end{enumerate}

\end{rmk}

\subsubsection{$(\phi_c,[\mathscr{F}_c])$ is a quantum correspondence} The main result of this paper is summarized in the following theorem.
\begin{thm}\label{main result}
Let the correspondence $[\mathscr{F}_{c, c_\infty}]$ be defined by
\[[\mathscr{F}_{c,c_\infty}]=[\mathcal{Y}_{(0,0,0)}]+(-1)^{p+q}\sum_{i=0}^{p-1}c_ih_+^{p-1-i}h_-^i[\mathcal{Y}_{(0,\frac{1}{2},0)}]+c_\infty[\mathcal{Y}_{(1,\frac{1}{2},\frac{1}{2})}],\]
where $h_\pm$ is the pullback of $h$ from $\mathcal{X}^\pm$ to $\mathcal{Y}$. If $c_ic_{p-1-i}=(-1)^{p+q}$ for all $0\leq i\leq p-1$ and $c_\infty=1$, then $(\phi_c,[\mathscr{F}_{c,1}])$ is a quantum correspondence
where $\phi_c$ is constructed in Theorem \ref{well_def_regularization}. 
\end{thm}
\begin{proof}
If $c_ic_{p-1-i}=(-1)^{p+q}$ for all $0\leq i\leq p-1$ and $c_\infty=\pm 1$, then by Lemma \ref{lm_preserve_orbifold_prod}, the correspondence $[\mathscr{F}_{c,c_\infty}]$
preserves the orbifold Poincar\'e pairing. If it is a quantum correspondence, then $c_\infty=1$ by Proposition \ref{quan_isom^+_easy} and \ref{quan_isom^-_easy} and the following computation
\begin{align*}
\dfrac{(Q_1^-Q_2^-)^{1/2}}{2^{p}}\cdot(\xi-h) & =\Phi\left((\xi-h)\qsstar h(\xi-h)^{p-1}y_\infty^+\right) \\
& =h\qsstar c_\infty h^{p-1}(\xi-h)y_\infty^-=c_\infty\cdot \dfrac{(Q_1^-Q_2^-)^{1/2}}{2^{p}}\cdot(\xi-h).
\end{align*}

For convenience, set $[\mathscr{F}_{c}]\coloneqq[\mathscr{F}_{c,1}]$ and $\Phi_c\coloneqq [\mathscr{F}_c]\otimes\phi_c$. By Theorem \ref{generator of quantum cohomology}, Theorem \ref{div_3pt_in_F}, and Remark \ref{rmk_of_generator_of_quantum_cohomology}, the orbifold cohomology of $\mathcal{X}^+$ is generated by divisors with respect to $\qsstar$ over $F^+$, so the remaining job is to check that $(\phi_c,[\mathscr{F}_c])$ satisfies \eqref{Psi_preserve_small_quan_prod} for any $c_ic_{p-1-i}=(-1)^{p+q}$.

For the case where $\deg_{\operatorname{orb}} T\leq p+q-1$ or $T\in \xi A_{\operatorname{orb}}^*(\mathcal{X}^+)$, it is easy to check. 

For $T=h^k(\xi-h)^{p+q}$ with $0\leq k\leq p-1$, we have
\begin{align*}
&\ \Phi_c(h\qsstar h^k(\xi-h)^{p+q}) \\
=&-c_k\cdot \dfrac{W^-_{k-1}\cdot W^-_{k+1}}{(W^-_k)^2}h^{k+1}y_{1/2}^-+\phi_c\left(\dfrac{W_{p-2}^+\cdot W_p^+}{(W_{p-1}^+)^2}\right)\cdot (-1)^{p+q-1}(\xi-h)^{2p+q}\cdot\delta_{k,p-1}\\
&-2(Q_1^-Q_2^-)^{1/2}\sum_{n=0}^k(-1)^nc_k\cdot\theta_{Q_1^-}\left(\dfrac{(W^-)_{[k-1]}^{[k]\setminus\{n\}}}{W_k^-}\right)\cdot 2^{k-n}(\xi-h)^n\xi^{k-n}y_\infty^-.
\end{align*} 
By Lemma \ref{Wron_k^+}, 
\[\phi_c\left(\dfrac{W_{p-2}^+\cdot W_p^+}{(W_{p-1}^+)^2}\right)=(-1)^{p-1}\cdot 2^pc_{p-1}\cdot \dfrac{W_{p-2}^-}{W_{p-1}^-}\cdot\dfrac{1}{1-(-1)^q2^{2p}(Q_1^-)^{-1}}.\]
On the other hand, $\Phi_c(h)\qsstar\Phi_c(h^k(\xi-h)^{p+q})$ equals 
\begin{align*}
&\ (\xi-h)\qsstar c_kh^ky^-_{1/2} \\
\overset{\mathclap{(\ref{quan_prod_-_exceptional})}}{=}&-c_k\cdot\dfrac{W_{k-1}^-\cdot W_{k+1}^-}{(W_k^-)^2}h^{k+1}y_{1/2}^--2^pc_{p-1}\cdot\dfrac{W_{p-2}^-}{W_{p-1}^-}\cdot\dfrac{2^{-2p}Q_1^-}{1-(-1)^q2^{-2p}Q_1^-}(\xi-h)^{2p+q}\cdot\delta_{k,p-1} \\
& -2c_k(Q_1^-Q_2^-)^{1/2}\sum_{n=0}^k (-1)^n\theta_{Q_1^-}\left(\dfrac{(W^-)^{[k]\setminus\{n\}}_{[k-1]}}{W_k^-}\right)\cdot 2^{k-n}(\xi-h)^n\xi^{k-n}y_\infty^-\\
\overset{\mathclap{(\ref{quan_prod_+_iden})}}{=}&\ \Phi_c(h\qsstar h^k(\xi-h)^{p+q})
\end{align*}

It is clear that $\Phi_c(\xi\qsstar h^k(\xi-h)^{p+q})=\Phi_c(\xi)\qsstar \Phi_c(h^k(\xi-h)^{p+q})$ when $0\leq k\leq p-1$.

For $T=h^k(\xi-h)^{p+q}$ with $p\leq k\leq 2p+q-1$, $\Phi_c(h)\qsstar\Phi_c(h^k(\xi-h)^{p+q})$ equals
\begin{align*}
&\ {(\xi-h)\qsstar (-1)^{q-1-k}h^{k-p}(\xi-h)^{2p+q}}\\
=&\ (-1)^{q-1-k}(\xi-h)\qsstar \sum_{m=0}^{2p+q-1-k} (-1)^m\xi h^{k-p+m}(\xi-h)^{2p+q-1-m} \\
\overset{\mathclap{(\ref{infinite_cohomo_quantize_-})}}{=}&\ (-1)^{q-1-k}\left(-h^{k+1-p}(\xi-h)^{2p+q+1}+Q_2^-h^{k-p}\right) \\
\overset{\mathclap{(\ref{quan_prod_+_iden})}}{=}&\ \Phi_c(h\qsstar h^k(\xi-h)^{p+q}),
\end{align*}
and $\Phi_c(\xi)\qsstar\Phi_c(h^k(\xi-h)^{p+q})$ equals
\begin{align*}
&\ \xi\qsstar (-1)^{q-1-k}h^{k-p}(\xi-h)^{2p+q}\\
=&\ (-1)^{q-1-k}\xi\qsstar \sum_{m=0}^{2p+q-1-k} (-1)^m\xi h^{k-p+m}(\xi-h)^{2p+q-1-m} \\
\overset{\mathclap{(\ref{infinite_cohomo_quantize_-})}}{=}&\ (-1)^{q-1-k}\left((-1)^{2p+q-1-k}\cdot\frac{(Q_1^-Q_2^-)^{1/2}}{2^p}\cdot (\xi-h)^ky_\infty^++Q_2^-h^{k-p}\right) \\
\overset{\mathclap{(\ref{quan_prod_+_iden})}}{=}&\ \Phi_c(h\qsstar h^k(\xi-h)^{p+q}).
\end{align*}

For the case of $T=h^a(\xi-h)^by_\infty^+$ when $a\leq p+q-1$, it is clear. 

For the case of $T=h^{p+q+a}(\xi-h)^b$ with $b\leq p-1$ such that $k\coloneqq a+b-p\geq 0$, we have
\begin{align*}
&\ \Phi_c(h\qsstar h^{p+q+a}(\xi-h)^by^+_\infty) \\
\overset{\mathclap{(\ref{quan_isom^+_hard})}}{=}&\ h^b(\xi-h)^{p+q+a+1}y^+_\infty-(-1)^{p+q-a}\cdot 2^p(Q_1^-Q_2^-)^{1/2}\cdot \theta_{Q_1^-}\left(\dfrac{W^-_{b,a}}{W^-_k}\right)\cdot h^{k+1}y^-_{1/2} \\
&+(-1)^{b+q}\cdot 2^pQ_1^-Q_2^-\cdot\sum_{0\leq i,n\leq k}(-1)^{i+n}\theta_{Q_1^-}\left(g_{b,a,i}^-\dfrac{(W^-)^{[k]\setminus\{n\}}_{[k]\setminus\{i\}}}{W_k^-}\right)\cdot 2^{k+1-n}(\xi-h)^n\xi^{k-n}y^-_\infty \\
\overset{\mathclap{(\ref{quan_isom^-_hard})}}{=}&\ \Phi_c(h)\qsstar \Phi_c(h^{p+q+a}(\xi-h)^by^+_\infty).
\end{align*}
and 
\begin{align*}
&\ \Phi_c(\xi\qsstar h^{p+q+a}(\xi-h)^by^+_\infty) \\
\overset{\mathclap{(\ref{quan_isom^+_hard})}}{=}&\ \xi h^b(\xi-h)^{p+q+a}y^+_\infty+\dfrac{(Q_1^-Q_2^-)^{1/2}}{2^p}\cdot ((\xi-h)^{p+q+a}+(-1)^{p+q}c_ah^ay^-_{1/2})\cdot\delta_{b,p-1} \\
&\ +2^{p-1}(Q_1^-Q_2^-)^{1/2}\cdot\left((-1)^{p+q-a}c_{k+1}^{-1}\dfrac{W^-_{b,a}}{W^-_k}-(-1)^{b+q+1}2^{-2p+1}\cdot\delta_{b,p-1}\right)\cdot c_{k+1}h^{k+1}y_{1/2}^- \\
&\ -(-1)^{q+b}\cdot 2^pQ_1^-Q_2^-\sum_{0\leq i,n\leq k}(-1)^{i+n}g_{b,a,i}^-\dfrac{(W^-)^{[k]\setminus\{n\}}_{[k]\setminus\{i\}}}{W_k^-}\cdot 2^{k+1-n}h^n\xi^{k-n}y^+_\infty \\
\overset{\mathclap{(\ref{quan_isom^-_hard})}}{=}&\ \Phi_c(\xi)\qsstar \Phi_c(h^{p+q+a}(\xi-h)^by^+_\infty).
\end{align*}
\end{proof}

\section{Higher-genus invariance}
In this section, we show that the generating functions of Gromov--Witten invariants with ancestors are invariant under a non-smooth toric flop of the simplest type, for all genera, after the regularization map $\phi_c$. Our proof generalizes \cite{ILLW12} to orbifolds.
\subsection{Givental--Teleman's formula of ancestor potentials}
Let $\mathcal X$ be a proper smooth Deligne--Mumford stack with the projective coarse moduli space $X$. Let $\{T_\mu\}$ be a homogeneous basis of $A^*_{\operatorname{orb}}(\mathcal X)_\C$ and $\{\tau^{\mu}\}$ be the dual coordinates. The tangent vector $\partial/\partial\tau^\mu$ is identified with $T_\mu$. Let $\mathcal S=\operatorname{NE}(\mathcal X)\cap H_2(X,\Z).$ The orbifold quantum cohomology of $\mathcal X$:  
\[QH_{\operatorname{orb}}(\mathcal X)=A^*_{\operatorname{orb}}(\mathcal X)_\C\otimes\C[\![\mathcal S]\!],\]
is equipped with 
\begin{enumerate}
    \item quantum multiplication
    \[\star:QH_{\operatorname{orb}}(\mathcal X)\times QH_{\operatorname{orb}}(\mathcal X)\to QH_{\operatorname{orb}}(\mathcal X),~(T_1,T_2)\mapsto T_1\star_{\tau}T_2\]with $\tau$ varies in $A^*_{\operatorname{orb}}(\mathcal X)_\C$, (the commutativity and the associativity follow from its construction and the WDVV equation).
    \item Orbifold Poincar\'e metric $(\ ,\ )_{\operatorname{orb}}$,
    \item  Euler vector field:
    \begin{equation}\label{Euler_vector_field}
    E=\sum_\mu(1-\deg_{\operatorname{orb}}T_\mu)\tau^\mu\frac{\partial}{\partial\tau_\mu}+c_1(\mathcal X),
    \end{equation}
    \item Dubrovin connection:
    \[\nabla^z=d-\frac{1}{z}\sum_\mu d\tau^\mu(T_\mu\star\ ).\]
\end{enumerate}
By identifying $(QH_{\operatorname{orb}}(\mathcal X),\star_\tau)$ with the tangent space of $A^*_{\operatorname{orb}}(\mathcal X)_\C$ at $\tau$ with coefficients in $\C[\![\mathcal S]\!]$, the quintuple
$(A^*_{\operatorname{orb}}(\mathcal X)_\C,(~,~)_{\operatorname{orb}},\star_{\tau},1,E)$ defines a \textit{conformal Frobenius manifold} (See \cite{lee2004frobenius}). We denote it by $H$ and a point $\tau\in H$ means $\tau\in A^*_{\operatorname{orb}}(\mathcal X)_\C$. 

\subsubsection{Semisimplicity}
\label{subsubsec_semisimplicity}
A point $\tau\in H$ is called a semisimple point if $(QH_{\operatorname{orb}}(\mathcal X),\star_\tau)$ is isomorphic to the product algebra $\C[\![\mathcal S]\!]^{\dim A^*_{\operatorname{orb}}(\mathcal X)}$. A Frobenius manifold is said to be semisimple if the set of semisimple points is Zariski dense in $H$. In this case, there exist idempotent elements (generically defined) $\epsilon_i$, for $i=1,\ldots,\dim A^*_{\operatorname{orb}}(\mathcal X),$ i.e.
\[\epsilon_i\star\epsilon_j=\delta_{ij}\epsilon_i,\]
which are uniquely well-defined up to permutations. 
Since we have that $[\epsilon_i,\epsilon_j]=0$, there is a canonical coordinate system $\{u^i\}$ near a semisimple point such that $\partial/\partial u^i=\epsilon_i$. 
The canonical coordinate system is unique up to translation, and is normalized by the condition 
\begin{equation}\label{Euler_vf_in_canonical_coordinate}
E=\sum_i u^i\partial/\partial u^i
\end{equation}
(see \cite[Chapter 1, Section 3.6]{lee2004frobenius}). By the metric-compactibility of $\star$, $\{\epsilon_i\}$ defines an orthogonal frame. We consider their normalization $\tilde\epsilon_i=\epsilon_i/\sqrt{(\epsilon_i,\epsilon_i})_{\operatorname{orb}}$ for all $i$. 
With respect to the orthonormal frame $\{\tilde \epsilon_i\}$, the quantum differential equation $\nabla^zS=0$ has the fundamental solution of the form $R(\tau,z)e^{\mathbf u/z}$ where $\mathbf u$ is the diagonal matrix with $\mathbf u^{i}_j=u^i\delta_{ij}$ and $R(\tau,z)=\operatorname{id}+\sum_{n\geq 1}R_nz^n$ is a formal power series in $z$. The matrix $R(\tau,z)$ is uniquely determined by the homogeneity condition $(z\partial_z+\mathcal L_E)(R)=0$. If $\Psi(\tau)^{-1}$ is the transition matrix from the orthonormal frame $\{\tilde\epsilon_i\}$ to the flat frame $\{T_\mu\}$, then the fundemental solution with respect to the frame $\{T_\mu\}$ is of the form $\Psi(\tau)^{-1}R(\tau,z)e^{\mathbf u/z}$. Here, the matrix $\Psi$ is given by $\Psi_{\mu i}=(T_\mu,\tilde\epsilon_i)_{\operatorname{orb}}$.

\subsubsection{Givental's quantization formalism}
Let $\mathcal H=H[z,z^{-1}]\!]$ and the symplectic form $\Omega$ be as defined in Section 2.2. The decomposition $\mathcal H=\mathcal H_+\oplus\mathcal H_-$ with $\mathcal H_+=A^*_{\operatorname{orb}}(\mathcal X)_\C\otimes\C[\![\operatorname{NE}(\mathcal X)\cap H_2(X,\Z)]\!]$ and $\mathcal H_-=z^{-1}A^*_{\operatorname{orb}}(\mathcal X)_\C\otimes\C[\![\operatorname{NE}(\mathcal X)\cap H_2(X,\Z)]\!][\![z^{-1}]\!]$ defines a polarization which identifies $\mathcal H$ with the cotangent bundle $T^*\mathcal H_+$. Let $\{T_\mu z^k\}$ be a basis of $\mathcal H_+$ and $\{\mathbf q^\mu_k\}$ be the dual coordinates. Also, let
$\{T^\mu(-z)^{-k-1}\}$ be a basis of $\mathcal H_-$ and $\{\mathbf p_k^\mu\}$ be the dual coordinates.
We get that $\Omega=\sum_{\mu,k}d\mathbf p_\mu^k\wedge d\mathbf q^\mu_k.$

The quantization is defined by \begin{align}
    \hat 1&=1,\\\hat {\mathbf q}_\alpha=\frac{1}{\sqrt{\hbar}}\mathbf q_\alpha&,~\hat {\mathbf p}_\alpha=\sqrt{\hbar}\frac{\6}{\6\mathbf q_\alpha}\\
    (\mathbf q_\alpha \mathbf q_\beta)^\wedge=\frac{\mathbf q_\alpha\mathbf  q_\beta}{\hbar},~(\mathbf q_\alpha \mathbf p_\beta)^\wedge&=\mathbf q_\alpha\frac{\6}{\6\mathbf  q_\beta},~(\mathbf p_\alpha \mathbf p_\beta)^\wedge=\hbar\frac{\6^2}{\6 \mathbf q_\alpha\6 \mathbf q_\beta}.
\end{align}

For an infinitesimal symplectic transformation $T:\mathcal H\to \mathcal H$, the quantization is defined as the quantization of the quadratic hamiltonian $f\mapsto \Omega(Tf,f)/2$ in terms of the Darboux coordinate $(\mathbf q,\mathbf p)$. For a symplectic automorphism $e^A$, the quantization is defined by $\widehat{e^A}=e^{\hat A}.$

\subsubsection{Ancestor potentials}
For $2g+n-2>0$, let \[\sigma:\overline{\mathcal M}_{g,n+l}(\mathcal X,d)\to\overline{\mathcal M}_{g,n},~(\mathcal C,\mathfrak p_1,\ldots,\mathfrak p_{n+l},\mu:\mathcal C\to \mathcal X)\mapsto (C^{st},p_1,\ldots,p_n)\]be the composition of the stablization morphism on $\overline{\mathcal M}_{g,n+l}(\mathcal X,d)$ and the forgetful morphism on $\overline{\mathcal M}_{g,n+l}$. The ancestors  are defined to be 
 \[\bar\psi_{i}=\sigma^*\psi_i\]for $i=1,\ldots,n$. Denote  
$t=\sum_{\mu,k} t^\mu_k\bar\psi^k T_\mu$ and $\tau=\sum_\mu\tau^\mu T_\mu$. Let
\begin{align*}
    \bar{F}^{\mathcal X}_g( t,\tau)&=\sum_{n,l,d}\frac{q^d}{n!l!}\langle  t^{\otimes n},\tau^{\otimes l}\rangle_{g,n+l,d }\\&=\sum_{n,l,d}\frac{q^d}{n!l!}\int_{[\overline{\mathcal M}_{g,n+l}(\mathcal X,d)]^w}\prod_{j=1}^n\sum_{k,\mu} t^\mu_k\bar\psi^k_j\operatorname{ev}^*_jT_\mu\prod_{j=n+1}^{n+l}\operatorname{ev}_j^*\tau
\end{align*}
be the generating function of genus $g$ ancestor invariants.

The ancestor potential is defined to be 
\begin{equation}\label{def_ancestor_potential}
\mathscr A_\mathcal X( t,\tau)=\exp\left(\sum_{g=0}^\infty\hbar^{g-1}\bar F^\mathcal X_g(t,\tau)\right).
\end{equation}

For a Frobenius manifold $H$, let the orthonormal frame $\{\tilde\epsilon_i\}$ be induced from an idempotent basis $\{\epsilon_i\}$. Over a semisimple point $\tau\in H$,  $\C\tilde{\epsilon_i}$ is regarded as a 1-dimensional Frobenius manifold with Givental's symplectic formalism as above. Using the basis $\{\tilde \epsilon_iz^k\}$ and $\{\tilde\epsilon_iz^{-k-1}\}$ for $\mathcal H_+$ and $\mathcal H_-$, the Darboux coordinate is denoted by $(\mathbf q^i_k,\mathbf p^k_i)$.

Let $\mathscr A_{\dim A^*_{\operatorname{orb}}(\mathcal X)}(\mathbf t)=\prod_{i=1}^{\dim A^*_{\operatorname{orb}}(\mathcal X)}\mathscr A_{pt}(\mathbf t^i)$. For each factor $\mathscr A_{pt}(\mathbf t^i)$, we regard it as a formal function in $\mathbf t^i_k$ via dilaton shift 
\[\mathbf q_k^i=\mathbf t^i_k-\delta_{k1}.\]
By the Teleman's classification of semisimple cohomological field theories (See \cite{teleman}), we have the following formula:

\begin{theorem}\label{Givental_formula}
For $\mathcal X$ a proper smooth Deligne--Mumford stack with projective coarse moduli space and semisimple quantum cohomology,
\[\mathscr A_{\mathcal X}(t,\tau)=e^{\bar c(\tau)}\hat\Psi^{-1}(\tau)\widehat{\mathbf R}_{\mathcal X}(\tau,z)e^{\widehat{\mathbf u/z}}\mathscr A_{\dim A^*_{\operatorname{orb}}(\mathcal X)}(\mathbf t),\]where $\bar c(\tau)=\frac{1}{48}\log\det(\epsilon_i,\epsilon_j)_{\operatorname{orb}}$.
    
\end{theorem}
\begin{remark}
    Since $\Psi^{-1}(\tau)$ is independent of $z$, the quantization $\hat\Psi^{-1}(\tau)$ should be understood as the coordinate change of Givental's spaces from the normalized canonical frame to the flat coordinate frame $\mathbf q^i_k=\sum_\mu \mathbf q_k^\mu(T_\mu,\tilde\epsilon_i)$. 
\end{remark}

\subsection{Main result} Compared with \cite{ILLW12}, the key issue here is to deal with the effects induced by the regularization map $\phi_c$. 

First, we would like to give a proof for semisimplicity.
Let $\widetilde{F}^+_0$ be the splitting field of the polynomial 
\[\Xi^+_0(\lambda)\coloneqq\prod_{k_1=0}^{2p+q-1}\left(\lambda^{2p+q+1}-2^{-2p}Q_2^+\left(1+\zeta_{2p+q}^{k_1}2^{2p/(2p+q)}(Q_1^+)^{1/(2p+q)}\right)^{2p+q}\right)\]
over $F_0^+\coloneqq F^+((Q_2^+)^{1/2})$, and $\widetilde{F}^-_0$ be the splitting field of the polynomial 
\[\Xi^-_0(\lambda)\coloneqq \prod_{k_1=0}^{2p+q-1}\left(\lambda^{2p+q+1}-Q_2^-\left(1+\zeta_{2p+q}^{k_1}2^{-2p/(2p+q)}(Q_1^-)^{1/(2p+q)}\right)^{2p+q}\right)\]
over $F_0^-\coloneqq F^-((Q_1^-Q_2^-)^{1/2})$, where $\zeta_k=e^{2\pi i/k}$. It is clear that $\Xi_0^\pm(\lambda)\in\C(Q_1^\pm,Q_2^\pm)[\lambda]$ and $\phi_c(\Xi_0^+)=\Xi_0^-$.

\begin{lm}\label{eigenvalue_in_F_tilde}
If $\tau$ lies in the formal neighborhood of $0$, then the operator $c_1(\mathcal{X}^\pm)\star_\tau$ on $A_{\operatorname{orb}}^*(\mathcal{X}^\pm)\otimes \widetilde{F}_0^\pm[\![\tau]\!]$ has distinct eigenvalues in $\widetilde{F}_0^\pm[\![\tau]\!]$.
\end{lm}

\begin{proof}
We give the proof for $\mathcal{X}^+$ and the proof for $\mathcal{X}^-$ is similar. Since $c_1(\mathcal{X}^+)=(2p+q+1)\xi^+$,
it suffices to show that the assertion holds for the operator $\xi^+\star_\tau$.

We first consider the case of $\tau=0$. Recall that the small quantum cohomology ring of $\mathcal{X}^+$ is generated by divisors over $\widetilde{F}_0^+$ by Remark \ref{rmk_of_generator_of_quantum_cohomology}, so we get the surjective map
\begin{equation}\label{eq_polynomial_surj_to_small_quantum}
\begin{tikzcd}[row sep=0pt]
(\widetilde{F}_0^+[h,\xi],\ \cdot\ ) \arrow[rr, two heads] && (A^*_{\operatorname{orb}}(\mathcal{X}^+)_\C\otimes \widetilde{F}_0^+,\ \qsstar\ ) \\
h,\xi \arrow[rr, mapsto] && h^+,\xi^+.
\end{tikzcd}
\end{equation}
Similar to the proof of Theorem \ref{generator of quantum cohomology}, by applying the dequantization procedure to the Picard--Fuch equations $\square_{\ell^+}$ and $\square_{\ell^++\gamma^+}$, we get two relations in the small quantum cohomology ring:
\[(h^+)^{\qsstar(2p+q)}-2^{2p}Q_1^+(\xi^+-h^+)^{\qsstar(2p+q)},\quad \xi^+\qsstar (h^+)^{\qsstar(2p+q)}-Q_1^+Q_2^+\]
and thus the surjective map \eqref{eq_polynomial_surj_to_small_quantum} factors through the Batyrev's ring
\begin{equation}
\begin{tikzcd}\label{eq_Batyrev_surj_to_small_quantum}
\widetilde{F}^+_0[h,\xi]/(h^{2p+q}-2^{2p}Q_1^+(\xi-h)^{2p+q},\xi h^{2p+q}-Q_1^+Q_2^+)\arrow[rr, two heads] && A^*_{\operatorname{orb}}(\mathcal{X}^+)_\C\otimes \widetilde{F}^+_0.
\end{tikzcd}
\end{equation}
Since both sides of \eqref{eq_Batyrev_surj_to_small_quantum} have dimension $(2p+q)(2p+q+1)$ over $\widetilde{F}_0^+$, we conclude that it is an isomorphism. On the Batyrev's ring, 
\[h^{(2p+q)(2p+q+1)}=2^{2p}Q_1^+(h^{2p+q}(\xi-h))^{2p+q}=2^{2p}Q_1^+(Q_1^+Q_2^+-h^{2p+q+1})^{2p+q},\]
which shows that the eigenvalues of the operator $h\cdot$ are exactly
\[2^{2p/(2p+q)(2p+q+1)}(Q_1^+)^{1/(2p+q)}(Q_2^+)^{1/(2p+q+1)}(1+\zeta_{2p+q}^{k_1}2^{2p/(2p+q)}(Q_1^+)^{1/(2p+q)})^{-1/(2p+q+1)}\zeta_{2p+q+1}^{k_2}\]
for $0\leq k_1\leq 2p+q-1$, $0\leq k_2\leq 2p+q$, where $\zeta_k=e^{2\pi i/k}$. By applying $\xi h^{2p+q}$ on the eigenvectors of $h\cdot$, we get that the eigenvalues of $\xi\cdot$ are exactly the following $(2p+q)(2p+q+1)$ distinct roots of $\Xi_0^+$ in $\widetilde{F}_0^+$
\begin{equation}\label{eigenvalue_of_xi}
2^{-2p/(2p+q+1)}(Q_2^+)^{1/(2p+q+1)}(1+\zeta_{2p+q}^{k_1}2^{2p/(2p+q)}(Q_1^+)^{1/(2p+q)})^{(2p+q)/(2p+q+1)}\zeta_{2p+q+1}^{k_2}
\end{equation}
for $0\leq k_1\leq 2p+q-1$, $0\leq k_2\leq 2p+q$.

For a general $\tau$ in the formal neighborhood of $0$, let $\Xi^+(\tau,\lambda)\in\widetilde{F}_0[\![\tau]\!][\lambda]$ be the characteristic polynomial of $\xi\star_\tau$ on $A_{\operatorname{orb}}^*(\mathcal{X}^+)_\C\otimes\widetilde{F}_0^+$. For $\tau=0$, we have known that $\Xi_0^+(0,\lambda)=\Xi_0^+(\lambda)$ has distinct roots in $\widetilde{F}_0^+$. By Hensel's lemma, each root in \eqref{eigenvalue_of_xi} can be uniquely lifted to a root of $\Xi_0^+(\tau,\lambda)$ which belongs to $\widetilde{F}_0^+[\![\tau]\!]$.
\end{proof}
Hence the following proposition follows immediately.

\begin{prop}
$A_{\operatorname{orb}}^*(\mathcal{X}^\pm)_\C\otimes\widetilde{F}_0^\pm$ is a semisimple Frobenius manifold.
\end{prop}

Now, as in Section 5.1.1, let $\{\epsilon_i^\pm\}$ be the canonical frame of $A^*_{\operatorname{orb}}(\mathcal{X}^\pm)_\C\otimes\widetilde{F}_0^\pm$ and  $\{\tilde{\epsilon}_i^\pm=\epsilon_i^\pm/\sqrt{(\epsilon_i^\pm,\epsilon_i^\pm)}_{\operatorname{orb}}\}$ be its normalized frame. Let $\widetilde{F}^\pm$ be the splitting field of the polynomial
\[{\Xi}^\pm(\lambda)\coloneqq \Xi_0^\pm(\lambda)\cdot\prod_i\prod_{\sigma\in\operatorname{Gal}(\widetilde{F}_0^\pm/F_0^\pm)}(\lambda^2-\sigma((\epsilon_i^\pm,\epsilon_i^\pm)_{\operatorname{orb}}))\in F_0^\pm[\lambda]\]
over $F_0^\pm$. 

Since $\phi_c(\Xi_0^+)=\Xi_0^-$, we can extend the regularization map $\phi_c:F_0^+\to F_0^-$ to $\phi_c:\widetilde{F}_0^+\to \widetilde{F}_0^-$. By the definition of a quantum correspondence, $\Phi_c$ is an isomorphism between the Frobenius manifolds $A_{\operatorname{orb}}^*(\mathcal{X}^+)_\C\otimes\widetilde{F}_0^+$ and $A_{\operatorname{orb}}^*(\mathcal{X}^-)_\C\otimes\widetilde{F}_0^-$, and thus it sends the primitive central idempotents of the $\mathcal{X}^+$-side to the primitive central idempotents of the $\mathcal{X}^-$-side up to permutation. We may reorder the indices so that $\Phi_c(\epsilon_i^+)=\epsilon_i^-$. Since $\Phi_c$ also preserves the orbifold Poincar\'e pairing, we have 
\begin{equation}\label{compatible_pairing_on_epsilon}
\phi_c((\epsilon_i^+,\epsilon_j^+)_{\operatorname{orb}})=(\epsilon_i^-,\epsilon_j^-)_{\operatorname{orb}}.
\end{equation} 

Moreover, since the relation \eqref{compatible_pairing_on_epsilon} implies $\phi_c(\Xi^+)=\Xi^-$,
 we can further extend the regularization map $\phi_c:\widetilde{F}_0^+\to \widetilde{F}_0^-$ to $\phi_c:\widetilde{F}^+\to \widetilde{F}^-$. Once we have fixed the extended map $\phi_c:\widetilde{F}^+\to\widetilde{F}^-$ and the branch of $\sqrt{(\epsilon_i^+,\epsilon_i^+)}_{\operatorname{orb}}$, there exists an unique choice of the branch $\sqrt{(\epsilon_i^-,\epsilon_i^-)}_{\operatorname{orb}}$ such that $\Phi_c(\tilde{\epsilon}_i^+)=\tilde{\epsilon}_i^-$.

Via the identification $\Phi_c(\epsilon_i^+)=\epsilon_i^-$, we choose the canonical coordinate $u_i$ such that $\partial/\partial u^i=\epsilon_i^\pm$ on $\mathcal{X}^\pm$ and the Euler vector field $E^\pm$ is of the form in \eqref{Euler_vf_in_canonical_coordinate}. 

We choose a homogeneous basis $\{T^+_\mu\}$ of $A_{\operatorname{orb}}^*(\mathcal{X}^+)$ with the coordinate $\tau^\mu$ and fix the homogeneous basis of $A_{\operatorname{orb}}^*(\mathcal{X}^-)_\C$ by $T_\mu^-=\Phi_c(T_\mu^+)$ with the same coordinate symbol $\tau^\mu$. Since $\Phi_c$ preserves the orbifold degree and $\Phi_c(c_1(\mathcal{X}^+))=c_1(\mathcal{X^-})$, the Euler vector field defined in \eqref{Euler_vector_field} is compatible with $\Phi_c$, that is, $\Phi_c(E^+)=E^-$. 

We may regard $u^i$ as a formal power series in $\tau$, and the expression \eqref{Euler_vf_in_canonical_coordinate} implies

\[u^i(\tau)=(E^\pm,\epsilon^\pm_i)_{\operatorname{orb}}=\sum_\mu(1-\deg_{\operatorname{orb}}T_\mu^\pm)\tau^\mu(T_\mu^\pm,\epsilon_i^\pm)_{\operatorname{orb}}+(c_1(\mathcal X),\epsilon_i^\pm)_{\operatorname{orb}}\in\widetilde{F}^\pm[\![\tau]\!].\]
Since we extend $\phi_c$ on the variables $\tau^\mu$ by identity in Definition \ref{quantum correspond}, the regularization map $\phi_c$ is an identity map on the coordinate $u^i$ by the following calculation
\[u^i\frac{\partial}{\partial u^i}=E^-=\Phi(E^+)=\phi(u^i)\frac{\partial}{\partial u^i}.\]

Let $\Psi^\pm=(\Psi_{\mu i}^\pm)$ be the transition matrix of the frame $\{\tilde{\epsilon}_i^\pm\}$ to the frame $\{T_\mu^\pm\}$. Lemma \ref{eigenvalue_in_F_tilde} implies that $\Psi_{\mu i}^\pm\in\widetilde{F}^\pm[\![\tau]\!]$. Again, since $\Phi_c$ preserves the orbifold Poincar\'e pairing, we have $\phi_c(\Psi^+)=\Psi^-$. Under this identification, the coordinate $\mathbf{t}^i$ of Givental's space of points is compatible when we do the comparsion of Givental--Teleman formula for $\mathcal{X}^+$ and $\mathcal{X}^-$.

\begin{lm}\label{compatible_R_matrix}
The $R$-matrix $R_{\mathcal{X}^\pm}(\tau,z)$ has entries in $\widetilde{F}^\pm[\![u^i-u^i(0)]\!][\![z]\!]$ and 
\[\Phi_c\circ R_{\mathcal{X}^+}(\tau,z)\circ\Phi_c^{-1}=R_{\mathcal{X}^-}(\Phi_c(\tau),z)\]
as $\operatorname{End}$-valued functions in $\tau=\sum\tau^\mu T_\mu^+$.
\end{lm}

\begin{proof}
Recall that the $R$-matrix is determined by the initial condition $R_0=\operatorname{id}$ and the following recursion relations:
\[(R_n)_{ij}(du^i-du^j)=[(\Psi d\Psi^{-1}+d)R_{n-1}]_{ij}\]
for each $i\neq j$, and
\[(dR_n)_{ii}=\sum_{k\neq i}(\Psi d\Psi^{-1})_{ki}(R_n)_{ik}\]
for all $i$ (See \cite[Theorem 1]{lee2004frobenius} for more details). 

Now, we work on $\mathcal{X}^\pm$. According to Lemma \ref{eigenvalue_in_F_tilde}, the coordinate transform matrix is valued in $\widetilde{F}^\pm[\![u^i-u^i(0)]\!]$. If we expand $R_n$ as a power series in $u^i-u^i(0)$, then each coefficient still lies in $\widetilde{F}^\pm$ by the recursive formula. Moreover, both $\{\phi_c(R_n^+)\}_{n\geq 0}$ and $\{R_n^-\}_{n\geq 0}$ satisfy the recession formula in $\mathcal{X}^-$ since $\phi_c(\Psi^+)=\Psi^-$. Since $\Phi_c$ identifies the Euler vector fields of both sides, the homogeneous condition of $R$-matrix is compatible. By the uniqueness of the homogeneous $R$-matrix in \cite[Theorem 1 (iv)]{lee2004frobenius}, we conclude that $\phi_c(R_n^+)=R_n^-$ for all $n\geq 0$. 
\end{proof}

\begin{thm}\label{higher-genus-invariance}
The quantum correspondence $\Phi_c=\phi_c\otimes[\mathscr{F}_c]$ preserves the ancestor potentials in the sense that
\[\phi_c(\mathcal{A}_{\mathcal{X}^+}(t,\tau)^{48})=\mathcal{A}_{\mathcal{X}^-}(\Phi_c(t),\Phi_c({\tau}))^{48},\]
where $\Phi_c(t)$ acts trivially on $t_k^\mu$ and $\overline{\psi}^k$. Moreover, every Gromov--Witten invariant 
\[\gen{T_1,\ldots,T_n}_{g,n}^{\mathcal{X}^\pm}(\tau),\quad \forall g\geq 0\]
is valued in $F_0^\pm[\![\tau]\!]$, and $\Phi_c$ identifies Gromov--Witten potentials for all genus $g\geq 0$ in the sense that
\[\phi_c(\overline{F}_g^{\mathcal{X}^+}(0,\tau))=\overline{F}_g^{\mathcal{X}^-}(0,\Phi_c(\tau)).\] 
\end{thm}

\begin{proof}
Since we have identified the canonical coordinates $\tilde{\epsilon}^\pm$, the Darboux coordinate systems are compatible. Moreover, Givental's quantizations are also compatible. By Theorem \ref{Givental_formula}, we have that
\[\mathscr A_{\mathcal X}(t,\tau)^{48}=\det(\varepsilon_i,\varepsilon_j)_{\operatorname{orb}}\left(\hat\Psi^{-1}(\tau)\widehat{\mathbf R}_{\mathcal X}(\tau,z)e^{\widehat{u/z}}\prod_{i=1}^{\dim A^*_{\operatorname{orb}}(\mathcal X)}\mathscr A_{pt}(\mathbf t^i)\right)^{48}\]
for $\mathcal{X}=\mathcal{X}^\pm$. By the relation \eqref{compatible_pairing_on_epsilon} and Lemma \ref{compatible_R_matrix}, we get that
\begin{equation}\label{compatible_A48}
\phi_c(\mathcal{A}_{\mathcal{X}^+}(t,\tau)^{48})=\mathcal{A}_{\mathcal{X}^-}(\Phi_c(t),\Phi_c({\tau}))^{48}.
\end{equation}
Since $\mathcal{A}_{\mathcal{X}^\pm}(t,\tau)$ can be regarded as a formal power series in $t$, we may evaluate it at $t=0$. Since $\overline{F}_0^{\mathcal{X}^+}$ has coefficients in $\widetilde{F}^+$, through the regularization map $\phi_c$, the equation \eqref{compatible_A48} turns out to be 
\begin{align*}
&\exp\left(\frac{48}{\hbar}\left(\phi_c(\overline{F}_0^{\mathcal{X}^+}(0,\tau))-\overline{F}_0^{\mathcal{X}^-}(0,\Phi_c(\tau))\right)\right) \\
=&\ \phi_c\left(\exp\left(-48\sum_{g\geq 1}\hbar^{g-1}\overline{F}_g^{\mathcal{X}^+}(0,\tau)\right)\right)\exp\left(48\sum_{g\geq 1}\hbar^{g-1}\overline{F}_g^{\mathcal{X}^-}(0,\Phi_c(\tau))\right).
\end{align*}
It forces that $\phi_c(\overline{F}_0^{\mathcal{X}^+}(0,\tau))=\overline{F}_0^{\mathcal{X}^-}(0,\Phi_c(\tau))$ and 
\begin{equation}\label{compatible_g>0_potential}
\phi_c\left(\exp\left(48\sum_{g\geq 1}\hbar^{g-1}\overline{F}_g^{\mathcal{X}^+}(0,\tau)\right)\right)=\exp\left(48\sum_{g\geq 1}\hbar^{g-1}\overline{F}_g^{\mathcal{X}^-}(0,\Phi_c(\tau))\right).
\end{equation}
We expand the equation \eqref{compatible_g>0_potential} as formal power series in $\hbar$ and compare the coefficients to conclude that 
\[\overline{F}_g^{\mathcal{X}^\pm}(0,\tau)\in\widetilde{F}^\pm[\![\tau]\!],\]

\begin{equation}\label{compatible_GW_g>0_potential}
\phi_c(\overline{F}_g^{\mathcal{X}^+}(0,\tau))=\overline{F}_g^{\mathcal{X}^-}(0,\Phi_c(\tau))
\end{equation}
for all $g\geq 1$. Moreover, since the formal series in the right hand side of \eqref{compatible_GW_g>0_potential} is independent of the lifting $\phi_c\colon\widetilde{F}^+\to\widetilde{F}^-$ of $\phi_c\colon F_0^+\to F_0^-$, the Galois theorem forces $\overline{F}_g^{\mathcal{X}^\pm}(0,\tau)\in F_0^\pm[\![\tau]\!]$.
\end{proof}

\newpage

\let\OLDthebibliography\thebibliography
\renewcommand\thebibliography[1]{
  \OLDthebibliography{#1}
  \settowidth{\labelwidth}{[#1]} 
  \setlength{\labelsep}{0.5em}
  \setlength{\leftmargin}{0pt}
  \setlength{\leftskip}{0pt}     
  \setlength{\itemindent}{\labelwidth}
  \addtolength{\itemindent}{\labelsep}
  \setlength{\itemsep}{0pt plus 0.3ex}
  \setlength{\parskip}{0pt}
}
\bibliographystyle{alpha}

\end{document}